\documentclass[12pt]{amsart}
\usepackage{amsfonts,amssymb,amsmath,amsthm,amstext}

\usepackage[colorlinks, linkcolor=blue, citecolor=blue, urlcolor=blue,pagebackref,
hypertexnames=false]{hyperref}
\usepackage{geometry}
\usepackage{mathrsfs}
\usepackage{txfonts}
\allowdisplaybreaks
\numberwithin{equation}{section}
\newtheorem{theorem}{Theorem}[section]
\newtheorem{proposition}[theorem]{Proposition}
\newtheorem{lemma}[theorem]{Lemma}
\newtheorem{corollary}[theorem]{Corollary}
\theoremstyle{remark}
\newtheorem{remark}[theorem]{Remark}

\newcommand{\bbR}{\mathbb R}
\newcommand{\bbZ}{\mathbb Z}
\newcommand{\calF}{\mathcal F_\kappa}
\newcommand{\calK}{\mathcal K}
\newcommand{\calS}{\mathcal S}
\newcommand{\calW}{\mathcal W}
\newcommand{\abs}[1]{\left|#1\right|}
\newcommand{\norm}[1]{\left\|#1\right\|}
\newcommand{\ip}[2]{\left\langle #1,#2\right\rangle}
\newcommand{\supp}{\operatorname{supp}}
\newcommand{\BMO}{\mathrm{BMO}}
\newcommand{\ch}{\mathrm{ch}}

\title[Non-radial Dunkl multipliers]
{Non-radial Dunkl multipliers at the $L^2$-Sobolev threshold}

\author{Der-Chen Chang}
\address{Der-Chen Chang, Department of Mathematics and Statistics, Georgetown University, Washington, DC 20057, USA; \,
Graduate Institute of Business Administration, College of Management, Fu Jen Catholic University, New Taipei City 242,
Taiwan, ROC}
\email{chang@georgetown.edu}

\author{Ji Li}
\address{Ji Li, School of Mathematical and Physical Sciences, Macquarie University, NSW 2109, Australia}
\email{ji.li@mq.edu.au}

\author{Chaojie Wen}
\address{Chaojie Wen, Department of Mathematics, Sun Yat-sen University, Guangzhou, 510275, P.R.~China}
\email{wenchj@mail2.sysu.edu.cn}

\author{Liangchuan Wu}
\address{Liangchuan Wu, School of Mathematical Sciences, Anhui University, Hefei, 230601, P.R.~China}
\email{wuliangchuan@ahu.edu.cn}


\hypersetup{
 pdftitle={Non-radial Dunkl multipliers at the L2-Sobolev threshold},
 pdfauthor={Der-Chen Chang, Ji Li, Chaojie Wen, and Liangchuan Wu},
 pdfsubject={Dunkl multiplier theorems without spectral wall separation},
 pdfkeywords={Dunkl transform, Hormander multiplier theorem, non-radial multipliers,
 finite reflection groups, spectral walls}
}

\begin{document}

\begin{abstract}
We prove a H\"ormander multiplier theorem for the Dunkl transform associated with an arbitrary
finite reflection group.  Uniform $H^\sigma(\bbR^N)$ bounds for the normalised dyadic pieces of a
measurable symbol, with $\sigma>\mathbf N/2$ and $\mathbf N$ the homogeneous dimension, imply
$L^p(d\omega)$ boundedness for $1<p<\infty$ and weak type $(1,1)$.  No radiality or reflection-group
invariance is assumed, and the dyadic pieces may meet the reflecting hyperplanes.  The main new estimate
controls the Dunkl kernel $E_\kappa(i\xi,x)$ and its first Euclidean derivatives in $x$, uniformly near
arbitrary intersections of reflecting hyperplanes.  Here $\xi$ lies in a fixed compact annulus.
Writing $w_\kappa$ for the Dunkl weight and
$\Omega_\kappa$ for its inhomogeneous counterpart, it gives the bound
$C\Omega_\kappa(x)^{-1/2}w_\kappa(\xi)^{-1/2}$.  The proof decomposes each root contribution into
Hermitian two-dimensional blocks, applies oscillatory changes of variables near the rootwise transition
points, and uses a commutator identity to cancel the non-integrable phase derivative.  Chamber lifting also yields endpoint
bounds for the pullbacks of the componentwise Coifman--Weiss atomic $H^1$ space and the corresponding
BMO space on a fixed chamber.
\end{abstract}

\maketitle

\section{Introduction}

H\"ormander's Fourier multiplier theorem gives $L^p$ boundedness,
$1<p<\infty$, under a uniform $L^2$-Sobolev condition of order greater
than one half of the Euclidean dimension \cite{Hormander1960}.
For the Dunkl transform, the corresponding kernel calculation involves
ordinary derivatives and reflection differences.
We combine estimates for these terms with a finite matrix
representation on a chamber to prove multiplier bounds for general
scalar symbols and functions.

Let $R\subset\bbR^N$ be a reduced finite root system, let $G$ be the
finite reflection group generated by the reflections $\sigma_\alpha$,
and let $\kappa:R\to[0,\infty)$ be $G$-invariant.
Fix a positive subsystem $R_+$.
Multiplying all roots in each $G$-orbit by the same positive constant
leaves the reflections and the Dunkl operators unchanged.
We normalize the roots so that $\norm{\alpha}^2=2$ for $\alpha\in R$.
This changes the weights below by fixed positive constants.
Set
$$
 w_\kappa(x)
 =
 \prod_{\alpha\in R_+}\abs{\ip{\alpha}{x}}^{2\kappa(\alpha)},
 \qquad
 d\omega(x)=w_\kappa(x)\,dx,
$$
and let
$\mathbf N=N+2\sum_{\alpha\in R_+}\kappa(\alpha)$
be the homogeneous dimension.

Let $E_\kappa$ denote the Dunkl kernel.
With the conventions of \cite{Dunkl1989,deJeu1993,DunklXu2014}, define
$$
 \calF f(\xi)
 =
 (c_\kappa^{\mathrm M})^{-1}
 \int_{\bbR^N}f(x)E_\kappa(-i\xi,x)\,d\omega(x),
 \qquad
 c_\kappa^{\mathrm M}
 =
 \int_{\bbR^N}e^{-\norm{x}^2/2}\,d\omega(x).
$$
The superscript $\mathrm M$ denotes the Macdonald--Mehta normalization.
The transform is unitary on $L^2(d\omega)$, and its inverse is given
by the same integral with $E_\kappa(i\xi,x)$.
The orbit distance
$$
 d(x,y)=\min_{g\in G}\norm{x-gy}
$$
is $G$-invariant in either variable.
The coordinate Dunkl operators are
$$
 T_jf(x)
 =
 \partial_jf(x)
 +
 \sum_{\alpha\in R_+}
 \kappa(\alpha)\alpha_j
 \frac{f(x)-f(\sigma_\alpha x)}{\ip{\alpha}{x}}.
$$
They satisfy
$$
 T_j^\xi E_\kappa(z,\xi)=z_jE_\kappa(z,\xi),
 \qquad
 x_j\calF^{-1}q(x)=\calF^{-1}(iT_jq)(x),
$$
where the second identity holds for $q\in\calS(\bbR^N)$.
For a bounded measurable symbol $m$, define the Dunkl multiplier
$T_m$ on $L^2(d\omega)$ by
$$
 \calF(T_mf)=m\calF f.
$$

Let
$\calW=\bigcup_{\alpha\in R}\alpha^\perp$.
We call these reflecting hyperplanes spectral walls when they occur
in frequency space.
A function $\varphi\in C_c^\infty((0,\infty))$ is an admissible cutoff
if
$\sum_{j\in\bbZ}\varphi(2^{-j}r)=1$ for $r>0$.
For such a cutoff and $\sigma\ge0$, set
$$
 \calS_\sigma^\varphi(m)
 =
 \sup_{j\in\bbZ}
 \norm{\varphi(\norm{\cdot})m(2^j\cdot)}_{H^\sigma(\bbR^N)}.
$$
Fix a non-negative admissible cutoff
$\psi\in C_c^\infty((1/2,2))$ and write
$\calS_\sigma(m)=\calS_\sigma^\psi(m)$.
Lemma~\ref{lem:dyadic} shows that the resulting Sobolev condition
is independent of the choice of admissible cutoff.
For $\sigma>N/2$, it also gives
$\norm{m}_{L^\infty}\le C\calS_\sigma(m)$.
The dyadic pieces cover all nonzero frequencies, and $T_m$ is
determined by the almost everywhere values of $m$.

Dziuba\'nski and Hejna proved the dyadic $H^\sigma$ multiplier theorem
for arbitrary root systems when $\sigma>\mathbf N$.
They obtained the threshold $\sigma>\mathbf N/2$ for radial
multipliers, and for general multipliers when the Dunkl translations
are uniformly bounded on $L^1$
\cite[Theorems~1.2 and~8.1, Remarks~8.3--8.4]{DH2019}.
Bui obtained $L^p$ boundedness, $1<p<\infty$, at the same threshold
for arbitrary symbols under the stronger dyadic $W_\infty^s$
condition \cite[Theorem~1.3]{Bui2025}.
Related results for non-radial translated kernels, radial multipliers,
product systems, and intrinsic Dunkl Hardy spaces appear in
\cite{DH2023,MT2025,Wrobel2015,DHL2026}.
See also \cite{TX2005} for the underlying convolution theory.
For weighted and variable-exponent multiplier results, see
Li and Zhao \cite{LiZhao2026}.

In \cite[Theorems~1.1 and~1.3]{Chamber2026}, we established a
finite matrix Calder\'on--Zygmund reduction on a chamber for
arbitrary finite reflection groups.
For the group generated by reflections across the
coordinate hyperplanes, we verified the scalar Sobolev kernel
estimates for Walsh components supported in fixed frequency cones
separated from the coordinate hyperplanes.

The aim of our paper is to provide a full multiplier theorem in the Dunkl setting.
We prove these scalar estimates for every finite
reflection group at $\sigma>\mathbf N/2$.
The supports of the dyadic symbols may meet the reflecting hyperplanes.
Applying the estimates to finite sums of localized multipliers
also gives weak type $(1,1)$ and the chamber Hardy and BMO bounds.
\begin{theorem}\label{thm:main}
Let $m$ be a measurable function on $\bbR^N$ satisfying
$\calS_\sigma(m)<\infty$ for some $\sigma>\mathbf N/2$.
Then $m$ is essentially bounded, and its Dunkl multiplier $T_m$
is bounded on $L^p(d\omega)$ for $1<p<\infty$, is of weak type
$(1,1)$, and satisfies
$$
 T_m:H^1_{\ch}\longrightarrow L^1(d\omega),
 \qquad
 T_m:L^\infty(d\omega)\longrightarrow\BMO_{\ch}.
$$
Quantitatively,
\begin{align*}
 \norm{T_m}_{L^p\to L^p}
 &\le C_p\calS_\sigma(m),
 \qquad 1<p<\infty.
\end{align*}
Moreover,
$$
 \norm{T_m}_{L^1\to L^{1,\infty}}
 +\norm{T_m}_{H^1_{\ch}\to L^1}
 +\norm{T_m}_{L^\infty\to\BMO_{\ch}}
 \le C\calS_\sigma(m).
$$
The constants depend only on $R$, $\kappa$, $\sigma$, $\psi$,
and $p$ where applicable.
The symbol and the functions may be nonradial and non-$G$-invariant.
The dyadic pieces of $m$ may meet the reflecting hyperplanes.
\end{theorem}

The proof begins with the rescaled dyadic symbols
$$
 b_j(\xi)=\psi(\norm{\xi})m(2^j\xi),
 \qquad
 \sup_{j\in\bbZ}\norm{b_j}_{H^\sigma}=\calS_\sigma(m).
$$
We first prove the kernel estimates for smooth amplitudes $a$
supported in a fixed annulus containing the supports of all $b_j$.
The estimates are then extended to $b_j$ by interpolation and
approximation.
For such an amplitude, the Dunkl product rule gives ordinary
derivatives of $a$ and reflection terms involving
$$
 \delta_\alpha a(\xi)
 =
 \frac{a(\xi)-a(\sigma_\alpha\xi)}{\ip{\alpha}{\xi}},
 \qquad \ip{\alpha}{\xi}\ne0.
$$
Lemma~\ref{lem:delta} shows that $\delta_\alpha a$ extends smoothly across $\alpha^\perp$ and satisfies the global estimate
$$
 \norm{\delta_\alpha a}_{H^s}
 \le C_{\alpha,s}\norm{a}_{H^{s+1}},
 \qquad s\ge0.
$$
Thus each reflection term is controlled by one Sobolev derivative
of $a$, as is an ordinary first derivative.

We combine these estimates with uniform bounds for the Dunkl
kernel and its first Euclidean spatial derivatives.
R\"osler and de Jeu \cite{RdJ2002} and Langen \cite{Langen2026}
obtained spatial estimates for fixed regular spectral parameters.
The dihedral estimates of Anker and Trojan \cite{AT2025} give
related bounds with explicit factors involving the root pairings.
For estimates on conical subsets, see also \cite{AG2021}.

One fundamental observation is the following Dunkl kernel estimate.
\begin{proposition}\label{thm:joint}
For every compact annulus $\mathcal A\subset\bbR^N\setminus\{0\}$
and every multi-index $\nu$ with $\abs{\nu}\le1$, there exists
$C_{\mathcal A,\nu}<\infty$ such that
\begin{equation}\label{eq:joint}
 \sup_{\xi\in \mathcal A\setminus\calW,\,y\in\bbR^N}
 w_\kappa(\xi)^{1/2}\Omega_\kappa(y)^{1/2}
 \abs{\partial_y^\nu E_\kappa(i\xi,y)}
 \le C_{\mathcal A,\nu},
\end{equation}
where $\Omega_\kappa(y)
 =
 \prod_{\alpha\in R_+}
 \bigl(1+\abs{\ip{\alpha}{y}}\bigr)^{2\kappa(\alpha)}$.
The estimate is uniform as the regular spectral parameter approaches
any intersection of reflecting hyperplanes.
\end{proposition}
Since $d\omega(\xi)=w_\kappa(\xi)d\xi$, estimate \eqref{eq:joint}
gives
\begin{equation}\label{eq:packet-joint}
 \sup_{y\in\bbR^N}
 \Omega_\kappa(y)^{1/2}
 \norm{u(\xi)\partial_y^\nu E_\kappa(i\xi,y)}_{L^2(d\omega(\xi))}
 \le
 C_{\mathcal A,\nu}\norm{u}_{L^2(\mathcal A,d\xi)}
\end{equation}
for every $u\in L^2(\mathcal A,d\xi)$, extended by zero outside
$\mathcal A$.
The Euclidean $L^2$ norm on the right allows us to apply the Sobolev
estimates for the localized symbols.

Fix an open chamber $\mathcal C$ and define
$$
 Uf(x)=\bigl(f(gx)\bigr)_{g\in G},
 \qquad x\in\mathcal C.
$$
The map $U$ identifies $L^p(\bbR^N,d\omega)$ isometrically with
$L^p(\mathcal C,d\omega;\ell^p(G))$ for $1\le p\le\infty$.
Its components record the values of $f$ on all chambers.
We write $T_m^G=UT_mU^{-1}$.
This is a finite matrix operator on $\mathcal C$.
We estimate every scalar kernel in the matrix representations of
finite dyadic sums and then pass to $T_m^G$ by $L^2$ convergence.

For $z\in\mathcal C$ and $r>0$, the chamber balls are
$$
 B_{\mathcal C}(z,r)=B(z,r)\cap\mathcal C.
$$
With the Euclidean distance and measure $d\omega$, the chamber
$\mathcal C$ is a doubling space.
We use its componentwise Coifman--Weiss atomic $H^1$ and BMO spaces,
and denote their pullbacks under $U$ by $H^1_{\ch}$ and $\BMO_{\ch}$.
The scalar atomic space consists of $L^1(\mathcal C,d\omega)$
functions admitting atomic representations.
The scalar BMO space consists of functions integrable on every
chamber ball with bounded mean oscillation, modulo constants.
The complete definitions are given in Section~\ref{sec:4}.

For non-$G$-invariant symbols, we retain the reflection
differences of the localized amplitudes and control them
by Lemma~\ref{lem:delta}. In the moment identity \eqref{eq:moment}, the reflection term
involving $\delta_\alpha a$ changes the centre of the kernel
from $gy$ to $\sigma_\alpha gy$.
We estimate together all the kernels with centres $gy$, $g\in G$.
Since $\sigma_\alpha g\in G$, every new centre belongs to the
same finite orbit.
Using the Sobolev estimate above, we control the $L^2(d\xi)$ norm of each amplitude obtained after $L$ applications of the identity by $C_L\norm{a}_{H^L}$.
Applying \eqref{eq:packet-joint} to these amplitudes and using
Plancherel's identity gives the weighted $L^2$ kernel estimates.
Moreover, orthogonality of $g$ gives
$$
 \norm{b_j\circ g}_{H^\sigma}
 =
 \norm{b_j}_{H^\sigma}
 \le\calS_\sigma(m),
 \qquad j\in\bbZ,\quad g\in G.
$$
Thus the same estimates apply to all reflected dyadic symbols.
Corollaries~\ref{cor:angular} and~\ref{cor:mikhlin} give angular
Sobolev and Mikhlin consequences of Theorem~\ref{thm:main}.

This paper is organised as follows.
Section~\ref{sec:2} proves cutoff independence and the global
Sobolev estimates for the reflection terms.
Section~\ref{sec:5} treats the coordinate reflection group,
establishes the required matrix differential equation estimate,
and proves Proposition~\ref{thm:joint} for an arbitrary root system.
Section~\ref{sec:4} combines the joint kernel estimate with the chamber lifting argument to prove Theorem~\ref{thm:main}.

\section{preliminaries}\label{sec:2}

We first show that the dyadic Sobolev condition does not depend
on the choice of cutoff. We then prove Sobolev estimates for
divided difference operators and their compositions with partial
derivatives. Finally, we use the heat kernel to obtain an
$L^2$ estimate for the Dunkl kernel.

\subsection{The dyadic Sobolev condition}

We first verify that the dyadic condition used in Theorem~\ref{thm:main} does not depend on the cutoff.

\begin{lemma}\label{lem:dyadic}
Let $\sigma\ge0$, and let $\psi$ and $\varphi$ be admissible cutoffs.  For every measurable $m$ on
$\bbR^N$,
$$
 C^{-1}\calS_\sigma^\psi(m)
 \le
 \calS_\sigma^\varphi(m)
 \le
 C\calS_\sigma^\psi(m),
$$
where $C=C_{\sigma,\psi,\varphi}\ge1$ is independent of $m$.  If $\sigma>N/2$ and
$\calS_\sigma^\psi(m)<\infty$, then $m\in L^\infty(\bbR^N)$ and
\begin{equation}\label{eq:automatic-Linfty}
 \norm{m}_{L^\infty(\bbR^N)}
 \le
 C_{\sigma,\psi}\calS_\sigma^\psi(m).
\end{equation}
\end{lemma}

\begin{proof}
To prove the upper bound, assume $\calS_\sigma^\psi(m)<\infty$, otherwise there is nothing to prove.
For $\theta\in\{\psi,\varphi\}$ and $j\in\bbZ$, set
$$
 b_j^\theta(\eta)=\theta(\norm{\eta})m(2^j\eta),
$$
and define the finite set
$$
 I_{\varphi,\psi}
 =
 \bigl\{\ell\in\bbZ:\supp\varphi\cap2^\ell\supp\psi\ne\varnothing\bigr\}.
$$
The partition identity gives
\begin{align*}
 b_j^\varphi(\eta)
 &=
 \varphi(\norm{\eta})m(2^j\eta)
 \sum_{\ell\in\bbZ}\psi(2^{-\ell}\norm{\eta})
 =
 \sum_{\ell\in I_{\varphi,\psi}}
 \varphi(\norm{\eta})b_{j+\ell}^\psi(2^{-\ell}\eta),
\end{align*}
The Sobolev multiplier bound and dilation give
\begin{align*}
 \norm{b_j^\varphi}_{H^\sigma}
 &\le
 C_{\sigma,\varphi}
 \sum_{\ell\in I_{\varphi,\psi}}
 \norm{b_{j+\ell}^\psi(2^{-\ell}\cdot)}_{H^\sigma} \le
 C_{\sigma,\varphi}
 \sum_{\ell\in I_{\varphi,\psi}}
 2^{\ell N/2}\max\{1,2^{-\ell\sigma}\}
 \norm{b_{j+\ell}^\psi}_{H^\sigma}
 \le
 C_{\sigma,\psi,\varphi}\calS_\sigma^\psi(m).
\end{align*}
Taking the supremum over $j$ and then interchanging $\psi$ and $\varphi$ proves the first assertion.

Suppose that $\sigma>N/2$ and $\calS_\sigma^\psi(m)<\infty$.  Choose the continuous Sobolev
representatives of $b_j^\psi$, still denoted by $b_j^\psi$.  Since $\bbZ$ is countable, outside a
single null set in $\bbR^N\setminus\{0\}$,
$$
 b_j^\psi(2^{-j}\xi)
 =
 \psi(2^{-j}\norm{\xi})m(\xi),
 \qquad j\in\bbZ.
$$
Note that
$$
 \sup_{r>0}\#\{j\in\bbZ:\psi(2^{-j}r)\ne0\}<\infty.
$$
Then, Sobolev embedding and finite overlap give, for almost every $\xi\ne0$,
\begin{align*}
 \abs{m(\xi)}
 &=
 \left|\sum_{j\in\bbZ}\psi(2^{-j}\norm{\xi})m(\xi)\right|
 \le
 \sum_{j:\,\psi(2^{-j}\norm{\xi})\ne0}
 \abs{b_j^\psi(2^{-j}\xi)}
 \le
 C_\psi\sup_{j\in\bbZ}\norm{b_j^\psi}_\infty
 \le
 C_{\sigma,\psi}\sup_{j\in\bbZ}\norm{b_j^\psi}_{H^\sigma}
 =
 C_{\sigma,\psi}\calS_\sigma^\psi(m).
\end{align*}
Taking the essential supremum proves \eqref{eq:automatic-Linfty}, and the value at the origin is immaterial.
\end{proof}

\subsection{The divided difference operators}

This subsection proves the algebraic kernel estimates without wall separation.  A reflection divided difference operator is not a zero-order operator when a a smooth amplitude supported in a fixed annulus may concentrate near a wall.  The operator maps $H^{s+1}$ to $H^s$ for $s\ge0$, just as an ordinary first derivative does, so each spatial moment costs one derivative.

For $\alpha\in R$, let
$$
 n_\alpha=\frac{\alpha}{\norm{\alpha}},
 \qquad
 \sigma_\alpha\xi
 =
 \xi-\frac{2\ip{\alpha}{\xi}}{\norm{\alpha}^2}\alpha.
$$
For a smooth function $a$, define the divided difference operator associated with $\sigma_\alpha$ away from $\alpha^\perp$ by, 
\begin{equation}\label{eq:delta}
 \delta_\alpha a(\xi)
 =
 \frac{a(\xi)-a(\sigma_\alpha\xi)}{\ip{\alpha}{\xi}}.
\end{equation}

\begin{lemma}\label{lem:delta}
For $a\in C_c^\infty(\bbR^N)$, the quotient in \eqref{eq:delta} has a unique smooth extension across
$\alpha^\perp$, denoted by $\delta_\alpha a$.  For every $s\ge0$, this operator extends uniquely to a
bounded operator
$
 \delta_\alpha:H^{s+1}(\bbR^N)\longrightarrow H^s(\bbR^N).
$
For $a\in H^{s+1}(\bbR^N)$,
\begin{equation}\label{eq:delta-bound}
 \norm{\delta_\alpha a}_{H^s(\bbR^N)}
 \le
 \frac{C_{s,N}}{\norm{\alpha}}\norm{a}_{H^{s+1}(\bbR^N)}.
\end{equation}
For $a\in C_c^\infty(\bbR^N)$,
\begin{equation}\label{eq:delta-support}
 \delta_\alpha a\in C_c^\infty(\bbR^N),
 \qquad
 \supp(\delta_\alpha a)
 \subset
 \supp a\cup\sigma_\alpha(\supp a).
\end{equation}
\end{lemma}

\begin{proof}
Write $\xi=z+tn_\alpha$, where $z\in n_\alpha^\perp$ and $\partial_t=\partial_{n_\alpha}$.  Since
$$
 \ip{\alpha}{z+tn_\alpha}=\norm{\alpha}t,
 \qquad
 \sigma_\alpha(z+tn_\alpha)=z-tn_\alpha,
$$
the fundamental theorem of calculus gives, for $t\ne0$,
\begin{equation*}
 \delta_\alpha a(z+tn_\alpha)
 =
 \frac{a(z+tn_\alpha)-a(z-tn_\alpha)}{\norm{\alpha}t}
 =
 \frac1{\norm{\alpha}}
 \int_{-1}^{1}
 \partial_{n_\alpha}a(z+rtn_\alpha)\,dr.
\end{equation*}
The last integral defines the unique smooth extension across $t=0$, with
$$
 \delta_\alpha a(z)
 =
 \frac2{\norm{\alpha}}\partial_{n_\alpha}a(z).
$$
For every tangential multi-index $\mu$ and integer $q\ge0$, differentiation gives
\begin{align*}
 \partial_z^\mu\partial_t^q\delta_\alpha a(z+tn_\alpha)
 =
 \frac1{\norm{\alpha}}
 \int_{-1}^{1}
 r^q\partial_z^\mu\partial_t^{q+1}a(z+rtn_\alpha)\,dr,
\end{align*}
Minkowski's inequality and the change of variables $t\mapsto rt$ give
\begin{align*}
 \norm{\partial_z^\mu\partial_t^q\delta_\alpha a}_2
 &\le
 \frac1{\norm{\alpha}}
 \int_{-1}^{1}\abs{r}^q
 \norm{(\partial_z^\mu\partial_t^{q+1}a)(z+rtn_\alpha)}_{L^2_{z,t}}\,dr\\
 &=
 \frac1{\norm{\alpha}}
 \int_{-1}^{1}\abs{r}^{q-1/2}\,dr\,
 \norm{\partial_z^\mu\partial_t^{q+1}a}_2\\
 &=
 \frac4{(2q+1)\norm{\alpha}}
 \norm{\partial_z^\mu\partial_t^{q+1}a}_2.
\end{align*}
Consequently, for every integer $m\ge0$,
\begin{align*}
 \norm{\delta_\alpha a}_{H^m}
 &\le
 C_{m,N}\sum_{\abs{\mu}+q\le m}
 \norm{\partial_z^\mu\partial_t^q\delta_\alpha a}_2
 \le
 \frac{C_{m,N}}{\norm{\alpha}}
 \sum_{\abs{\mu}+q\le m}
 \norm{\partial_z^\mu\partial_t^{q+1}a}_2
 \le
 \frac{C_{m,N}}{\norm{\alpha}}\norm{a}_{H^{m+1}}.
\end{align*}
Let $\delta_\alpha^{(m)}:H^{m+1}\to H^m$ be the extension obtained by density.  If
$a\in H^{m+2}$ and $a_k\to a$ in $H^{m+2}$ with $a_k\in C_c^\infty(\bbR^N)$, then
\begin{align*}
 \delta_\alpha^{(m)}a
 &=
 \lim_{k\to\infty}\delta_\alpha a_k
 =
 \delta_\alpha^{(m+1)}a
 \qquad\text{in }H^m.
\end{align*}
So the integer-order extensions are compatible.  If $s=m+\theta$, where $m\ge0$ is an integer and
$0<\theta<1$, complex interpolation yields
\begin{align*}
 \delta_\alpha:H^{s+1}
 &=[H^{m+1},H^{m+2}]_\theta
 \longrightarrow
 [H^m,H^{m+1}]_\theta=H^s.
\end{align*}
Moreover,
\begin{align*}
 \norm{\delta_\alpha a}_{H^s}
 &\le
 \left(\frac{C_{m,N}}{\norm{\alpha}}\right)^{1-\theta}
 \left(\frac{C_{m+1,N}}{\norm{\alpha}}\right)^\theta
 \norm{a}_{H^{s+1}}
 \le
 \frac{C_{s,N}}{\norm{\alpha}}\norm{a}_{H^{s+1}}.
\end{align*}
Density also gives uniqueness.

Finally, with $F_a=\supp a\cup\sigma_\alpha(\supp a)$ and $\sigma_\alpha^2=I$,
\begin{align*}
 \xi\in F_a^c\setminus\alpha^\perp
 &\Longrightarrow
 a(\xi)=a(\sigma_\alpha\xi)=0
 \Longrightarrow
 \delta_\alpha a(\xi)=0.
\end{align*}
Since $F_a^c\setminus\alpha^\perp$ is dense in $F_a^c$, the smooth extension vanishes on $F_a^c$.
Hence $\delta_\alpha a\in C_c^\infty$ and $\supp(\delta_\alpha a)\subset F_a$, proving
\eqref{eq:delta-support}.
\end{proof}

\begin{remark}
The one-derivative loss in \eqref{eq:delta-bound} is sharp when the fixed annulus meets
$\alpha^\perp$: for every $r\ge0$ and $0\le\theta<1$, no estimate
$$
 \norm{\delta_\alpha a}_{H^r}
 \le
 C\norm{a}_{H^{r+\theta}}
$$
can hold uniformly for functions $a$ supported in that annulus.  Let $\mathcal A$ be a fixed compact annulus such that
$\mathcal A^\circ\cap n_\alpha^\perp\ne\varnothing$. For a radial annulus this holds whenever $N\ge2$.  Choose
non-zero functions
$$
 \eta\in C_c^\infty(n_\alpha^\perp),
 \qquad
 \supp\eta\subset \mathcal A^\circ\cap n_\alpha^\perp,
 \qquad
 \zeta\in C_c^\infty((-1,1)),
$$
with $\zeta$ odd.  For some $0<\varepsilon_0\le1$ and every $0<\varepsilon\le\varepsilon_0$, set
$$
 a_\varepsilon(z+tn_\alpha)
 =
 \eta(z)\varepsilon^{-1/2}\zeta(t/\varepsilon),
 \qquad
 \supp a_\varepsilon\subset \mathcal A.
$$
Let $q(u)=\zeta(u)/u$ for $u\ne0$ and $q(0)=\zeta'(0)$.  Then
$q\in C_c^\infty(\bbR)\setminus\{0\}$ and
\begin{align*}
 \delta_\alpha a_\varepsilon(z+tn_\alpha)
 &=
 \frac{a_\varepsilon(z+tn_\alpha)-a_\varepsilon(z-tn_\alpha)}{\norm{\alpha}t}
 =
 \frac{2}{\norm{\alpha}}
 \eta(z)\varepsilon^{-3/2}q(t/\varepsilon).
\end{align*}
Writing $(\omega,\tau)$ for the Fourier variables dual to $(z,t)$ and changing variables
$\rho=\varepsilon\tau$, we have, for $r\ge0$,
$$
 \varepsilon^{-2r}\abs{\rho}^{2r}
 \le
 \bigl(1+\norm{\omega}^2+\varepsilon^{-2}\rho^2\bigr)^r
 \le
 \varepsilon^{-2r}\bigl(1+\norm{\omega}^2+\rho^2\bigr)^r.
$$
Consequently,
\begin{align*}
 \norm{a_\varepsilon}_{H^r}^2
 &\approx
 \int_{n_\alpha^\perp}\int_{\bbR}
 \bigl(1+\norm{\omega}^2+\varepsilon^{-2}\rho^2\bigr)^r
 \abs{\widehat\eta(\omega)}^2\abs{\widehat\zeta(\rho)}^2\,d\rho\,d\omega
 \approx
 \varepsilon^{-2r}.
\end{align*}
Meanwhile,
\begin{align*}
 \norm{\delta_\alpha a_\varepsilon}_{H^r}^2
 &\approx
 \frac{\varepsilon^{-2}}{\norm{\alpha}^2}
 \int_{n_\alpha^\perp}\int_{\bbR}
 \bigl(1+\norm{\omega}^2+\varepsilon^{-2}\rho^2\bigr)^r
 \abs{\widehat\eta(\omega)}^2\abs{\widehat q(\rho)}^2\,d\rho\,d\omega
 \approx
 \varepsilon^{-2(r+1)}.
\end{align*}
Hence, for $r\ge0$ and $0\le\theta<1$,
$$
 \frac{\norm{\delta_\alpha a_\varepsilon}_{H^r}}
 {\norm{a_\varepsilon}_{H^{r+\theta}}}
 \approx
 \varepsilon^{\theta-1}
 \longrightarrow\infty.
$$
Thus no estimate $H^{r+\theta}\to H^r$ with $\theta<1$ holds uniformly for functions $a$ supported in
$A$.  By contrast, if $a\in C_c^\infty(\bbR^N)$ and
$$
 \operatorname{dist}(\supp a,\alpha^\perp)\ge d>0,
$$
then a smooth cutoff of $1/\ip{\alpha}{\xi}$ away from the wall gives
$$
 \norm{\delta_\alpha a}_{H^r}
 \le
 C_{r,N,\alpha,d}
 \bigl(\norm{a}_{H^r}+\norm{a\circ\sigma_\alpha}_{H^r}\bigr)
 \le
 C_{r,N,\alpha,d}\norm{a}_{H^r}.
$$
Hence uniform wall separation makes $\delta_\alpha$ a zero-order operator, whereas without it one full
derivative is necessary.
\end{remark}

For smooth $b$ and $f$, away from $\alpha^\perp$,
\begin{align}\label{eq:delta-Leibniz}
 \delta_\alpha(bf)(\xi)
 &=
 \frac{b(\xi)f(\xi)-b(\sigma_\alpha\xi)f(\sigma_\alpha\xi)}
 {\ip{\alpha}{\xi}}
 =
 b(\xi)\frac{f(\xi)-f(\sigma_\alpha\xi)}
 {\ip{\alpha}{\xi}}
 +
 \frac{b(\xi)-b(\sigma_\alpha\xi)}
 {\ip{\alpha}{\xi}}f(\sigma_\alpha\xi)\nonumber\\
 &=
 b(\xi)\delta_\alpha f(\xi)
 +(\delta_\alpha b)(\xi)(f\circ\sigma_\alpha)(\xi).
\end{align}
This identity extends across $\alpha^\perp$ by continuity.

Let $K\subset\bbR^N$ be compact and set
$$
 K_G=\bigcup_{g\in G}gK.
$$

For $L\ge1$, let
\begin{align}\label{def:composition}
 \mathcal D=A_L\circ\cdots\circ A_1
\end{align}
be a composition of $L$ operators, 
where each $A_j$ is either a partial derivative
$\partial_{\xi_k}$ or a divided difference operator
$\delta_\alpha$. For $L=0$, set $\mathcal D=I$,
where $I$ is the identity operator.
We write $L(\mathcal D)=L$ for the number of factors
in this composition and set
$$
 \mathcal D_0=I,
 \qquad
 \mathcal D_j=A_j\circ\cdots\circ A_1,
 \qquad 1\le j\le L.
$$

\begin{proposition}\label{prop:compositions}
Let $\mathcal D$ be a composition defined in \eqref{def:composition} of order $L$. For every real $M\ge L$ and every
$a\in H^M(\bbR^N)$,
\begin{equation}\label{eq:composition-H}
 \norm{\mathcal D a}_{H^{M-L}}
 \le
 C_{\mathcal D,M}\norm{a}_{H^M}.
\end{equation}
If $\supp a\subset K$ in the distributional sense, then
$
 \supp(\mathcal D_j a)\subset K_G,
 \ \ 
 0\le j\le L.
$
\end{proposition}

\begin{proof}
For $q\ge0$, every order-one letter satisfies
$$
 \norm{\partial_{\xi_k}F}_{H^q}
 \le
 \norm{F}_{H^{q+1}},
 \qquad
 \norm{\delta_\alpha F}_{H^q}
 \le
 \frac{C_{q,N}}{\norm{\alpha}}\norm{F}_{H^{q+1}}.
$$
Since $M\ge L$, iteration gives, for $1\le j\le L$,
\begin{align*}
 \norm{\mathcal D_j a}_{H^{M-j}}
 &\le
 C_{A_j,M}
 \norm{\mathcal D_{j-1}a}_{H^{M-j+1}}
 \le
 \bigg(\prod_{\ell=1}^jC_{A_\ell,M}\bigg)\norm{a}_{H^M}
 \le
 C_{\mathcal D_j,M}\norm{a}_{H^M}.
\end{align*}
Taking $j=L$ proves \eqref{eq:composition-H}.

Suppose first that $a\in C_c^\infty(\bbR^N)$ and $\supp a\subset K$.  Since $K_G$ is
$G$-invariant, Euclidean differentiation gives
\begin{align*}
 \supp F\subset K_G
 &\Longrightarrow
 \supp(\partial_{\xi_k}F)\subset K_G.
\end{align*}
By \eqref{eq:delta-support},
\begin{align*}
 \supp F\subset K_G
 &\Longrightarrow
 \supp(\delta_\alpha F)
 \subset K_G\cup\sigma_\alpha K_G
 =K_G,
\end{align*}
and hence, by induction,
$$
 \supp(\mathcal D_ja)\subset K_G,
 \qquad
 0\le j\le L.
$$

Now let $a\in H^M$ have distributional support in $K$.  Choose
$\rho\in C_c^\infty(B(0,1))$ with
$\int_{\bbR^N}\rho(\xi)\,d\xi=1$, and set
$$
 \rho_\varepsilon(\xi)
 =
 \varepsilon^{-N}\rho(\xi/\varepsilon),
 \qquad
 a_\varepsilon=\rho_\varepsilon*a,
 \qquad
 K_\varepsilon=K+\overline{B(0,\varepsilon)}.
$$
The mollifier approximation gives
$$
 a_\varepsilon
 \longrightarrow a
 \qquad\text{in }H^M.
$$
Moreover,
$
 \supp a_\varepsilon
 \subset K_\varepsilon,
$
and
$$
 (K_\varepsilon)_G
 =K_G+\overline{B(0,\varepsilon)}.
$$
For $0\le j\le L$, the smooth support result gives
$$
 \supp(\mathcal D_ja_\varepsilon)
 \subset
 K_G+\overline{B(0,\varepsilon)},
$$
whereas \eqref{eq:composition-H} gives
\begin{align*}
 \norm{\mathcal D_j(a_\varepsilon-a)}_{H^{M-j}}
 &\le
 C_{\mathcal D_j,M}\norm{a_\varepsilon-a}_{H^M}
 \longrightarrow0.
\end{align*}
Thus, for $\phi\in C_c^\infty(\bbR^N\setminus K_G)$ and all sufficiently small $\varepsilon$,
\begin{align*}
 \left\langle\mathcal D_j a,\phi\right\rangle
 &=
 \lim_{\varepsilon\to0}
 \left\langle\mathcal D_j a_\varepsilon,\phi\right\rangle
 =0,
\end{align*}
so $\supp(\mathcal D_j a)\subset K_G$ for $0\le j\le L$.
\end{proof}

\section{Uniform Dunkl-kernel estimates across spectral walls}\label{sec:5}

We first prove the product case.  We then prove an evolution estimate for simultaneous transitions and
use induction over spatial sectors to prove the general case.

\subsection{Coordinate reflections}

We first consider the group generated by reflections across
the coordinate hyperplanes in $\bbR^N$. In this case, the
Dunkl kernel is a product of one-dimensional Dunkl kernels.
We use this product formula and the corresponding
one-dimensional estimates to prove Proposition~\ref{thm:joint}.

\begin{proposition}
For this group, Proposition~\ref{thm:joint} holds for all
non-negative multiplicities.
\end{proposition}

\begin{proof}
For the rank-one multiplicity $k\ge0$ \cite[Section~2]{RdJ2002},
$$
 E_k(iu)
 =
 j_{k-1/2}(u)
 +
 \frac{iu}{2k+1}j_{k+1/2}(u),
 \qquad
 j_\nu(u)
 =
 \Gamma(\nu+1)
 \sum_{n=0}^\infty
 \frac{(-1)^n(u/2)^{2n}}
 {n!\,\Gamma(n+\nu+1)},
 \qquad
 u\in\bbR.
$$
The function $j_\nu$ is even and entire. For $u>0$,
$$
 j_\nu(u)
 =
 2^\nu\Gamma(\nu+1)u^{-\nu}J_\nu(u).
$$
For $\nu\ge-1/2$, the standard Bessel estimate and
termwise differentiation give
$$
 \abs{j_\nu(u)}
 \le
 C_\nu(1+\abs{u})^{-\nu-1/2},
 \qquad
 j_\nu'(u)
 =
 -\frac{u}{2(\nu+1)}j_{\nu+1}(u).
$$
Consequently,
\begin{align}\label{eq:rank-one}
 \abs{E_k(iu)}
 +
 \abs{\frac d{du}E_k(iu)}
 &\le
 C_k\Bigl(
 (1+\abs{u})^{-k}
 +\abs{u}(1+\abs{u})^{-k-1}
 +(1+\abs{u})^{-k-1}
 +\abs{u}^2(1+\abs{u})^{-k-2}
 \Bigr)\notag\\
 &\le
 C_k(1+\abs{u})^{-k},
 \qquad
 u\in\bbR.
\end{align}

Let the multiplicities be $k_1,\ldots,k_N$ and set
$
 M_\mathcal A=\sup_{\xi\in \mathcal A}\norm{\xi}.
$

For $\xi\in \mathcal A\setminus\calW$, factorisation gives
$$
 E_\kappa(i\xi,y)
 =
 \prod_{\ell=1}^NE_{k_\ell}(i\xi_\ell y_\ell) \qquad \text{and}\qquad 
 w_\kappa(\xi)^{1/2}\Omega_\kappa(y)^{1/2}
 \approx
 \prod_{\ell=1}^N
 \abs{\xi_\ell}^{k_\ell}(1+\abs{y_\ell})^{k_\ell}.
$$
For each $1\le\ell\le N$,
\begin{align*}
 \abs{\xi_\ell}^{k_\ell}(1+\abs{y_\ell})^{k_\ell}
 (1+\abs{\xi_\ell y_\ell})^{-k_\ell}
 &=
 \left(
 \frac{\abs{\xi_\ell}+\abs{\xi_\ell y_\ell}}
 {1+\abs{\xi_\ell y_\ell}}
 \right)^{k_\ell}
 \le
 (1+M_\mathcal A)^{k_\ell}.
\end{align*}
Hence \eqref{eq:rank-one} gives
\begin{align*}
 w_\kappa(\xi)^{1/2}\Omega_\kappa(y)^{1/2}
 \abs{E_\kappa(i\xi,y)}
 &\le
 C_\kappa
 \prod_{\ell=1}^N
 \abs{\xi_\ell}^{k_\ell}(1+\abs{y_\ell})^{k_\ell}
 (1+\abs{\xi_\ell y_\ell})^{-k_\ell}
 \le
 C_\kappa\prod_{\ell=1}^N(1+M_\mathcal A)^{k_\ell}
 \le
 C_{A,\kappa}.
\end{align*}
Moreover, for $1\le j\le N$,
\begin{align*}
 \partial_{y_j}E_\kappa(i\xi,y)
 &=
 \xi_j
 \left.\frac d{du}E_{k_j}(iu)\right|_{u=\xi_jy_j}
 \prod_{\ell\ne j}E_{k_\ell}(i\xi_\ell y_\ell)
 \end{align*}
 and
 \begin{align*}
w_\kappa(\xi)^{1/2}\Omega_\kappa(y)^{1/2}
 \abs{\partial_{y_j}E_\kappa(i\xi,y)}
 &\le
 C_\kappa\abs{\xi_j}
 \prod_{\ell=1}^N
 \abs{\xi_\ell}^{k_\ell}(1+\abs{y_\ell})^{k_\ell}
 (1+\abs{\xi_\ell y_\ell})^{-k_\ell}
 \le
 C_\kappa M_\mathcal A\prod_{\ell=1}^N(1+M_\mathcal A)^{k_\ell}
 \le
 C_{A,\kappa}.
\end{align*}
This proves \eqref{eq:joint}.
\end{proof}

\subsection{An estimate for a matrix differential equation}

In the next subsection, we estimate the Dunkl kernel and its first
derivatives by studying ordinary differential equations along rays
in the spatial variable. These equations have oscillatory coefficients
with frequencies $|\langle\beta,\lambda\rangle|$, $\beta\in R_+$.
As $\lambda$ approaches the reflecting hyperplanes, several of these
frequencies may tend to zero at different rates. We prove below an
estimate for such equations that shows explicitly how the bound depends
on each frequency. This dependence will allow us to obtain a uniform
bound for the kernel and its first derivatives after multiplication
by the spectral weight $w_\kappa(\lambda)^{1/2}$.

We shall use the following estimate for a finite system of linear
differential equations. The coefficient matrix is a sum of matrices
whose entries oscillate at possibly different frequencies. We do not
assume that these matrices commute.

For $\varepsilon>0$, set
$$
 q_\varepsilon(u)=\frac{u}{1+\varepsilon u},
 \qquad u>0.
$$

\begin{proposition}\label{prop:simultaneous}
Let $\mathscr B$ be a finite index set, and let $k_\beta\ge0$ and
$\varepsilon_\beta>0$ for $\beta\in\mathscr B$.
Let $a\ge1$ and $c_0,b_0,C_{\rm rem}>0$.
Fix an orthonormal basis of a finite-dimensional complex Hilbert
space $\mathcal H$.

For each $\beta\in\mathscr B$, suppose that a permutation of this
basis puts $\mathsf A_\beta(t)$ into a direct sum of a possibly
empty zero block and finitely many blocks of the form
$$
 \frac{k_\beta}{t+d_{\beta,\ell}}
 \begin{pmatrix}
 0&
 e^{-i\eta_{\beta,\ell}\varepsilon_\beta
 b_{\beta,\ell}(t+d_{\beta,\ell})}\\[8pt]
 e^{i\eta_{\beta,\ell}\varepsilon_\beta
 b_{\beta,\ell}(t+d_{\beta,\ell})}&0
 \end{pmatrix},
 \qquad t\ge0.
$$
The permutation, the block decomposition, and all block parameters
are independent of $t$. They may depend on $\beta$, and
$$
 d_{\beta,\ell}\ge c_0a,
 \qquad
 b_{\beta,\ell}\ge b_0,
 \qquad
 \eta_{\beta,\ell}\in\{-1,1\}.
$$
Let $\mathsf R:[0,\infty)\to\mathcal L(\mathcal H)$ be strongly
measurable and satisfy
$$
 \norm{\mathsf R(t)}
 \le C_{\rm rem}(a+t)^{-2}
 \quad\text{for almost every }t\ge0.
$$
Set
$$
 \mathsf A(t)
 =\sum_{\beta\in\mathscr B}\mathsf A_\beta(t)+\mathsf R(t).
$$
For $s\ge0$, let $\mathscr U(t,s)$ be the unique solution, locally
absolutely continuous in $t\ge s$, of
$$
 \partial_t\mathscr U(t,s)=\mathsf A(t)\mathscr U(t,s)
 \quad\text{for almost every }t\ge s,
 \qquad
 \mathscr U(s,s)=\mathrm{Id}.
$$
Then
\begin{equation}\label{eq:simultaneous-evolution}
 \norm{\mathscr U(t,s)}
 \le
 C\prod_{\beta\in\mathscr B}
 \left(
 \frac{q_{\varepsilon_\beta}(a+t)}
 {q_{\varepsilon_\beta}(a+s)}
 \right)^{k_\beta},
 \qquad 0\le s\le t.
\end{equation}
Here the norms are operator norms on $\mathcal H$, and $C$ depends
only on $|\mathscr B|$, $\{k_\beta\}_{\beta\in\mathscr B}$,
$c_0$, $b_0$, and $C_{\rm rem}$. In particular, it is independent
of $\dim\mathcal H$, $a$, the number of blocks, and the parameters
$\varepsilon_\beta$, $d_{\beta,\ell}$, $b_{\beta,\ell}$, and
$\eta_{\beta,\ell}$.
\end{proposition}

\begin{proof}
Each nonzero block of $\mathsf A_\beta(t)$ is Hermitian and has
eigenvalues $\pm k_\beta/(t+d_{\beta,\ell})$.
The norm of a direct sum is the maximum of the norms of its
blocks. Let $c_*=\min\{1,c_0\}$ and $K_0=\sum_{\beta\in\mathscr B}k_\beta$. Since
$$
 t+d_{\beta,\ell}\ge t+c_0a\ge c_*(a+t),
$$
we have
$$
 \norm{\mathsf A_\beta(t)}
 \le\frac{k_\beta}{t+c_0a}
 \le\frac{k_\beta}{c_*(a+t)},
 \qquad
 \norm{\mathsf A(t)}
 \le\frac{K_0}{c_*(a+t)}
      +\frac{C_{\rm rem}}{(a+t)^2}
$$
for almost every $t\ge0$. Thus $\mathsf A$ is integrable on every
bounded interval. For fixed $s<T$, successive approximation in
$$
 \mathscr U(t,s)
 =\mathrm{Id}+\int_s^t\mathsf A(r)\mathscr U(r,s)\,dr
$$
converges uniformly for $s\le t\le T$: the norm of its term with
$n$ integrals is at most
$$
 \frac1{n!}\left(\int_s^T\norm{\mathsf A(r)}\,dr\right)^n.
$$
This gives a locally absolutely continuous solution, and
Gronwall's inequality gives uniqueness.

We first estimate an integral of each $\mathsf A_\beta$.
Fix one block and write
$$
 d=d_{\beta,\ell},
 \qquad
 \omega=\eta_{\beta,\ell}\varepsilon_\beta b_{\beta,\ell}.
$$
Then $|\omega|=\varepsilon_\beta b_{\beta,\ell}>0$.
Integration by parts on a finite interval, followed by passage to
the limit at its upper endpoint, gives
\begin{align*}
 \int_t^\infty
 \frac{k_\beta e^{-i\omega(r+d)}}{r+d}\,dr
 &=
 \frac{k_\beta e^{-i\omega(t+d)}}{i\omega(t+d)}
 -\frac{k_\beta}{i\omega}
  \int_t^\infty\frac{e^{-i\omega(r+d)}}{(r+d)^2}\,dr.
\end{align*}
The boundary term at infinity tends to zero, and the last
integral is absolutely convergent. Consequently,
\begin{align*}
 \left|
 \int_t^\infty
 \frac{k_\beta e^{-i\omega(r+d)}}{r+d}\,dr
 \right|
 &\le
 \frac{k_\beta}{|\omega|(t+d)}
 +\frac{k_\beta}{|\omega|}
  \int_t^\infty\frac{dr}{(r+d)^2}
 =\frac{2k_\beta}{|\omega|(t+d)}
 \le\frac{2k_\beta}{b_0c_*\varepsilon_\beta(a+t)}.
\end{align*}
The other off-diagonal entry is the complex conjugate.
Thus
$$
 \mathsf Q_\beta(t)
 =-\int_t^\infty\mathsf A_\beta(r)\,dr
$$
is well defined. Applying the same direct-sum norm formula as
above, and setting $C_0=2/(b_0c_*)$, gives
$$
 \norm{\mathsf Q_\beta(t)}
 \le\frac{C_0k_\beta}{\varepsilon_\beta(a+t)}.
$$
Also,
$$
 \mathsf Q_\beta(t+h)-\mathsf Q_\beta(t)
 =\int_t^{t+h}\mathsf A_\beta(r)\,dr.
$$
Since $\mathsf A_\beta$ is continuous, $\mathsf Q_\beta$ is
continuously differentiable and
$\mathsf Q_\beta'=\mathsf A_\beta$.

Choose
$$
 K=\max\{2,2C_0K_0\}.
$$
Then $C_0K_0/K\le1/2$.
Fix $0\le s<t$. Divide $[s,t]$ at the distinct points
$$
 \tau_\beta=\frac K{\varepsilon_\beta}-a,
 \qquad \beta\in\mathscr B,
$$
that lie strictly between $s$ and $t$.
Write the resulting partition as
$$
 s=u_0<u_1<\cdots<u_m=t,
 \qquad
 1\le m\le|\mathscr B|+1,
$$
where $u_1,\ldots,u_{m-1}$ are the distinct points $\tau_\beta$
in $(s,t)$, listed in increasing order.
Thus each open interval $(u_{j-1},u_j)$ contains no point
$\tau_\beta$.

We first prove the estimate on each subinterval.
Fix $j\in\{1,\ldots,m\}$ and write
$$
 u=u_{j-1},
 \qquad
 v=u_j.
$$
Each $\tau_\beta$ satisties either $\tau_\beta\le u$ or $\tau_\beta\ge v$. 
Set
$$
 \mathscr F=\{\beta\in\mathscr B:\tau_\beta\le u\},
 \qquad
 \mathscr S=\{\beta\in\mathscr B:\tau_\beta\ge v\}.
$$
These two sets partition $\mathscr B$. 
Note that $\tau_\beta$ is the point where $\varepsilon_\beta(a+r)=K$. 
Then, for every $r\in[u,v]$,
$$
 \varepsilon_\beta(a+r)\ge K
 \quad\text{if }\beta\in\mathscr F,
 \qquad
 \varepsilon_\beta(a+r)\le K
 \quad\text{if }\beta\in\mathscr S.
$$

Write
$$
 \mathsf A_{\mathscr F}
 =\sum_{\beta\in\mathscr F}\mathsf A_\beta,
 \qquad
 \mathsf A_{\mathscr S}
 =\sum_{\beta\in\mathscr S}\mathsf A_\beta,
 \qquad
 \mathsf Q_{\mathscr F}
 =\sum_{\beta\in\mathscr F}\mathsf Q_\beta.
$$
Empty sums are understood to be zero.
For $r\in[u,v]$,
$$
 \norm{\mathsf Q_{\mathscr F}(r)}
 \le
 C_0\sum_{\beta\in\mathscr F}
 \frac{k_\beta}{\varepsilon_\beta(a+r)}
 \le\frac{C_0K_0}{K}
 \le\frac12.
$$
If we define
$
 \mathsf M(r)=\mathrm{Id}+\mathsf Q_{\mathscr F}(r)
$, 
it follows that $\mathsf M'=\mathsf A_{\mathscr F}$ and 
$
 \mathsf M(r)
$
is invertible on the whole closed interval, with
$$
 \norm{\mathsf M(r)}\le\frac32,
 \qquad
 \norm{\mathsf M(r)^{-1}}
 \le\sum_{n=0}^\infty\norm{\mathsf Q_{\mathscr F}(r)}^n
 \le2.
$$

Let $Y:[u,v]\to\mathcal H$ solve $Y'=\mathsf A Y$, and define
$V=\mathsf M^{-1}Y$. This is a locally absolutely continuous
$\mathcal H$-valued function. Differentiation gives
\begin{align*}
 \mathsf M V'
 &=
 \bigl[
 (\mathsf A_{\mathscr S}+\mathsf A_{\mathscr F}+\mathsf R)
 (\mathrm{Id}+\mathsf Q_{\mathscr F})
 -\mathsf A_{\mathscr F}
 \bigr]V
 =
 \bigl[
 \mathsf A_{\mathscr S}
 +\mathsf A_{\mathscr S}\mathsf Q_{\mathscr F}
 +\mathsf A_{\mathscr F}\mathsf Q_{\mathscr F}
 +\mathsf R\mathsf M
 \bigr]V
\end{align*}
almost everywhere on $[u,v]$. Therefore
$$
 V'=(\mathsf A_{\mathscr S}
      +\widetilde{\mathsf R}_{\mathscr F})V,
$$
where
\begin{align*}
 \widetilde{\mathsf R}_{\mathscr F}
 =
 \mathsf M^{-1}\bigl(
 &\mathsf A_{\mathscr S}\mathsf Q_{\mathscr F}
 -\mathsf Q_{\mathscr F}\mathsf A_{\mathscr S}
 +\mathsf A_{\mathscr F}\mathsf Q_{\mathscr F}
 +\mathsf R\mathsf M
 \bigr).
\end{align*}
Here we used
$\mathsf M\mathsf A_{\mathscr S}
=\mathsf A_{\mathscr S}
 +\mathsf Q_{\mathscr F}\mathsf A_{\mathscr S}$.
No commutation assumption is used.

We next bound the integral of $\widetilde{\mathsf R}_{\mathscr F}$.
The preceding estimates imply that
\begin{align*}
 \norm{\widetilde{\mathsf R}_{\mathscr F}(r)}
 &\le
 3\norm{\mathsf R(r)}
 +\bigl(
 4\norm{\mathsf A_{\mathscr S}(r)}
 +2\norm{\mathsf A_{\mathscr F}(r)}
 \bigr)\norm{\mathsf Q_{\mathscr F}(r)}
 \le
 \frac{3C_{\rm rem}}{(a+r)^2}
 +\frac{4C_0K_0}{c_*}
  \sum_{\beta\in\mathscr F}
  \frac{k_\beta}{\varepsilon_\beta(a+r)^2}.
\end{align*}
Since
$$
 \int_u^v\frac{dr}{(a+r)^2}
 =\frac1{a+u}-\frac1{a+v}
 \le\frac1{a+u},
$$
we obtain
\begin{align*}
 \int_u^v\norm{\widetilde{\mathsf R}_{\mathscr F}(r)}\,dr
 &\le
 \frac{3C_{\rm rem}}{a+u}
 +\frac{4C_0K_0}{c_*}
  \sum_{\beta\in\mathscr F}
  \frac{k_\beta}{\varepsilon_\beta(a+u)}
 \le
 3C_{\rm rem}+\frac{4C_0K_0^2}{c_*K}
 =:C_1.
\end{align*}
The last inequality uses $a\ge1$ and
$\varepsilon_\beta(a+u)\ge K$ for $\beta\in\mathscr F$.

The matrix $\mathsf A_{\mathscr S}(r)$ is Hermitian, and its
quadratic form is bounded above by
$$
 \ip{\mathsf A_{\mathscr S}(r)z}{z}
 \le
 \sum_{\beta\in\mathscr S}
 \frac{k_\beta}{r+c_0a}\norm{z}^2,
 \qquad z\in\mathcal H.
$$
Consequently, for almost every $r\in[u,v]$,
\begin{align*}
 \frac12\frac d{dr}\norm{V(r)}^2
 &=
 \operatorname{Re}
 \ip{(\mathsf A_{\mathscr S}(r)
       +\widetilde{\mathsf R}_{\mathscr F}(r))V(r)}{V(r)}
 \le
 \left(
 \sum_{\beta\in\mathscr S}\frac{k_\beta}{r+c_0a}
 +\norm{\widetilde{\mathsf R}_{\mathscr F}(r)}
 \right)\norm{V(r)}^2.
\end{align*}
Applying Gronwall's inequality to $\norm{V(r)}^2$ and taking
square roots gives
\begin{align*}
 \norm{V(v)}
 &\le
 \exp\left(
 \int_u^v\norm{\widetilde{\mathsf R}_{\mathscr F}(r)}\,dr
 \right)
 \prod_{\beta\in\mathscr S}
 \exp\left(k_\beta\int_u^v\frac{dr}{r+c_0a}\right)
 \norm{V(u)}
  \le
 e^{C_1}
 \prod_{\beta\in\mathscr S}
 \left(\frac{v+c_0a}{u+c_0a}\right)^{k_\beta}
 \norm{V(u)}.
\end{align*}

Put
$$
 C_{c_0}=\frac{\max\{1,c_0\}}{\min\{1,c_0\}}.
$$
Recall that $q_\varepsilon(u)=\frac{u}{1+\varepsilon u}$. For $\beta\in\mathscr S$, the inequality
$\varepsilon_\beta(a+v)\le K$ gives
\begin{align*}
 \frac{v+c_0a}{u+c_0a}
 &\le C_{c_0}\frac{a+v}{a+u}
 =
 C_{c_0}
 \frac{q_{\varepsilon_\beta}(a+v)}
 {q_{\varepsilon_\beta}(a+u)}
 \frac{1+\varepsilon_\beta(a+v)}
 {1+\varepsilon_\beta(a+u)}
 \le
 C_{c_0}(1+K)
 \frac{q_{\varepsilon_\beta}(a+v)}
 {q_{\varepsilon_\beta}(a+u)}.
\end{align*}
For $\beta\in\mathscr F$, the ratio
$q_{\varepsilon_\beta}(a+v)/
 q_{\varepsilon_\beta}(a+u)$ is at least $1$.
Since $\sum_{\beta\in\mathscr S}k_\beta\le K_0$, it follows that
\begin{align*}
 \norm{V(v)}
 &\le
 e^{C_1}
 \prod_{\beta\in\mathscr S}
 \bigl(C_{c_0}(1+K)\bigr)^{k_\beta}
  \left(
 \frac{q_{\varepsilon_\beta}(a+v)}
 {q_{\varepsilon_\beta}(a+u)}
 \right)^{k_\beta}
 \norm{V(u)}\\
 &\le
 e^{C_1}\bigl(C_{c_0}(1+K)\bigr)^{K_0}
 \prod_{\beta\in\mathscr B}
 \left(
 \frac{q_{\varepsilon_\beta}(a+v)}
 {q_{\varepsilon_\beta}(a+u)}
 \right)^{k_\beta}\norm{V(u)}.
\end{align*}
Returning to $Y$ at both endpoints gives
$$
 \norm{Y(v)}
 \le
 D\prod_{\beta\in\mathscr B}
 \left(
 \frac{q_{\varepsilon_\beta}(a+v)}
 {q_{\varepsilon_\beta}(a+u)}
 \right)^{k_\beta}\norm{Y(u)},
 \qquad
 D=3e^{C_1}\bigl(C_{c_0}(1+K)\bigr)^{K_0}.
$$
Indeed,
$\norm{\mathsf M(v)}\norm{\mathsf M(u)^{-1}}\le3$.
All estimates were obtained on the closed interval $[u,v]$,
so this bound also holds when an endpoint equals some
$\tau_\beta$.

Finally, apply the preceding estimate on each $[u_j,u_{j+1}]$ to the
same solution $Y(r)=\mathscr U(r,s)y$, where $y\in\mathcal H$.
For every $\beta\in\mathscr B$,
$$
 \prod_{j=0}^{m-1}
 \frac{q_{\varepsilon_\beta}(a+u_{j+1})}
 {q_{\varepsilon_\beta}(a+u_j)}
 =
 \frac{q_{\varepsilon_\beta}(a+t)}
 {q_{\varepsilon_\beta}(a+s)}.
$$
Thus
$$
 \norm{\mathscr U(t,s)y}
 \le
 D^m
 \prod_{\beta\in\mathscr B}
 \left(
 \frac{q_{\varepsilon_\beta}(a+t)}
 {q_{\varepsilon_\beta}(a+s)}
 \right)^{k_\beta}\norm{y}.
$$
Taking the supremum over $\norm{y}=1$ proves
\eqref{eq:simultaneous-evolution} with
$$
 C=D^{|\mathscr B|+1}.
$$
The case $s=t$ follows from $\mathscr U(s,s)=\mathrm{Id}$ and
$D\ge1$. The formulas for $c_*$, $C_0$, $K$, $C_1$, $C_{c_0}$,
and $D$ give precisely the stated dependence of $C$.
\end{proof}

\subsection{Proof for an arbitrary root system}

We first estimate the kernel on a compact set.
We then move along rays in the chamber, starting each ray at a
point where a bound is already known.
The equations along a ray involve all the functions
$E_\kappa(ig\lambda,x)$, $g\in G$, so we estimate them together.
We multiply these functions by a weight and exponential factors
and apply Proposition~\ref{prop:simultaneous} to the resulting
matrix equation.
This gives a bound at the endpoint of the ray from the bound
at its starting point.
Repeating this step gives the estimate throughout the chamber.
Multiplication by $w_\kappa(\lambda)^{1/2}$ makes the bound uniform
as $\lambda$ approaches the reflecting hyperplanes.
We then differentiate the equations and use the same argument
to estimate the first spatial derivatives.

\begin{proof}[Proof of Proposition~\ref{thm:joint}]

Fix a compact annulus $\mathcal A$ and put
$$
 M_{\mathcal A}=\sup_{\xi\in\mathcal A}\norm{\xi}.
$$
Let $\lambda\in\mathcal A\setminus\calW$.  Thus
$\ip{\beta}{\lambda}\ne0$ for every $\beta\in R$.  

\vskip 0.8cm
{\bf Step 1.}
We first assume that $R$ spans $\bbR^N$.

Here we use the subdivision of a chamber in \cite[Section~4]{Langen2026}
and the differential equation in \cite[Proposition~3.2]{Langen2026}.
The constants in the resulting kernel estimates depend only on $R$, $\kappa$, and the upper bound $M_{\mathcal A}$ for $\norm{\lambda}$.
We will use this fact to treat the case in which $R$ does not span $\bbR^N$.

\vskip0.5cm
{\bf Step 1.1.}

Let $\alpha_1,\ldots,\alpha_N$ be the simple roots corresponding to $R_+$.
They form a basis of $\bbR^N$.
For $1\le j\le N$, define the linear function
$$
 \rho_j(x)=\ip{\alpha_j}{x}.
$$
The positive chamber and its closure are
$$
 \mathcal C_+=\{x:\rho_j(x)>0\text{ for all }j\},
 \qquad
 \overline{\mathcal C_+}
 =\{x:\rho_j(x)\ge0\text{ for all }j\}.
$$
Let $\mathfrak S_N$ denote the set of all permutations of
$\{1,\ldots,N\}$.  For each $\pi\in\mathfrak S_N$, define
$$
 S_\pi
 =
 \{x:\rho_{\pi(1)}(x)\ge\cdots\ge\rho_{\pi(N)}(x)\ge0\}.
$$
Every list of $N$ non-negative numbers can be put in decreasing order.
Consequently,
$$
 \overline{\mathcal C_+}=\bigcup_{\pi\in\mathfrak S_N}S_\pi.
$$
Here $\pi$ only specifies an ordering of the simple-root coordinates.
The sets $S_\pi$ may overlap when some coordinates are equal.
Since there are only finitely many such sets, it suffices to prove
the estimates on each of them and take the largest constant.

Fix one of these sets and relabel the simple roots so that it is
$$
 S=\{x:\rho_1(x)\ge\cdots\ge\rho_N(x)\ge0\}.
$$
Let $\varpi_1,\ldots,\varpi_N$ be the basis dual to the simple roots
with respect to the Euclidean inner product:
$$
 \ip{\alpha_k}{\varpi_j}
 =
 \begin{cases}
  1,&k=j,\\
  0,&k\ne j.
 \end{cases}
$$
This basis exists and is unique because the simple roots form a basis.
In particular,
$$
 x=\sum_{j=1}^N\rho_j(x)\varpi_j.
$$
For $1\le i\le N$, define
$$
 H_i=\sum_{j=1}^i\varpi_j.
$$
Thus $\rho_j(H_i)$ means the number $\ip{\alpha_j}{H_i}$, and
$$
 \rho_j(H_i)
 =
 \begin{cases}
  1,&1\le j\le i,\\
  0,&i<j\le N.
 \end{cases}
$$
Adding $rH_i$ to a point increases its first $i$ coordinates
$\rho_j$ by $r$ and leaves the remaining coordinates unchanged.

Define
$
 \mathfrak a^0
 =
 \{x\in S:\rho_j(x)\le1\text{ for }1\le j\le N\},
$
and, for $1\le i\le N$, define
$
 \mathfrak a^i
 =
 \{x\in S:\rho_1(x)=\cdots=\rho_i(x)\ge1\}.
$

Also put $\mathfrak a^{N+1}=\varnothing$ and $\rho_{N+1}\equiv0$.
The set $\mathfrak a^0$ is compact, since it is closed and
$$
 \norm{x}
 \le\sum_{j=1}^N\norm{\varpi_j}
 \qquad (x\in\mathfrak a^0).
$$
Moreover,
$$
 S=\mathfrak a^0\cup\mathfrak a^1,
 \qquad
 \mathfrak a^{i+1}\subset\mathfrak a^i
 \quad (1\le i<N).
$$
We will proceed from $i=N$ down to $i=1$.
At each step, the estimate on $\mathfrak a^i$ will follow from
the estimates on $\mathfrak a^{i+1}$ and the compact set $\mathfrak a^0$.

To explain this step, fix $x\in\mathfrak a^i$ and set
$
 a=\max\{1,\rho_{i+1}(x)\},
 \
 t=\rho_i(x)-a,
 \
 x_0=x-tH_i.
$

Since $\rho_i(x)\ge\rho_{i+1}(x)$ and $\rho_i(x)\ge1$, we have $t\ge0$.
The coordinates of $x_0$ are
$$
 \rho_j(x_0)
 =
 \begin{cases}
  a,&j\le i,\\
  \rho_j(x),&j>i.
 \end{cases}
$$
If $a=1$, all these coordinates belong to $[0,1]$, so
$x_0\in\mathfrak a^0$.
If $a>1$, then $i<N$ and
$\rho_{i+1}(x_0)=\rho_{i+1}(x)=a$, so
$x_0\in\mathfrak a^{i+1}$.
For $r\ge0$, write
$$
 x_r=x_0+rH_i.
$$
The first $i$ coordinates of $x_r$ equal $\rho_i(x_r)=a+r$, and the other
coordinates are at most $a$.
Hence $x_r\in\mathfrak a^i$ for every $r\ge0$.
The point originally chosen is $x_t=x$, and $\rho_i(x_t)=a+t$.
For $i=N$, the convention $\rho_{N+1}=0$ gives $a=1$,
so the first induction step always starts in $\mathfrak a^0$.

Every positive root has an expansion
$$
 \alpha=\sum_{j=1}^Nc_j\alpha_j,
 \qquad c_j\ge0.
$$
For such a root, we have 
$
 \ip{\alpha}{H_i}=\sum_{j=1}^ic_j,
$
which is zero exactly when $c_1=\cdots=c_i=0$.
Therefore, we can define 
$
 I=\{\alpha_{i+1},\ldots,\alpha_N\}
 $ and $
 R_I=\{\alpha\in R:\ip{\alpha}{H_i}=0\}=R\cap\operatorname{span}I$.
Note that $R_I$ consists precisely of the roots whose values do not change
along the line $x_0+rH_i$.

For $\alpha\in R_+\setminus R_I$, define
\begin{align}\label{def:b-d}
 b_\alpha=\ip{\alpha}{H_i}=\sum_{j=1}^ic_j>0,
 \qquad
 d_\alpha=\frac{\ip{\alpha}{x_0}}{b_\alpha}.
\end{align}
Using the coordinates of $x_0$, we obtain
 $
 \ip{\alpha}{x_r}=b_\alpha(r+d_\alpha)
$
and
$
 d_\alpha
 =
 a+\frac{\sum_{j>i}c_j\rho_j(x_0)}{b_\alpha}
 \ge a.
 $
 
The numbers $b_\alpha$ depends only
on the ordering of the simple roots, the index $i$, and the root
$\alpha$.
There are finitely many such choices, and $b_\alpha>0$ whenever
$\alpha\notin R_I$.
We may therefore choose constants $b_0>0$ depending only on $R$ such that
$$
 b_0\le b_\alpha$$
for all these choices.
It follows that
\begin{equation}\label{eq:ray-geometry}
 \begin{gathered}
  b_0(a+r)\le\ip{\alpha}{x_r},
  \qquad
  \alpha\in R_+\setminus R_I,\quad r\ge0.
 \end{gathered}
\end{equation}
In particular, none of these inner products vanishes on the ray.

Define
$$
 W(u)
 =
 \prod_{\alpha\in R_+\setminus R_I}
 \abs{\ip{\alpha}{u}}^{\kappa(\alpha)}.$$
Then $W$ is positive and smooth on the open set
\begin{align}\label{a reg}
 \mathfrak a_{\mathrm{reg}}^I
 =
 \{u:\ip{\alpha}{u}\ne0
       \text{ for every }\alpha\in R\setminus R_I\},
\end{align}
which contains the ray.
We now use the weight $W$ to derive a matrix differential
equation along the ray $x_r=x_0+rH_i$, $r\ge0$.
For $g\in G$, set
$$
 F_g(u)=E_\kappa(ig\lambda,u).
$$
The components of the unknown vector will be
$$
 W(x_r)e^{-i\ip{x_r}{g\lambda}}F_g(x_r),
 \qquad g\in G.
$$
The derivatives of $W$ and of the exponential will cancel
the terms multiplying $F_g$ itself.
The weight comparison above will then allow us to recover
estimates with the full spatial weight
$\Omega_\kappa(x_r)^{1/2}$.

To derive this equation, let $\partial_H$ denote the ordinary
directional derivative for a real vector $H$, and recall
the corresponding Dunkl operator:
$$
 T_Hf(u)
 =
 \partial_Hf(u)
 +
 \sum_{\alpha\in R_+}
 \kappa(\alpha)\ip{\alpha}{H}
 \frac{f(u)-f(\sigma_\alpha u)}{\ip{\alpha}{u}}.
$$

For smooth $f$, each fraction is understood by its smooth extension
on the reflecting hyperplane.
The eigenfunction identity and covariance of the Dunkl kernel give
$$
 T_{H_i}F_g=i\ip{H_i}{g\lambda}F_g,
 \qquad
 F_g(\sigma_\alpha u)=F_{\sigma_\alpha g}(u).
$$
For $\alpha\in R_+\setminus R_I$, write
$$
 c_\alpha(u)
 =
 \kappa(\alpha)\frac{\ip{\alpha}{H_i}}{\ip{\alpha}{u}}.
$$
The definition of $T_{H_i}$ now yields
\begin{align*}
 \partial_{H_i}F_g(u)
 &=
 i\ip{H_i}{g\lambda}F_g(u)
 -
 \sum_{\alpha\in R_+\setminus R_I}
 c_\alpha(u)\bigl(F_g(u)-F_{\sigma_\alpha g}(u)\bigr).
\end{align*}
The preceding identity holds for every
$u$ in the open set $\mathfrak a_{\mathrm{reg}}^I$ as defined in \eqref{a reg},
including points on the reflecting hyperplanes associated with $R_I$.
Indeed, the roots in $R_I$ do not contribute to $T_{H_i}$, since
$\ip{\alpha}{H_i}=0$ for these roots.
For every remaining root $\alpha\in R_+\setminus R_I$,
the denominator $\ip{\alpha}{u}$ is nonzero on
$\mathfrak a_{\mathrm{reg}}^I$ by definition.
Thus both sides of the equation are smooth on this open set, so we may differentiate
the equation in any spatial direction.

Since
$$
 \partial_{H_i}W(u)
 =
 W(u)\sum_{\alpha\in R_+\setminus R_I}c_\alpha(u),
$$
define
$$
 \Phi_g(u)=W(u)e^{-i\ip{u}{g\lambda}}F_g(u),
 \qquad
 \Phi(r)=\bigl(\Phi_g(x_r)\bigr)_{g\in G}\in\ell^2(G).
$$
Differentiating $\Phi_g$ in the direction $H_i$ cancels both the term
$i\ip{H_i}{g\lambda}F_g$ and the term
$-\sum_\alpha c_\alpha F_g$.
The remaining terms are
$$
 \partial_{H_i}\Phi_g(u)
 =
 \sum_{\alpha\in R_+\setminus R_I}
 c_\alpha(u)
 e^{-i\ip{u}{g\lambda-\sigma_\alpha g\lambda}}
 \Phi_{\sigma_\alpha g}(u).
$$
Under our normalization $\norm{\alpha}^2=2$,
$$
 g\lambda-\sigma_\alpha g\lambda
 =\ip{\alpha}{g\lambda}\alpha.
$$
It follows that $\Phi$ satisfies
$$
 \Phi'(r)=\mathsf A(r)\Phi(r),
$$
where
\begin{equation}\label{eq:langen-entry}
 \mathsf A_{g,h}(r)
 =
 \begin{cases}
 \displaystyle
 \kappa(\alpha)\frac{\ip{\alpha}{H_i}}{\ip{\alpha}{x_r}}
 e^{-i\ip{\alpha}{x_r}\ip{\alpha}{g\lambda}},
 &h=\sigma_\alpha g,\quad \alpha\in R_+\setminus R_I,\\[2mm]
 0,&\text{otherwise}.
 \end{cases}
\end{equation}
This is the equation of \cite[Proposition~3.2]{Langen2026}
for the present choice of $H_i$ and spectral parameter $i\lambda$.

\vskip0.5cm
{\bf Step 1.2.}

For $\beta\in R_+$, set
$$
 \varepsilon_\beta=\abs{\ip{\beta}{\lambda}},
 \qquad
 \mathfrak q_\lambda(u)
 =
 \prod_{\beta\in R_+}
 \left(\frac{u}{1+\varepsilon_\beta u}\right)^{\kappa(\beta)},
 \qquad u>0.
$$

To apply Proposition~\ref{prop:simultaneous}, we group the entries of
$\mathsf A$ according to $\varepsilon_\beta$.
For a root $\gamma$, denote the unique positive root in
$\{\gamma,-\gamma\}$ as $\gamma^+$.
Given $g\in G$ and $\beta\in R_+$, put
$$
 \alpha_{g,\,\beta}=(g\beta)^+
 =\eta_{g,\,\beta}g\beta,
 \qquad
 \eta_{g,\,\beta}\in\{-1,1\}.
$$
Also define
$$
 \widetilde\eta_{g,\,\beta}
 =
 \eta_{g,\,\beta}\operatorname{sgn}\ip{\beta}{\lambda}
 \in\{-1,1\}.
$$
Then $\alpha_{g,\,\beta}\in R_+$, and
\begin{align}\label{eq:g-beta-equality}
 \sigma_{\alpha_{g,\,\beta}}g=g\sigma_\beta,
 \qquad
 \ip{\alpha_{g,\,\beta}}{g\lambda}
 =\widetilde\eta_{g,\,\beta}\varepsilon_\beta,
 \qquad
 \kappa(\alpha_{g,\,\beta})=\kappa(\beta).
\end{align}
If $\alpha_{g,\,\beta}\notin R_I$, set
$b_{g,\,\beta}=b_{\alpha_{g,\,\beta}}$ and 
$d_{g,\,\beta}=d_{\alpha_{g,\,\beta}}$ as defined in \eqref{def:b-d}, 
and define
$$
 (\mathsf A_\beta(r))_{g,g\sigma_\beta}
 =
 \frac{\kappa(\beta)}{r+d_{g,\,\beta}}
 e^{-i\widetilde\eta_{g,\,\beta}\varepsilon_\beta
 b_{g,\,\beta}(r+d_{g,\,\beta})}.
$$
Using \eqref{eq:g-beta-equality} and $b_\alpha(r+d_\alpha)=\ip{\alpha}{x_r}$, we can show that
$$
\mathsf A_{g,\sigma_{\alpha_{g,\,\beta}} g}(r)
=
\kappa(\alpha_{g,\,\beta})\frac{\ip{\alpha_{g,\,\beta}}{H_i}}{\ip{\alpha}{x_r}}
 e^{-i\ip{\alpha_{g,\,\beta}}{g\lambda}\ip{\alpha_{g,\,\beta}}{x_r}}
 =(\mathsf A_\beta(r))_{g,g\sigma_\beta}.
$$
If $\alpha_{g,\,\beta}\in R_I$, let
$$
 (\mathsf A_\beta(r))_{g,g\sigma_\beta}
 =0
 =\mathsf A_{g,\sigma_{\alpha_{g,\,\beta}} g}(r).
 $$
Let all other entries of $\mathsf A_\beta$ be zero. Consequently, 
each nonzero entry in \eqref{eq:langen-entry} belongs to exactly one of
these matrices, with $\beta=(g^{-1}\alpha)^+$.
Thus
$$
 \mathsf A(r)=\sum_{\beta\in R_+}\mathsf A_\beta(r).
$$

For a fixed $\beta$, the map $g\mapsto g\sigma_\beta$ is an involution
without fixed points.  It divides $G$ into pairs $\{g,g\sigma_\beta\}$.
If $g'=g\sigma_\beta$, then $g'\beta=-g\beta$, so
$$
 \alpha_{g',\, \beta}=\alpha_{g,\,\beta},
 \qquad
 b_{g',\,\beta}=b_{g,\, \beta},
 \qquad
 d_{g',\,\beta}=d_{g,\,\beta},
 \qquad
 \widetilde\eta_{g',\,\beta}=-\widetilde\eta_{g,\,\beta}.
$$
Thus for $\alpha_{g',\,\beta}\notin R_I$ we have
$$
(\mathsf A_\beta(r))_{g',g'\sigma_\beta}
 =
 \frac{\kappa(\beta)}{r+d_{g,\,\beta}}
 e^{i\widetilde\eta_{g,\,\beta}\varepsilon_\beta
 b_{g,\,\beta}(r+d_{g,\,\beta})}.
$$
Order the coordinate vectors of $\ell^2(G)$ by these pairs.
In this order, every nonzero block of $\mathsf A_\beta(r)$ has the form
$$
 \frac{\kappa(\beta)}{r+d_{g,\,\beta}}
 \begin{pmatrix}
  0&
  e^{-i\widetilde\eta_{g,\,\beta}\varepsilon_\beta
       b_{g,\,\beta}(r+d_{g,\,\beta})}\\
  e^{i\widetilde\eta_{g,\,\beta}\varepsilon_\beta
       b_{g,\,\beta}(r+d_{g,\,\beta})}&0
 \end{pmatrix}.
$$
In particular, it is Hermitian.
The remaining pairs give zero blocks.
The pairing may depend on $\beta$, but it is independent of $r$.
Moreover, $d_{g,\,\beta}\ge a$ and $b_{g,\,\beta}\ge b_0$.
These are precisely the block assumptions in
Proposition~\ref{prop:simultaneous}, with
$$
 \mathscr B=R_+,\quad
 \mathcal H=\ell^2(G),\quad
 k_\beta=\kappa(\beta),\quad
 c_0=1,
 \quad \mathsf R=0.
$$
We may take $C_{\rm rem}=1$, since the remainder is zero.

Let $\mathscr U_{\rm ray}(r,s)$ denote the solution matrix for
$Z'=\mathsf A Z$, normalized by $\mathscr U_{\rm ray}(s,s)=\mathrm{Id}$.
The proposition gives
\begin{equation}\label{eq:ray-evolution}
 \norm{\mathscr U_{\rm ray}(r,s)}
 \le
 C_e\,\frac{\mathfrak q_\lambda(a+r)}
              {\mathfrak q_\lambda(a+s)},
 \qquad 0\le s\le r.
\end{equation}
The constant $C_e$ depends only on $R$ and $\kappa$.
It is independent of the ordering of the simple roots and of
$i$, $a$, $x_0$, and $\lambda$.

In particular,
\begin{align}\label{eq:ray-evolution-Phi}
 \norm{\Phi(t)}
 \le
 C_e\,\frac{\mathfrak q_\lambda(a+t)}
              {\mathfrak q_\lambda(a)}
 \norm{\Phi(0)}.
\end{align}

\vskip0.5cm
{\bf Step 1.3.}

We now use this estimate to prove the bound for $E_\kappa$.

We claim that constants $C_i$ can be chosen so that
\begin{align}\label{claim for v=0}
 \max_{g\in G}
 \Omega_\kappa(x)^{1/2}\abs{E_\kappa(ig\lambda,x)}
 \le
 C_i\,\frac{\mathfrak q_\lambda(\rho_i(x))}
              {\mathfrak q_\lambda(1)},
 \qquad x\in\mathfrak a^i,\quad 1\le i\le N.
\end{align}

Indeed, we begin with uniform bounds on the compact set $\mathfrak a^0$,
and then argue from $i=N$ down to $i=1$.
For each $x\in\mathfrak a^i$, the ray constructed above starts
at a point $x_0$ in $\mathfrak a^0$ or $\mathfrak a^{i+1}$,
where a bound is already available.

We first introduce a weight comparison here.
Recall that
$$
 W(u)
 =
 \prod_{\alpha\in R_+\setminus R_I}
 \abs{\ip{\alpha}{u}}^{\kappa(\alpha)},
 \qquad
\Omega_\kappa(u)
 =
 \prod_{\alpha\in R_+}
 \bigl(1+\abs{\ip{\alpha}{u}}\bigr)^{2\kappa(\alpha)}.
$$
Let
$$
 \Omega_I(u)
 =
 \prod_{\alpha\in R_+\cap R_I}
 \bigl(1+\abs{\ip{\alpha}{u}}\bigr)^{2\kappa(\alpha)}.
$$
For $\alpha\in R_I$, we have
$\ip{\alpha}{x_r}=\ip{\alpha}{x_0}$, and hence
$
 \Omega_I(x_r)=\Omega_I(x_0).
$
Consequently, for every $r\ge0$,
\begin{align}\label{omega r}
 1\le \frac{\Omega_\kappa(x_r)^{1/2}}
 {W(x_r)\Omega_I(x_0)^{1/2}} &=
 \prod_{\alpha\in R_+\setminus R_I}
 \left(1+\frac{1}{\ip{\alpha}{x_r}}\right)^{\kappa(\alpha)}
 \le
 (1+b_0^{-1})^{\sum_{\alpha\in R_+}\kappa(\alpha)}
 =:C_w.
\end{align}
Here we used \eqref{eq:ray-geometry} and $a+r\ge1$.

Then consider the compact set $\mathfrak a^0$.
The Dunkl kernel and its first derivatives are continuous jointly
in the two variables.
The union of the sets $\mathfrak a^0$ over all the finitely many
orderings is compact, and so is
$\{\zeta\in\bbR^N:\norm{\zeta}\le M_{\mathcal A}\}$.
It follows that there is a constant $B_{\mathcal A}\ge1$ such that
\begin{align}\label{eq:bound a0}
 \max_{g\in G}
 \Omega_\kappa(u)^{1/2}
 \abs{E_\kappa(ig\lambda,u)}
 \le B_{\mathcal A},
 \qquad u\in\mathfrak a^0.
\end{align}
Using the Cauchy--Schwarz inequality, we can also choose $B_{\mathcal A}$ large enough such that
\begin{align}\label{eq:partial bound a0}
 \max_{g\in G}
 \Omega_\kappa(u)^{1/2}
 \abs{\partial_vE_\kappa(ig\lambda,u)}
 \le
  B_{\mathcal A}\norm{v},
 \qquad u\in\mathfrak a^0,\quad v\in\bbR^N.
\end{align}

To prove \eqref{claim for v=0}, we put $C_{N+1}=0$ and proceed downwards in $i$.

For a fixed $x\in\mathfrak a^i$, use the construction of $x_0$ above and denote $x=x_t$.
If $a=1$, then $x_0\in\mathfrak a^0$, and we can applie \eqref{eq:bound a0} at $x_0$.
If $a>1$, then $x_0\in\mathfrak a^{i+1}$ and
$\rho_{i+1}(x_0)=a$, so the induction hypothesis applies there.
In both cases,
$$
 \max_{g\in G}
 \Omega_\kappa(x_0)^{1/2}|F_g(x_0)|
 \le
 (B_{\mathcal A}+C_{i+1})
 \frac{\mathfrak q_\lambda(a)}{\mathfrak q_\lambda(1)}.
$$
Therefore, using \eqref{omega r} we obtain that
\begin{align*}
 \norm{\Phi(0)}
 &=
 W(x_0)\left(\sum_{g\in G}|F_g(x_0)|^2\right)^{1/2}
 \le
 \sqrt{|G|}\,(B_{\mathcal A}+C_{i+1})
 \frac{\mathfrak q_\lambda(a)}{\mathfrak q_\lambda(1)}
 \Omega_I(x_0)^{-1/2}.
\end{align*}
Applying \eqref{eq:ray-evolution-Phi} cancels $\mathfrak q_\lambda(a)$:
\begin{align}\label{eq:Phi-q}
 \norm{\Phi(t)}
 \le
 C_e\sqrt{|G|}\,(B_{\mathcal A}+C_{i+1})
 \frac{\mathfrak q_\lambda(a+t)}{\mathfrak q_\lambda(1)}
 \Omega_I(x_0)^{-1/2}.
\end{align}
For each $g$, we have $|F_g(x_t)|=W(x_t)^{-1}|\Phi_g(x_t)|$.
The weight comparison \eqref{omega r} thus gives
$$
 \max_{g\in G}
 \Omega_\kappa(x_t)^{1/2}|F_g(x_t)|
 \le
 C_wC_e\sqrt{|G|}\,(B_{\mathcal A}+C_{i+1})
 \frac{\mathfrak q_\lambda(a+t)}{\mathfrak q_\lambda(1)}.
$$
Since $\rho_i(x_t)=a+t$, this proves the induction step and we obtain \eqref{claim for v=0}.

Now we start to prove \eqref{eq:joint} for $\nu=0$ in the spanning case.

For $u\ge1$, using $w_\kappa(\lambda)^{1/2}
=\prod_{\beta\in R_+}\varepsilon_\beta^{\kappa(\beta)}$ and
$\varepsilon_\beta\le\norm{\beta}M_{\mathcal A}$ we can obtain that
\begin{equation}\label{eq:product-compensation}
 \begin{aligned}
 w_\kappa(\lambda)^{1/2}
 \frac{\mathfrak q_\lambda(u)}{\mathfrak q_\lambda(1)}
 &=
 \prod_{\beta\in R_+}
 \left(
  (1+\varepsilon_\beta)
  \frac{\varepsilon_\beta u}{1+\varepsilon_\beta u}
 \right)^{\kappa(\beta)}
 \le
 \prod_{\beta\in R_+}
 (1+\norm{\beta}M_{\mathcal A})^{\kappa(\beta)}
 =:B'_{\mathcal A}.
 \end{aligned}
\end{equation}
It shows that the spectral weight $w_\kappa(\lambda)^{1/2}$ cancels the possible growth of
$\mathfrak q_\lambda(u)/\mathfrak q_\lambda(1)$. 

For $x\in\mathfrak a^1$, we have $\rho_1(x)\ge 1$, which allows us to combine \eqref{eq:product-compensation}
with \eqref{claim for v=0} for $i=1$.
For $x\in\mathfrak a^0$, we use \eqref{eq:bound a0}
 and
$w_\kappa(\lambda)^{1/2}\le B'_{\mathcal A}$.
Since $S=\mathfrak a^0\cup\mathfrak a^1$, this proves
$$
 \sup_{x\in S}\max_{g\in G}
 w_\kappa(\lambda)^{1/2}\Omega_\kappa(x)^{1/2}
 \abs{E_\kappa(ig\lambda,x)}
 \le C_{\mathcal A}.
$$
The same constant $C_{\mathcal A}$ works for every ordering of the simple roots. 
The estimate therefore holds on $\overline{\mathcal C_+}$.
Every point of $\bbR^N$ has the form $hx_+$ for some $h\in G$
and $x_+\in\overline{\mathcal C_+}$.
Covariance and the $G$-invariance of $\Omega_\kappa$ give
$$
 E_\kappa(i\lambda,hx_+)=E_\kappa(ih^{-1}\lambda,x_+),
 \qquad
 \Omega_\kappa(hx_+)=\Omega_\kappa(x_+).
$$
Taking $g=h^{-1}$ in the estimate on $\overline{\mathcal C_+}$ gives the required bound at $hx_+$. 

\vskip0.5cm
{\bf Step 1.4.}

We next prove the estimate for first derivatives.
Recall that
\begin{align*}
 \partial_{H_i}F_g(u)
 &=
 i\ip{H_i}{g\lambda}F_g(u)
 -
 \sum_{\alpha\in R_+\setminus R_I}
 c_\alpha(u)\bigl(F_g(u)-F_{\sigma_\alpha g}(u)\bigr),
 \qquad
 c_\alpha(u)
 =
 \kappa(\alpha)\frac{\ip{\alpha}{H_i}}{\ip{\alpha}{u}}.
\end{align*}

Fix a real vector $v\in\bbR^N$ and retain the notation for a ray
$x_r=x_0+rH_i$.
Differentiate the equation for $F_g$ in the direction $v$ on
$\mathfrak a_{\mathrm{reg}}^I$.
All its terms are smooth there, and
\begin{align*}
 \partial_{H_i}\partial_vF_g(u)
 &=
 i\ip{H_i}{g\lambda}\partial_vF_g(u)-
 \sum_{\alpha\in R_+\setminus R_I}
 c_\alpha(u)
 \bigl(\partial_vF_g(u)-\partial_vF_{\sigma_\alpha g}(u)\bigr)\\
 &\quad-
 \sum_{\alpha\in R_+\setminus R_I}
 \partial_vc_\alpha(u)
 \bigl(F_g(u)-F_{\sigma_\alpha g}(u)\bigr).
\end{align*}
The coefficient $i\ip{H_i}{g\lambda}$ is constant in $u$,
so differentiating it gives no additional term.

Define a vector of weighted derivatives by
$$
 (\Theta_v(r))_g
 =
 W(x_r)e^{-i\ip{x_r}{g\lambda}}\partial_vF_g(x_r),
 \qquad g\in G.
$$
Differentiating this expression in $r$, the derivatives of $W$
and of the exponential cancel the same two terms as before.
Writing $\Phi_g(r)$ for $\Phi_g(x_r)$, the resulting equation is
\begin{align*}
 (\Theta_v'(r))_g
 &=
 \sum_{\alpha\in R_+\setminus R_I}
 c_\alpha(x_r)
 e^{-i\ip{\alpha}{x_r}\ip{\alpha}{g\lambda}}
 (\Theta_v(r))_{\sigma_\alpha g}-
 \sum_{\alpha\in R_+\setminus R_I}
 \partial_vc_\alpha(x_r)
 \left(
  \Phi_g(r)
  -
  e^{-i\ip{\alpha}{x_r}\ip{\alpha}{g\lambda}}
  \Phi_{\sigma_\alpha g}(r)
 \right).
\end{align*}
We further write
\begin{align}\label{expression for Theta'}
 \Theta_v'(r)=\mathsf A(r)\Theta_v(r)+\mathsf R_v(r)\Phi(r),
 \qquad
 \mathsf R_v(r)
 =
 -\sum_{\alpha\in R_+\setminus R_I}
 \partial_vc_\alpha(x_r)
 \bigl(\mathrm{Id}-\mathsf P_\alpha(r)\bigr),
\end{align}
where $\mathsf P_\alpha(r)$ is the operator on $\ell^2(G)$ defined by
$$
 (\mathsf P_\alpha(r)z)_g
 =
 e^{-i\ip{\alpha}{x_r}\ip{\alpha}{g\lambda}}
 z_{\sigma_\alpha g}.
$$
The exponential has modulus one and $g\mapsto\sigma_\alpha g$
permutes $G$.  Therefore,
$$
 \norm{\mathsf P_\alpha(r)z}^2
 =
 \sum_{g\in G}|z_{\sigma_\alpha g}|^2
 =
 \norm{z}^2,
 \qquad
 \norm{\mathrm{Id}-\mathsf P_\alpha(r)}\le2.
$$
For
$$
\partial_vc_\alpha(u)
 =
 -\kappa(\alpha)
 \frac{\ip{\alpha}{H_i}\ip{\alpha}{v}}
      {\ip{\alpha}{u}^2},
$$
using $\ip{\alpha}{x_r}=b_\alpha(r+d_\alpha)$, we have
$$
 |\partial_vc_\alpha(x_r)|
 =
 \kappa(\alpha)
 \frac{|\ip{\alpha}{v}|}{b_\alpha(r+d_\alpha)^2}
 \le
 \frac{\kappa(\alpha)\norm{\alpha}}{b_0}
 \frac{\norm{v}}{(a+r)^2}.
$$
Consequently,
\begin{equation}\label{eq:jet-remainder}
 \norm{\mathsf R_v(r)}
 \le
 C_R\frac{\norm{v}}{(a+r)^2},
 \qquad
 C_R=\frac{2}{b_0}
 \sum_{\alpha\in R_+}\kappa(\alpha)\norm{\alpha}.
\end{equation}
This constant is independent of $\lambda$.
In particular, no inverse power of any $\varepsilon_\beta$ occurs
when we differentiate the kernel.

We claim that there are constants $D_i$ such that
\begin{align}\label{claim for v=1}
 \max_{g\in G}
 \Omega_\kappa(x)^{1/2}
 \abs{\partial_vE_\kappa(ig\lambda,x)}
 \le
 D_i\norm{v}
 \frac{\mathfrak q_\lambda(\rho_i(x))}
      {\mathfrak q_\lambda(1)},
 \qquad x\in\mathfrak a^i,\quad 1\le i\le N.
\end{align}
Similar to the proof of \eqref{claim for v=0}, we again put $D_{N+1}=0$ and argue from $i=N$ down to $i=1$.

Using the weight comparison \eqref{omega r}, at the initial point $x_0$ of the ray, either \eqref{eq:partial bound a0} or the induction hypothesis gives
$$
 \norm{\Theta_v(0)}
 \le
 \sqrt{|G|}\,(B_{\mathcal A}+D_{i+1})\norm{v}
 \frac{\mathfrak q_\lambda(a)}{\mathfrak q_\lambda(1)}
 \Omega_I(x_0)^{-1/2}.
$$
Also, \eqref{eq:Phi-q} gives
$$
 \norm{\Phi(r)}
 \le
 \sqrt{|G|}\,C_i
 \frac{\mathfrak q_\lambda(a+r)}{\mathfrak q_\lambda(1)}
 \Omega_I(x_0)^{-1/2}.
$$

By \eqref{expression for Theta'}, the variation-of-constants formula for the equation for $\Theta_v$ is
$$
 \Theta_v(t)
 =
 \mathscr U_{\rm ray}(t,0)\Theta_v(0)
 +
 \int_0^t
 \mathscr U_{\rm ray}(t,r)\mathsf R_v(r)\Phi(r)\,dr.
$$
Combining \eqref{eq:ray-evolution} and \eqref{eq:jet-remainder}
with the last two bounds yields
\begin{align*}
 \norm{\Theta_v(t)}
 &\le
 C_e\sqrt{|G|}\,(B_{\mathcal A}+D_{i+1})\norm{v}
 \frac{\mathfrak q_\lambda(a+t)}{\mathfrak q_\lambda(1)}
 \Omega_I(x_0)^{-1/2}\\
 &\quad+
 C_eC_R\sqrt{|G|}\,C_i\norm{v}\,
 \Omega_I(x_0)^{-1/2}
 \int_0^t
 \frac{\mathfrak q_\lambda(a+t)}{\mathfrak q_\lambda(a+r)}
 \frac{1}{(a+r)^2}
 \frac{\mathfrak q_\lambda(a+r)}{\mathfrak q_\lambda(1)}
 \,dr\\
 &\le
 C_e\sqrt{|G|}
 (B_{\mathcal A}+D_{i+1}+C_RC_i)\norm{v}
 \frac{\mathfrak q_\lambda(a+t)}{\mathfrak q_\lambda(1)}
 \Omega_I(x_0)^{-1/2}.
\end{align*}
In the last step, the factors $\mathfrak q_\lambda(a+r)$ cancel,
and $a\ge1$ gives
$$
 \int_0^t\frac{dr}{(a+r)^2}
 =
 \frac1a-\frac1{a+t}
 \le1.
$$

For each $g$, we have $|\partial_vF_g(x_t)|=W(x_t)^{-1}|(\Theta_v(x_t))_g|$.
Thus, multiplying by $\Omega_\kappa(x_t)^{1/2}W(x_t)^{-1}$
and using the weight comparison \eqref{omega r} again gives
\begin{align*}
 \max_{g\in G}
 \Omega_\kappa(x_t)^{1/2}|\partial_vF_g(x_t)|
 &\le
 C_wC_e\sqrt{|G|}
 (B_{\mathcal A}+D_{i+1}+C_RC_i)\norm{v}
 \frac{\mathfrak q_\lambda(a+t)}{\mathfrak q_\lambda(1)}.
\end{align*}
Since $\rho_i(x_t)=a+t$, this proves the induction step and we obtain \eqref{claim for v=1} with constants depending only on $R$, $\kappa$, and $M_{\mathcal A}$.

Apply \eqref{eq:product-compensation} on $\mathfrak a^1$ and
\eqref{eq:partial bound a0} on $\mathfrak a^0$.
Taking the maximum over all the sectors gives
$$
 w_\kappa(\lambda)^{1/2}\Omega_\kappa(x)^{1/2}
 \abs{\partial_vE_\kappa(ig\lambda,x)}
 \le C_{\mathcal A}\norm{v},
 \qquad
 x\in\overline{\mathcal C_+},\quad g\in G.
$$
For $x=hx_+$ with $h\in G$ and
$x_+\in\overline{\mathcal C_+}$, differentiate the covariance identity:
$$
 \left.\partial_vE_\kappa(i\lambda,u)\right|_{u=hx_+}
 =
 \left.\partial_{h^{-1}v}E_\kappa(ih^{-1}\lambda,u)
 \right|_{u=x_+}.
$$
The transformation $h$ is orthogonal, so
$\norm{h^{-1}v}=\norm{v}$.
The preceding bound therefore holds for every $x\in\bbR^N$.
Taking $v$ to be each coordinate vector proves \eqref{eq:joint}
for $|\nu|=1$ when $R$ spans $\bbR^N$.

\vskip0.8cm
{\bf Step 2.}

It remains to remove the spanning assumption.
If $R=\varnothing$, then $w_\kappa=\Omega_\kappa=1$ and
$E_\kappa(i\lambda,x)=e^{i\ip{\lambda}{x}}$.
The desired bounds follow immediately from
$|e^{i\ip{\lambda}{x}}|=1$ and
$|\ip{\lambda}{v}|\le M_{\mathcal A}\norm{v}$.
Suppose now that $R\ne\varnothing$, and let
$$
 V=\operatorname{span}R.
$$
Write the orthogonal decompositions
$$
 x=x_V+x_\perp,\qquad
 \lambda=\lambda_V+\lambda_\perp,\qquad
 v=v_V+v_\perp,
$$
where $x_V,\lambda_V,v_V\in V$,
and $x_\perp,\lambda_\perp,v_\perp\in V^\perp$.
For every $\alpha\in R$,
$$
 \ip{\alpha}{\lambda_V}=\ip{\alpha}{\lambda}\ne0,
 \qquad
 \norm{\lambda_V}\le M_{\mathcal A}.
$$
Thus $\lambda_V$ is regular for the root system in $V$.
Its norm may tend to zero as $\lambda$ varies in $\mathcal A$.
This causes no difficulty: the spanning proof used only
the upper bound on the spectral norm.
The compact-set estimates were taken over the whole spectral ball,
and \eqref{eq:product-compensation} used only
$\varepsilon_\beta\le\norm{\beta}M_{\mathcal A}$.
We may therefore apply that proof in $V$ uniformly to all the
projected parameters $\lambda_V$ under consideration.

Let $E_\kappa^V$ denote the Dunkl kernel for the root system in $V$.
The factorization in \cite[Remark~2.6]{Langen2026} is
\begin{align}\label{Ek}
 E_\kappa(i\lambda,x)
 =
 e^{i\ip{\lambda_\perp}{x_\perp}}
 E_\kappa^V(i\lambda_V,x_V).
\end{align}
Indeed, every reflection fixes $V^\perp$ pointwise, so the Dunkl
operators in directions in $V^\perp$ are ordinary derivatives.
The right-hand side of \eqref{Ek} satisfies the Dunkl eigenfunction equations
in both orthogonal subspaces and equals one at the origin.
The uniqueness of the Dunkl kernel proves the formula.
Differentiating it gives
\begin{align*}
 \partial_vE_\kappa(i\lambda,x)
 =
 e^{i\ip{\lambda_\perp}{x_\perp}}
 \left(
  \partial_{v_V}E_\kappa^V(i\lambda_V,x_V)
  +
  i\ip{\lambda_\perp}{v_\perp}E_\kappa^V(i\lambda_V,x_V)
 \right).
\end{align*}
The weights in the two spaces agree:
$$
 w_\kappa(\lambda)=w_\kappa^V(\lambda_V),
 \qquad
 \Omega_\kappa(x)=\Omega_\kappa^V(x_V),
$$
because every root belongs to $V$.
Moreover, 
$
|\ip{\lambda_\perp}{v_\perp}|\le \norm{\lambda_\perp}\norm{v_\perp}\le M_{\mathcal A}\norm{v_\perp}.
$
Thus, applying \eqref{eq:joint} in $V$ gives
$$
 w_\kappa(\lambda)^{1/2}\Omega_\kappa(x)^{1/2}
 \abs{\partial_vE_\kappa(i\lambda,x)}
 \le
 C_{\mathcal A}
 \bigl(\norm{v_V}+M_{\mathcal A}\norm{v_\perp}\bigr)
 \le
 C_{\mathcal A}(1+M_{\mathcal A})\norm{v}.
$$
Taking $v$ to be the coordinate vectors completes the proof of
\eqref{eq:joint}.
\end{proof}

\section{From the joint kernel estimate to the multiplier theorem}\label{sec:4}

We use the chamber-lifting argument from \cite{Chamber2026}.  The joint estimate
\eqref{eq:packet-joint} gives the scalar kernel bounds on a fixed annulus, and the argument for finite dyadic sums below
yields weak type $(1,1)$ and the chamber-pullback Hardy and BMO endpoints.  Strong $L^p$ bounds then
follow by interpolation and duality.

\subsection{Kernel estimates on a fixed annulus}

Choose symmetric, $G$-invariant compact annuli $\mathcal A_0$, $\mathcal A_1$, and $\mathcal A_2$ such that
$
 \supp\bigl(\psi(\norm{\cdot})\bigr)\subset\operatorname{int}\mathcal A_0,
 \,
 \mathcal A_0\subset\operatorname{int}\mathcal A_1,
 \,
 \mathcal A_1\subset\operatorname{int}\mathcal A_2.
$

We apply Proposition~\ref{thm:joint} to the annuli $\mathcal A_2$. Every reflected amplitude and the support cutoff used below
are supported in $\mathcal A_2$.  Since $\mathcal A_2=-\mathcal A_2$, the same annulus is used for both signs of the spectral
parameter.  Symmetry and covariance give
$$
 E_\kappa(-igy,\xi)=E_\kappa(-i\xi,gy),
 \qquad
 \partial_v^yE_\kappa(-igy,\xi)
 =
 \partial_{gv}^{(2)}E_\kappa(-i\xi,gy).
$$
Here $\partial^{(2)}$ acts in the second variable.  Since $\mathcal A_2=-\mathcal A_2$,
$\Omega_\kappa(gy)=\Omega_\kappa(y)$, and
$$
 \partial_{gv}^{(2)}
 =
 \sum_{k=1}^N(gv)_k\partial_k^{(2)},
 \qquad
 \sum_{k=1}^N\abs{(gv)_k}\le\sqrt N\norm{v},
$$
\eqref{eq:packet-joint} gives, for $\supp u\subset \mathcal A_2$,
\begin{align}\label{eq:packet-size-bound}
 \Omega_\kappa(y)^{1/2}
 \norm{u(\xi)E_\kappa(-igy,\xi)}_{L^2_\xi(d\omega)}
 &=
 \Omega_\kappa(gy)^{1/2}
 \norm{u(-\eta)E_\kappa(i\eta,gy)}_{L^2_\eta(d\omega)}
 \le
 C\norm{u}_{L^2(d\xi)}.
\end{align}
Moreover,
\begin{align}\label{eq:packet-jet-bound}
& \Omega_\kappa(y)^{1/2}
 \norm{u(\xi)\partial_v^yE_\kappa(-igy,\xi)}_{L^2_\xi(d\omega)} \notag\\
 \le\,&
 \sum_{k=1}^N\abs{(gv)_k}
 \Omega_\kappa(gy)^{1/2}
 \norm{u(-\eta)\partial_k^{(2)}E_\kappa(i\eta,gy)}_{L^2_\eta(d\omega)}
 \le
 C\norm{v}\norm{u}_{L^2(d\xi)}.
\end{align}
The constants are uniform in $g$, $y$, and $v$.

For $g\in G$, set
$
 h_g(\xi,y)=E_\kappa(-igy,\xi)
$
and, for $a\in C_c^\infty(\bbR^N)$,
$$
 \calK_g[a](x,y)
 =
 (c_\kappa^{\mathrm M})^{-1}
 \calF^{-1}\bigl(a(\xi)h_g(\xi,y)\bigr)(x)
 =
 (c_\kappa^{\mathrm M})^{-2}
 \int_{\bbR^N}
 a(\xi)E_\kappa(-igy,\xi)E_\kappa(i\xi,x)\,d\omega(\xi).
$$

\begin{lemma}\label{lem:covariant-moment}
Let $g\in G$ and let $a\in C_c^\infty(\bbR^N)$ satisfy $\supp a\subset \mathcal A_2$.  Then, for
$j=1,\ldots,N$,
\begin{equation}\label{eq:moment}
 \bigl(x_j-(gy)_j\bigr)\calK_g[a](x,y)
 =
 i\calK_g[\partial_ja](x,y)
 +
 i\sum_{\alpha\in R_+}
 \kappa(\alpha)\alpha_j
 \calK_{\sigma_\alpha g}[\delta_\alpha a](x,y).
\end{equation}
\end{lemma}

\begin{proof}
The eigenfunction identity and covariance give
$$
 T_j^\xi h_g=-i(gy)_jh_g,
 \qquad
 h_g(\sigma_\alpha\xi,y)
 =
 E_\kappa(-i\sigma_\alpha gy,\xi)
 =
 h_{\sigma_\alpha g}(\xi,y).
$$
By \eqref{eq:delta-Leibniz},
\begin{align*}
 T_j^\xi(ah_g)
 &=
 (\partial_ja)h_g+a\partial_jh_g
 +
 \sum_{\alpha\in R_+}
 \kappa(\alpha)\alpha_j
 \bigl(a\delta_\alpha h_g
 +\delta_\alpha a\,(h_g\circ\sigma_\alpha)\bigr)\\
 &=
 aT_j^\xi h_g+(\partial_ja)h_g
 +
 \sum_{\alpha\in R_+}
 \kappa(\alpha)\alpha_j
 \delta_\alpha a\,(h_g\circ\sigma_\alpha)\\
 &=
 -i(gy)_jah_g+(\partial_ja)h_g
 +
 \sum_{\alpha\in R_+}
 \kappa(\alpha)\alpha_j
 \delta_\alpha a\,(h_g\circ\sigma_\alpha).
\end{align*}
Consequently,
\begin{align*}
 x_j\calK_g[a]
 &=
 (c_\kappa^{\mathrm M})^{-1}
 \calF^{-1}\bigl(iT_j^\xi(ah_g)\bigr)
 =
 i\calK_g[\partial_ja]
 +(gy)_j\calK_g[a]
 +i\sum_{\alpha\in R_+}
 \kappa(\alpha)\alpha_j
 \calK_{\sigma_\alpha g}[\delta_\alpha a].
\end{align*}
Rearranging proves \eqref{eq:moment}.  
\end{proof}

\begin{lemma}\label{lem:packet-composition-expansion}
For every integer $M\ge0$, there exist a finite family $\mathscr D_M$ of compositions as defined in \eqref{def:composition} of order at
most $M$ and a constant $C_M$ such that, for $g\in G$, $x,y\in\bbR^N$, and
$a\in C_c^\infty(\bbR^N)$ with $\supp a\subset \mathcal A_2$,
\begin{equation}\label{eq:packet-composition-expansion}
 d(x,y)^M\abs{\calK_g[a](x,y)}
 \le
 C_M
 \sum_{\substack{h\in G\\\mathcal D\in\mathscr D_M}}
 \abs{\calK_h[\mathcal D a](x,y)},
\end{equation}
where $C_M$ and $\mathscr D_M$ depend only on $M$, $R$, and $\kappa$.
\end{lemma}

\begin{proof}
Set
$
 \mathscr D_0=\{I\}
 $
 and 
 $
 \mathscr D_{m+1}
 =
 \mathscr D_m\cup
 \bigl\{\partial_j\mathcal D,\delta_\alpha\mathcal D:
 \mathcal D\in\mathscr D_m,\ 1\le j\le N,\ \alpha\in R_+\bigr\}.
$
Every composition in $\mathscr D_m$ has order at most $m$, and Lemma~\ref{lem:delta} and
Proposition~\ref{prop:compositions} give
$$
 \mathcal D a\in C_c^\infty(\bbR^N),
 \qquad
 \supp(\mathcal D a)\subset \mathcal A_2,
 \qquad
 \mathcal D\in\mathscr D_m.
$$
For $h\in G$,
$$
 d(x,y)=d(x,hy)\le\norm{x-hy}
 \le
 \sum_{j=1}^N\abs{x_j-(hy)_j}.
$$
The case $M=0$ follows from $\mathscr D_0=\{I\}$.  Assume \eqref{eq:packet-composition-expansion} at order
$m$.  For $\mathcal D\in\mathscr D_m$, Lemma~\ref{lem:covariant-moment} gives
\begin{align*}
 d(x,y)\abs{\calK_h[\mathcal D a](x,y)}
 &\le
 \sum_{j=1}^N
 \abs{\bigl(x_j-(hy)_j\bigr)\calK_h[\mathcal D a](x,y)}\\
 &\le
 \sum_{j=1}^N
 \abs{\calK_h[\partial_j\mathcal D a](x,y)}
+
 \sum_{j=1}^N\sum_{\alpha\in R_+}
 \kappa(\alpha)\abs{\alpha_j}
 \abs{\calK_{\sigma_\alpha h}[\delta_\alpha\mathcal D a](x,y)}.
\end{align*}
Multiplying the order-$m$ estimate by $d(x,y)$ and using that
$h\mapsto\sigma_\alpha h$ permutes $G$, we obtain
\begin{align*}
 d(x,y)^{m+1}\abs{\calK_g[a](x,y)}
 &\le
 C_m
 \sum_{\substack{h\in G\\\mathcal D\in\mathscr D_m}}
 d(x,y)\abs{\calK_h[\mathcal D a](x,y)}
 \le
 C_{m+1}
 \sum_{\substack{h\in G\\\mathfrak v\in\mathscr D_{m+1}}}
 \abs{\calK_h[\mathfrak v a](x,y)}.
\end{align*}
This proves the induction step.
\end{proof}

Let $v\in\bbR^N$ and set
$$
 h_{g,v}^{(1)}(\xi,y)=\partial_v^y h_g(\xi,y),
 \qquad
 \calK_{g,v}^{(1)}[a](x,y)
 =
 (c_\kappa^{\mathrm M})^{-1}
 \calF^{-1}\bigl(a(\xi)h_{g,v}^{(1)}(\xi,y)\bigr)(x).
$$
Differentiation in the direction $v$ gives
\begin{align*}
 T_j^\xi h_{g,v}^{(1)}
 &=
 \partial_v^y\bigl(-i(gy)_jh_g\bigr)
 =
 -i(gy)_jh_{g,v}^{(1)}-i(gv)_jh_g,
 \end{align*}
 and
 \begin{align*}
 h_{g,v}^{(1)}(\sigma_\alpha\xi,y)
 &=
 \partial_v^y h_{\sigma_\alpha g}(\xi,y)
 =
 h_{\sigma_\alpha g,v}^{(1)}(\xi,y).
\end{align*}
By \eqref{eq:delta-Leibniz},
\begin{align*}
 T_j^\xi(ah_{g,v}^{(1)})
 &=
 aT_j^\xi h_{g,v}^{(1)}
 +(\partial_ja)h_{g,v}^{(1)}
 +\sum_{\alpha\in R_+}
 \kappa(\alpha)\alpha_j
 \delta_\alpha a\,
 (h_{g,v}^{(1)}\circ\sigma_\alpha)\\
 &=
 -i(gy)_jah_{g,v}^{(1)}
 -i(gv)_jah_g
 +(\partial_ja)h_{g,v}^{(1)}
 +\sum_{\alpha\in R_+}
 \kappa(\alpha)\alpha_j
 \delta_\alpha a\,
 (h_{g,v}^{(1)}\circ\sigma_\alpha).
\end{align*}
Applying the inverse-coordinate identity and rearranging yields
\begin{align}\label{eq:jet-moment}
 \bigl(x_j-(gy)_j\bigr)\calK_{g,v}^{(1)}[a](x,y)
 =\,
 i\calK_{g,v}^{(1)}[\partial_ja](x,y)
 +
 i\sum_{\alpha\in R_+}
 \kappa(\alpha)\alpha_j
 \calK_{\sigma_\alpha g,v}^{(1)}[\delta_\alpha a](x,y)
 +(gv)_j\calK_g[a](x,y). 
\end{align}

\begin{proposition}\label{prop:unit}
Let $s\ge0$ and let $a\in C_c^\infty(\bbR^N)$ satisfy $\supp a\subset \mathcal A_1$.  Then
\begin{equation}\label{eq:unit-size}
 \Omega_\kappa(y)^{1/2}
 \norm{(1+d(\cdot,y))^s\calK_g[a](\cdot,y)}_{L^2(d\omega)}
 \le C_s\norm{a}_{H^s}.
\end{equation}
Moreover,
\begin{equation}\label{eq:unit-jet}
 \Omega_\kappa(y)^{1/2}
 \norm{(1+d(\cdot,y))^s\calK_{g,v}^{(1)}[a](\cdot,y)}_{L^2(d\omega)}
 \le C_s\norm{v}\norm{a}_{H^s}.
\end{equation}
The constants are uniform in $g$, $y$, and $v$.
\end{proposition}

\begin{proof}
First let $M\ge0$ be an integer.  Plancherel and
\eqref{eq:packet-size-bound}--\eqref{eq:packet-jet-bound} give, for $\supp u\subset \mathcal A_2$,
\begin{align*}
 \Omega_\kappa(y)^{1/2}
 \norm{\calK_h[u](\cdot,y)}_{L^2(d\omega)}
 &=
 (c_\kappa^{\mathrm M})^{-1}
 \Omega_\kappa(y)^{1/2}
 \norm{u(\xi)E_\kappa(-ihy,\xi)}_{L^2_\xi(d\omega)}
 \le
 C\norm{u}_{L^2(d\xi)},
 \end{align*}
 and
 \begin{align*}
 \Omega_\kappa(y)^{1/2}
 \norm{\calK_{h,v}^{(1)}[u](\cdot,y)}_{L^2(d\omega)}
 &\le
 C\norm{v}\norm{u}_{L^2(d\xi)}.
\end{align*}
Lemma~\ref{lem:packet-composition-expansion} and Proposition~\ref{prop:compositions} now yield
\begin{align*}
\Omega_\kappa(y)^{1/2}
 \norm{d(\cdot,y)^M\calK_g[a](\cdot,y)}_{L^2(d\omega)}
 &\le
 C_M
 \sum_{\substack{h\in G\\\mathcal D\in\mathscr D_M}}
 \Omega_\kappa(y)^{1/2}
 \norm{(\mathcal D a)(\xi)E_\kappa(-ihy,\xi)}_{L^2_\xi(d\omega)}\\
 &\le
 C_M
 \sum_{\mathcal D\in\mathscr D_M}
 \norm{\mathcal D a}_{L^2(d\xi)}
 \le
 C_M\norm{a}_{H^M}.
\end{align*}
The same estimate with $0\le q\le M$ gives
$$
 \Omega_\kappa(y)^{1/2}
 \norm{d(\cdot,y)^q\calK_g[a](\cdot,y)}_{L^2(d\omega)}
 \le
 C_M\norm{a}_{H^M}.
$$
Hence
\begin{align*}
 \Omega_\kappa(y)^{1/2}
 \norm{(1+d(\cdot,y))^M\calK_g[a](\cdot,y)}_{L^2(d\omega)}
 &\le
 \sum_{q=0}^M\binom Mq
 \Omega_\kappa(y)^{1/2}
 \norm{d(\cdot,y)^q\calK_g[a](\cdot,y)}_{L^2(d\omega)}
 \le
 C_M\norm{a}_{H^M}.
\end{align*}

For $M=1$, the coordinate bound and \eqref{eq:jet-moment} give
\begin{align*}
 d(x,y)\abs{\calK_{g,v}^{(1)}[a](x,y)}
 &\le
 C
 \sum_{\substack{h\in G\\\mathcal D\in\mathscr D_1}}
 \abs{\calK_{h,v}^{(1)}[\mathcal D a](x,y)}
 +
 C\norm{v}
 \sum_{h\in G}
 \abs{\calK_h[a](x,y)}.
\end{align*}
Assume the same estimate at order $m\ge1$.  Apply \eqref{eq:jet-moment} to each kernel $\calK_{h, v}^{(1)}[\mathcal D a]$ and \eqref{eq:moment} to each kernel $\calK_h[\mathcal D a]$.
Since $h\mapsto\sigma_\alpha h$ permutes $G$,
\begin{align*}
 &d(x,y)^{m+1}\abs{\calK_{g,v}^{(1)}[a](x,y)}\\
 \le\,&
 C_m
 \sum_{\substack{h\in G\\\mathcal D\in\mathscr D_m}}
 d(x,y)\abs{\calK_{h,v}^{(1)}[\mathcal D a](x,y)}  +
 C_m\norm{v}
 \sum_{\substack{h\in G\\\mathcal D\in\mathscr D_{m-1}}}
 d(x,y)\abs{\calK_h[\mathcal D a](x,y)}
 \\
 \le\,&
 C_{m+1}
 \sum_{\substack{h\in G\\\mathcal D\in\mathscr D_{m+1}}}
 \abs{\calK_{h,v}^{(1)}[\mathcal D a](x,y)}
+
 C_{m+1}\norm{v}
 \sum_{\substack{h\in G\\\mathcal D\in\mathscr D_m}}
 \abs{\calK_h[\mathcal D a](x,y)}.
\end{align*}
Thus, for every $M\ge1$,
\begin{equation}\label{eq:jet-composition-expansion}
\begin{split}
 d(x,y)^M\abs{\calK_{g,v}^{(1)}[a](x,y)}
 &\le
 C_M
 \sum_{\substack{h\in G\\\mathcal D\in\mathscr D_M}}
 \abs{\calK_{h,v}^{(1)}[\mathcal D a](x,y)}
 +
 C_M\norm{v}
 \sum_{\substack{h\in G\\\mathcal D\in\mathscr D_{M-1}}}
 \abs{\calK_h[\mathcal D a](x,y)}.
\end{split}
\end{equation}
The preceding two kernel bounds and Proposition~\ref{prop:compositions} give
\begin{align*}
 \Omega_\kappa(y)^{1/2}
 \norm{d(\cdot,y)^M\calK_{g,v}^{(1)}[a](\cdot,y)}_{L^2(d\omega)}
 &\le
 C_M\norm{v}
 \left(
 \sum_{\mathcal D\in\mathscr D_M}\norm{\mathcal D a}_2
 +
 \sum_{\mathcal D\in\mathscr D_{M-1}}\norm{\mathcal D a}_2
 \right)
 \le
 C_M\norm{v}\norm{a}_{H^M}.
\end{align*}
For $M=0$ this follows directly from the second kernel bound.  Repeating the estimate for
$0\le q\le M$ and expanding $(1+d)^M$ proves \eqref{eq:unit-jet} at every integer order.

Let $s=M+\theta$, where $M\ge0$ is an integer and $0<\theta<1$.  Choose
$\chi\in C_c^\infty(\operatorname{int}\mathcal A_2)$ with $\chi=1$ on a neighbourhood of $\mathcal A_1$, and, for
$v\ne0$, define
$$
 \mathsf T_{g,y}a
 =
 \Omega_\kappa(y)^{1/2}\calK_g[\chi a](\cdot,y),
 \qquad
 \mathsf T_{g,y,v}^{(1)}a
 =
 \frac{\Omega_\kappa(y)^{1/2}}{\norm{v}}
 \calK_{g,v}^{(1)}[\chi a](\cdot,y).
$$
Since $(\supp\chi)_G\subset\operatorname{int}\mathcal A_2$, the preceding integer argument applies to
$\chi a$.  Writing $\mathsf T$ for either operator, the integer estimates, density, and the Sobolev
multiplier bound for $\chi$ give, for $r=0,1$,
$$
 \norm{\mathsf T a}_{
 L^2((1+d(\cdot,y))^{2(M+r)}d\omega)}
 \le
 C_{M+1}\norm{\chi a}_{H^{M+r}}
 \le
 C_{M,\chi}\norm{a}_{H^{M+r}}.
$$
Complex interpolation, together with $[H^M,H^{M+1}]_\theta=H^{M+\theta}$ and
\begin{align*}
 \bigl[
 L^2((1+d(\cdot,y))^{2M}d\omega),
 L^2((1+d(\cdot,y))^{2M+2}d\omega)
 \bigr]_\theta
 &=
 L^2((1+d(\cdot,y))^{2(M+\theta)}d\omega),
\end{align*}
gives \eqref{eq:unit-size} and \eqref{eq:unit-jet} at order $s$.  Since $\chi a=a$ for
$\supp a\subset \mathcal A_1$, this is the required estimate.  The case $v=0$ is immediate.
\end{proof}

The preceding kernel estimates were proved for smooth amplitudes.  The next lemma extends them to the
Sobolev amplitudes arising in the dyadic decomposition.

\begin{lemma}\label{lem:sobolev-kernel-closure}
Let $s>N/2$ and let $b\in H^s(\bbR^N)$ have continuous representative supported in $\mathcal A_0$.  For
$g\in G$, define
\begin{equation}\label{eq:sobolev-packet-kernel}
 \calK_g[b](x,y)
 =
 (c_\kappa^{\mathrm M})^{-2}
 \int_{\bbR^N}
 b(\xi)E_\kappa(i\xi,x)E_\kappa(-igy,\xi)\,d\omega(\xi).
\end{equation}
The integral converges absolutely and defines a jointly continuous function.  Uniformly in $g$,
$$
 \Omega_\kappa(y)^{1/2}
 \norm{(1+d(\cdot,y))^s\calK_g[b](\cdot,y)}_{L^2(d\omega)}
 \le
 C_s\norm{b}_{H^s}.
$$
If $\norm{y-y'}\le1$, then
\begin{equation}\label{eq:unit-difference}
 \Omega_\kappa(y)^{1/2}
 \norm{
 (1+d(\cdot,y))^s
 \bigl(\calK_g[b](\cdot,y)-\calK_g[b](\cdot,y')\bigr)
 }_{L^2(d\omega)}
 \le
 C_s\norm{y-y'}\norm{b}_{H^s}.
\end{equation}
The corresponding estimates in the first variable are
$$
 \Omega_\kappa(x)^{1/2}
 \norm{(1+d(x,\cdot))^s\calK_g[b](x,\cdot)}_{L^2(d\omega)}
 \le
 C_s\norm{b}_{H^s}
$$
and, for $\norm{x-x'}\le1$,
$$
 \Omega_\kappa(x)^{1/2}
 \norm{
 (1+d(x,\cdot))^s
 \bigl(\calK_g[b](x,\cdot)-\calK_g[b](x',\cdot)\bigr)
 }_{L^2(d\omega)}
 \le
 C_s\norm{x-x'}\norm{b}_{H^s}.
$$
Moreover,
$$
 K_b(x,y):=\calK_e[b](x,y)
$$
is an associated kernel of $T_b$ on $L_c^\infty(\bbR^N)$.  These assertions are preserved under
$b\mapsto b\circ h$, $b\mapsto\overline b\circ h$, $h\in G$, and finite sums.
\end{lemma}

\begin{proof}
Choose $\varepsilon_0>0$ such that
$
 \mathcal A_0+\overline{B(0,\varepsilon_0)}
 \subset
 \operatorname{int}\mathcal A_1.
$
Set
$
 K=\mathcal A_0+\overline{B(0,\varepsilon_0)}.
$
Let $\rho$ be a standard mollifier supported in $B(0,1)$, choose
$0<\varepsilon_n<\varepsilon_0$ with $\varepsilon_n\to0$, and define
$
 b_n=b*\rho_{\varepsilon_n}.
$
Then
$
 b_n\in C_c^\infty(\operatorname{int}\mathcal A_1).
$
Moreover,
$
 \supp b_n\subset K,
$
and
$$
 \norm{b_n-b}_{H^s}\longrightarrow0.
$$
Since $b_n-b$ is supported in $K$ and $w_\kappa$ is bounded there,
\begin{align*}
 \norm{b_n-b}_{L^1(d\omega)}
 &=
 \int_K\abs{b_n-b}w_\kappa\,d\xi
 \le
 \norm{w_\kappa}_{L^\infty(K)}
 \abs{K}^{1/2}\norm{b_n-b}_{L^2(d\xi)}
 \longrightarrow0.
\end{align*}
Hence $b\in L^1(d\omega)$.  No lower bound for $w_\kappa$ is used.

The positivity of the Dunkl intertwining operator
\cite[Corollary~5.4]{Rosler1999} gives
$$
 \abs{E_\kappa(iu,\xi)}\le1,
 \qquad
 u,\xi\in\bbR^N.
$$
Consequently,
\begin{align*}
 \int_{\bbR^N}
 \abs{b(\xi)E_\kappa(i\xi,x)E_\kappa(-igy,\xi)}\,d\omega(\xi)
 &\le
 \int_{\bbR^N}\abs{b(\xi)}\,d\omega(\xi)
 =
 \norm{b}_{L^1(d\omega)}.
\end{align*}
Moreover,
\begin{align*}
 \sup_{\substack{g\in G\\x,y\in\bbR^N}}
 \abs{\calK_g[b_n-b](x,y)}
 &\le
 (c_\kappa^{\mathrm M})^{-2}
 \norm{b_n-b}_{L^1(d\omega)}
 \longrightarrow0.
\end{align*}
Thus \eqref{eq:sobolev-packet-kernel} converges absolutely.  Each
$\calK_g[b_n]$ is jointly continuous by dominated convergence, and the
uniform convergence above proves the joint continuity of $\calK_g[b]$.

Proposition~\ref{prop:unit} and Fatou's lemma give
\begin{align*}
 \Omega_\kappa(y)^{1/2}
 \norm{(1+d(\cdot,y))^s\calK_g[b](\cdot,y)}_{L^2(d\omega)}
 \le\,&
 \liminf_{n\to\infty}
 \Omega_\kappa(y)^{1/2}
 \norm{(1+d(\cdot,y))^s\calK_g[b_n](\cdot,y)}_{L^2(d\omega)}\\
 \le\,&
 C_s\lim_{n\to\infty}\norm{b_n}_{H^s}
 =
 C_s\norm{b}_{H^s}.
\end{align*}

Let $v=y'-y$, $y_t=y+tv$, and $\norm{v}\le1$.  Since the orbit distance is
$1$-Lipschitz in each variable,
$$
 \abs{d(x,y_t)-d(x,y)}
 \le
 \norm{y_t-y}
 =
 t\norm{v}
 \le1.
$$
For every $\alpha\in R_+$,
\begin{align*}
 \frac{1+\abs{\ip{\alpha}{y}}}{1+\norm{\alpha}}
 &\le
 1+\abs{\ip{\alpha}{y_t}}
 \le
 (1+\norm{\alpha})
 \bigl(1+\abs{\ip{\alpha}{y}}\bigr).
\end{align*}
Hence, uniformly for $0\le t\le1$,
$$
 \Omega_\kappa(y_t)\approx\Omega_\kappa(y).
$$
The fundamental theorem of calculus gives
$$
 \calK_g[b_n](x,y)-\calK_g[b_n](x,y')
 =
 -\int_0^1\calK_{g,v}^{(1)}[b_n](x,y_t)\,dt.
$$
Minkowski's inequality and \eqref{eq:unit-jet} therefore yield
\begin{align*}
 &\Omega_\kappa(y)^{1/2}
 \norm{
 (1+d(\cdot,y))^s
 \bigl(\calK_g[b_n](\cdot,y)-\calK_g[b_n](\cdot,y')\bigr)
 }_{L^2(d\omega)}\\
 \le\,&
 C_s\int_0^1
 \Omega_\kappa(y_t)^{1/2}
 \norm{
 (1+d(\cdot,y_t))^s
 \calK_{g,v}^{(1)}[b_n](\cdot,y_t)
 }_{L^2(d\omega)}\,dt\\
 \le\,&
 C_s\int_0^1\norm{v}\norm{b_n}_{H^s}\,dt
 =\, 
 C_s\norm{v}\norm{b_n}_{H^s}.
\end{align*}
Uniform convergence and Fatou's lemma now give
\begin{align*}
 &\Omega_\kappa(y)^{1/2}
 \norm{
 (1+d(\cdot,y))^s
 \bigl(\calK_g[b](\cdot,y)-\calK_g[b](\cdot,y')\bigr)
 }_{L^2(d\omega)}\\
 \le\,&
 \liminf_{n\to\infty}
 \Omega_\kappa(y)^{1/2}
 \norm{
 (1+d(\cdot,y))^s
 \bigl(\calK_g[b_n](\cdot,y)-\calK_g[b_n](\cdot,y')\bigr)
 }_{L^2(d\omega)}\\
 \le\,&
 C_s\norm{v}\lim_{n\to\infty}\norm{b_n}_{H^s}
 =
 C_s\norm{y-y'}\norm{b}_{H^s},
\end{align*}
which proves \eqref{eq:unit-difference}.

Conjugation, symmetry, covariance, and the change of variables $\xi=g\eta$ give
\begin{equation}\label{eq:packet-transpose}
\begin{split}
 \overline{\calK_g[b](y,x)}
 &=
 (c_\kappa^{\mathrm M})^{-2}
 \int_{\bbR^N}
 \overline{b(\xi)}E_\kappa(-i\xi,y)E_\kappa(igx,\xi)\,d\omega(\xi)\\
 &=
 (c_\kappa^{\mathrm M})^{-2}
 \int_{\bbR^N}
 \overline{b(g\eta)}E_\kappa(-ig^{-1}y,\eta)E_\kappa(i\eta,x)\,d\omega(\eta)\\
 &=
 \calK_{g^{-1}}[\overline{b}\circ g](x,y).
\end{split}
\end{equation}
Since $\mathcal A_0$ is $G$-invariant,
$$
 \supp(\overline{b}\circ g)
 =
 g^{-1}(\supp b)
 \subset \mathcal A_0.
$$
Orthogonal invariance and conjugation give
$
 \|\overline{b}\circ g\|_{H^s}
 =
 \norm{b}_{H^s}.
$

Applying the estimates already proved to the right-hand side of
\eqref{eq:packet-transpose} yields
\begin{align*}
 \Omega_\kappa(x)^{1/2}
 \norm{(1+d(x,\cdot))^s\calK_g[b](x,\cdot)}_{L^2(d\omega)}
  =\,&
 \Omega_\kappa(x)^{1/2}
 \norm{
 (1+d(\cdot,x))^s
 \calK_{g^{-1}}[\overline{b}\circ g](\cdot,x)
 }_{L^2(d\omega)}\\
 \le\,&
 C_s\norm{\overline{b}\circ g}_{H^s}
 =
 C_s\norm{b}_{H^s}.
\end{align*}
If $\norm{x-x'}\le1$, then
\begin{align*}
 &\Omega_\kappa(x)^{1/2}
 \norm{
 (1+d(x,\cdot))^s
 \bigl(\calK_g[b](x,\cdot)-\calK_g[b](x',\cdot)\bigr)
 }_{L^2(d\omega)}\\
 =\,&
 \Omega_\kappa(x)^{1/2}
 \norm{
 (1+d(\cdot,x))^s
 \bigl(
 \calK_{g^{-1}}[\overline{b}\circ g](\cdot,x)
 -
 \calK_{g^{-1}}[\overline{b}\circ g](\cdot,x')
 \bigr)
 }_{L^2(d\omega)}\\
 \le\,&
 C_s\norm{x-x'}\norm{\overline{b}\circ g}_{H^s}\\
 =\,&
 C_s\norm{x-x'}\norm{b}_{H^s}.
\end{align*}

Sobolev embedding and Plancherel give, for $f\in L^2(d\omega)$,
\begin{align*}
 \norm{T_bf}_{L^2(d\omega)}
 &=
 \norm{b\calF f}_{L^2(d\omega)}
 \le
 \norm{b}_{L^\infty}\norm{\calF f}_{L^2(d\omega)}
 =
 \norm{b}_{L^\infty}\norm{f}_{L^2(d\omega)}
 \le
 C_s\norm{b}_{H^s}\norm{f}_{L^2(d\omega)}.
\end{align*}
Let $f,h\in L_c^\infty(\bbR^N)$.  The kernel bound gives
\begin{align*}
 &(c_\kappa^{\mathrm M})^{-2}
 \iiint
 \abs{
 b(\xi)E_\kappa(-i\xi,y)E_\kappa(i\xi,x)
 f(y)\overline{h(x)}
 }
 \,d\omega(\xi)\,d\omega(y)\,d\omega(x)\\
 \le\,&
 (c_\kappa^{\mathrm M})^{-2}
 \norm{b}_{L^1(d\omega)}
 \norm{f}_{L^1(d\omega)}
 \norm{h}_{L^1(d\omega)}
 <\infty.
\end{align*}
Hence Plancherel, symmetry, and Fubini's theorem give
\begin{align*}
 \ip{T_bf}{h}_{L^2(d\omega)}
 &=
 \ip{b\calF f}{\calF h}_{L^2(d\omega)}
 =
 (c_\kappa^{\mathrm M})^{-2}
 \iiint
 b(\xi)E_\kappa(-i\xi,y)E_\kappa(i\xi,x)
 f(y)\overline{h(x)}
 \,d\omega(\xi)\,d\omega(y)\,d\omega(x)\\
 &=
 \iint
 K_b(x,y)f(y)\overline{h(x)}
 \,d\omega(y)\,d\omega(x).
\end{align*}
Define
$$
 F_b(x)=\int_{\bbR^N}K_b(x,y)f(y)\,d\omega(y).
$$
Then
\begin{align*}
 \sup_{x\in\bbR^N}\abs{F_b(x)}
 &\le
 \sup_{x\in\bbR^N}
 \int_{\bbR^N}\abs{K_b(x,y)f(y)}\,d\omega(y)
 \le
 (c_\kappa^{\mathrm M})^{-2}
 \norm{b}_{L^1(d\omega)}
 \norm{f}_{L^1(d\omega)}.
\end{align*}
Thus $F_b\in L^1_{\mathrm{loc}}(d\omega)$, and the preceding pairing identity gives
$F_b=T_bf$ almost everywhere.  Hence $K_b$ is an associated kernel of $T_b$.

Finally, $G$-invariance of $\mathcal A_0$, orthogonal invariance of $H^s$, and conjugation show that the
preceding argument applies to $b\circ r$ and $\overline{b}\circ r$ for every $r\in G$.  Finite sums
follow by linearity.
\end{proof}

\subsection{Chamber endpoint spaces}

We record the chamber geometry and the endpoint spaces used below.  Write $$
 G=\{g_\rho:1\le\rho\le\abs{G}\}.
$$
Recall that $\mathcal C$ is the open chamber fixed in the introduction. For $z\in \mathcal C$ and $r>0$, set
$$
 B_\mathcal C(z,r)=B(z,r)\cap \mathcal C,
 \qquad
 V_\mathcal C(z,r)=\omega(B_\mathcal C(z,r)).
$$
The Euclidean ball $B(z,r)$ may cross the reflecting walls, and no wall separation is imposed.  Applying
\cite[Lemma~5.5]{Chamber2026} to the positive system determined by $\mathcal C$ gives
\begin{equation}\label{eq:chamber-volume-formula}
 V_\mathcal C(z,r)
 \approx
 r^N
 \prod_{\alpha\in R_+}
 \bigl(r\norm{\alpha}+\abs{\ip{\alpha}{z}}\bigr)^{2\kappa(\alpha)}.
\end{equation}
The absolute-value product is independent of the choice of positive system, and
$$
 \mathbf N=N+2\sum_{\alpha\in R_+}\kappa(\alpha).
$$
Hence
\begin{align}\label{eq:chamber-volume-scaling}
 V_\mathcal C(z,r)
 &\approx
 r^{\mathbf N}
 \prod_{\alpha\in R_+}
 \left(\norm{\alpha}+\abs{\ip{\alpha}{r^{-1}z}}\right)^{2\kappa(\alpha)}
 \approx
 r^{\mathbf N}\Omega_\kappa(r^{-1}z).
\end{align}
For $0<r\le R$, \eqref{eq:chamber-volume-formula} gives
\begin{align}\label{eq:chamber-relative-growth}
 V_\mathcal C(z,R)
 &\le
 C\left(\frac Rr\right)^N
 \prod_{\alpha\in R_+}
 \left(
 \frac{R\norm{\alpha}+\abs{\ip{\alpha}{z}}}
 {r\norm{\alpha}+\abs{\ip{\alpha}{z}}}
 \right)^{2\kappa(\alpha)}
 V_\mathcal C(z,r)
 \le
 C\left(\frac Rr\right)^{\mathbf N}V_\mathcal C(z,r).
\end{align}
Thus $(\mathcal C,\norm{\cdot},d\omega)$ is a doubling space of homogeneous type, and
$$
 V_\mathcal C(z,1)\approx\Omega_\kappa(z),
 \qquad
 \omega(\mathcal C)=\lim_{R\to\infty}V_\mathcal C(z,R)=\infty.
$$

We shall also use the fact
\begin{equation}\label{eq:chamber-distance}
 d(x,y)=\norm{x-y},
 \qquad
 x,y\in\overline{\mathcal C}.
\end{equation}

A Coifman--Weiss $(1,2)$-atom is a function $a\in L^2(\mathcal C,d\omega)$ supported in some
$B_\mathcal C(z,r)$ and satisfying
$$
 \int_\mathcal Ca\,d\omega=0,
 \qquad
 \norm{a}_{L^2(d\omega)}\le V_\mathcal C(z,r)^{-1/2}.
$$
In particular,
$$
 \norm{a}_{L^1(d\omega)}
 \le
 V_\mathcal C(z,r)^{1/2}\norm{a}_{L^2(d\omega)}
 \le1.
$$
Let $H^1_{\mathrm{at}}(\mathcal C)$ consist of all $f\in L^1(\mathcal C,d\omega)$ admitting a representation
$$
 f=\sum_{n=1}^\infty\lambda_na_n
 \quad\hbox{in }L^1(\mathcal C,d\omega),
 \qquad
 \sum_{n=1}^\infty\abs{\lambda_n}<\infty,
$$
where the $a_n$ are $(1,2)$-atoms, and set
$\displaystyle
 \norm{f}_{H^1_{\mathrm{at}}(\mathcal C)}
 =
 \inf_{f=\sum_n\lambda_na_n}
 \sum_{n=1}^\infty\abs{\lambda_n}.
$
For every such representation,
\begin{align*}
 \norm{f}_{L^1(d\omega)}
 \le
 \sum_{n=1}^\infty\abs{\lambda_n}\qquad \text{and}\qquad 
 \norm{
 f-\sum_{n=1}^M\lambda_na_n
 }_{H^1_{\mathrm{at}}(\mathcal C)}
 \le
 \sum_{n>M}\abs{\lambda_n}
 \longrightarrow0.
\end{align*}
Thus $H^1_{\mathrm{at}}(\mathcal C)$ is a subspace of $L^1(\mathcal C,d\omega)$, and finite atomic sums are
dense in it.

Assume that $b\in L^1(B_\mathcal C(z,r),d\omega)$ for every $z\in \mathcal C$ and $r>0$.  This includes integrability
up to the walls: if $\zeta\in\overline C$, $r>0$, and $z\in \mathcal C\cap B(\zeta,r)$, then
$
 B(\zeta,r)\cap C\subset B_\mathcal C(z,2r).
$

Define
\begin{align*}
 b_{B_\mathcal C(z,r)}
 =
 \frac1{V_\mathcal C(z,r)}
 \int_{B_\mathcal C(z,r)}b\,d\omega.
\end{align*}
Then
\begin{align*}
 \norm{b}_{\BMO(\mathcal C)}
 =
 \sup_{\substack{z\in \mathcal C\\r>0}}
 \frac1{V_\mathcal C(z,r)}
 \int_{B_\mathcal C(z,r)}
 \abs{b-b_{B_\mathcal C(z,r)}}\,d\omega.
\end{align*}
The space $\BMO(\mathcal C)$ consists of the functions with finite seminorm, modulo constants.  Since
$\omega(\mathcal C)=\infty$, no exceptional non-cancellative atom is needed.  Coifman--Weiss $(1,q)$-atomic
equivalence and duality, together with the John--Nirenberg inequality
\cite[Theorems~A and~B]{CW1977}, give
\begin{align*}
 \norm{b}_{\BMO(\mathcal C)}
 \approx
 \sup_{\substack{z\in \mathcal C\\r>0}}
 \left(
 \frac1{V_\mathcal C(z,r)}
 \int_{B_\mathcal C(z,r)}
 \abs{b-b_{B_\mathcal C(z,r)}}^2\,d\omega
 \right)^{1/2}.
\end{align*}
Moreover,
$
 \bigl(H^1_{\mathrm{at}}(\mathcal C)\bigr)^*
 \simeq
 \BMO(\mathcal C).
$

For the finite index set $G$, define
$$
 H^1_{\mathrm{at}}(\mathcal C;\ell^1(G))
 =
 \prod_{g\in G}H^1_{\mathrm{at}}(\mathcal C),
\qquad \text{with}\qquad
 \norm{F}_{H^1_{\mathrm{at}}(\mathcal C;\ell^1(G))}
 =
 \sum_{g\in G}\norm{F_g}_{H^1_{\mathrm{at}}(\mathcal C)}.
$$
Similarly,
$$
 \BMO(\mathcal C;\ell^\infty(G))
 =
 \prod_{g\in G}\BMO(\mathcal C),
\qquad \text{with}\qquad
 \norm{F}_{\BMO(\mathcal C;\ell^\infty(G))}
 =
 \max_{g\in G}\norm{F_g}_{\BMO(\mathcal C)}.
$$
Componentwise duality yields
$$
 \bigl(H^1_{\mathrm{at}}(\mathcal C;\ell^1(G))\bigr)^*
 \simeq
 \BMO(\mathcal C;\ell^\infty(G)),
$$
under the componentwise integral pairing, initially on finite atomic sums.

The chamber walls have $d\omega$-measure zero, and
$$
 (U^{-1}F)(gx)=F_g(x),
 \qquad
 x\in \mathcal C,\quad g\in G,
$$
defines $U^{-1}F$ almost everywhere.  Finally, set
$$
 H^1_{\ch}
 =
 U^{-1}H^1_{\mathrm{at}}(\mathcal C;\ell^1(G)),
\qquad \text{with}\qquad
 \norm{f}_{H^1_{\ch}}
 =
 \norm{Uf}_{H^1_{\mathrm{at}}(\mathcal C;\ell^1(G))}.
$$
Likewise,
$$
 \BMO_{\ch}
 =
 U^{-1}\BMO(\mathcal C;\ell^\infty(G)),
\qquad \text{with}\qquad
 \norm{f}_{\BMO_{\ch}}
 =
 \norm{Uf}_{\BMO(\mathcal C;\ell^\infty(G))}.
$$
Thus
$$
 \bigl(H^1_{\ch}\bigr)^*
 \simeq
 \BMO_{\ch}.
$$
The vector BMO space is taken modulo componentwise constants, which pull back to separate constants
on the chambers.  These are chamber-pullback spaces, and no identification with intrinsic Dunkl Hardy or
BMO spaces is intended.

\subsection{Endpoint estimates for finite dyadic sums}

We pass from the dyadic kernel estimates to weak type $(1,1)$ and the chamber endpoints.

\begin{proposition}\label{prop:finite-frequency-closure}
On the chamber space above, let $r_j=2^{-j}$ and fix $s>\mathbf N/2$.  For $j\in\bbZ$, let
$$
 S_j=\bigl(S_j^{\rho\tau}\bigr)_{1\le\rho,\tau\le\abs{G}}
$$
be a matrix operator with measurable kernel
$K_j=(K_j^{\rho\tau})_{\rho,\tau}$.  Assume that $S_j$ is associated with $K_j$ off the diagonal:
$$
 \ip{S_jF}{H}_{L^2(\mathcal C;\ell^2(G))}
 =
 \int_\mathcal C\int_\mathcal C
 \ip{K_j(x,y)F(y)}{H(x)}_{\ell^2(G)}
 \,d\omega(y)\,d\omega(x)
$$
whenever $F,H\in L_c^\infty(\mathcal C;\ell^2(G))$ and
$\operatorname{dist}(\supp F,\supp H)>0$.

Set
$$
 K_{j,0}(x,y)=K_j(x,y),
 \qquad
 K_{j,1}(x,y)=K_j(y,x)^*,
$$
where $^*$ denotes conjugate transpose.  Suppose that, for every
$j\in\bbZ$, $\varepsilon\in\{0,1\}$, $1\le\rho,\tau\le\abs{G}$, and $y\in \mathcal C$,
\begin{equation}\label{eq:closure-kernel-size}
 V_\mathcal C(y,r_j)^{1/2}
 \norm{
 \left(1+\frac{\norm{\cdot-y}}{r_j}\right)^s
 K_{j,\varepsilon}^{\rho\tau}(\cdot,y)
 }_{L^2(\mathcal C,d\omega)}
 \le A.
\end{equation}
For the same indices and all $y,y'\in \mathcal C$ with $\norm{y-y'}\le r_j$, assume that
\begin{equation}\label{eq:closure-kernel-diff}
 V_\mathcal C(y,r_j)^{1/2}
 \norm{
 \left(1+\frac{\norm{\cdot-y}}{r_j}\right)^s
 \bigl(K_{j,\varepsilon}^{\rho\tau}(\cdot,y)
       -K_{j,\varepsilon}^{\rho\tau}(\cdot,y')\bigr)
 }_{L^2(\mathcal C,d\omega)}
 \le
 A\frac{\norm{y-y'}}{r_j}.
\end{equation}

For finite $E\subset\bbZ$, set
$$
 S_E=\sum_{j\in E}S_j,
 \qquad
 K_{E,\varepsilon}=\sum_{j\in E}K_{j,\varepsilon},
$$
and suppose, uniformly in $E$, that
$$
 \norm{S_E}_{L^2(\mathcal C;\ell^2(G))\to L^2(\mathcal C;\ell^2(G))}
 \le M_2.
$$
Then, uniformly in finite $E$,
\begin{equation}\label{eq:closure-Hormander}
 \sup_{\substack{\varepsilon\in\{0,1\}\\1\le\rho,\tau\le\abs{G}}}
 \sup_{\substack{u,u'\in \mathcal C\\u\ne u'}}
 \int_{\{v\in \mathcal C:\,\norm{v-u}>2\norm{u-u'}\}}
 \abs{K_{E,\varepsilon}^{\rho\tau}(v,u)
       -K_{E,\varepsilon}^{\rho\tau}(v,u')}
 \,d\omega(v)
 \le CA.
\end{equation}
Let $X_G=\mathcal C\times G$ carry $d\omega$ times counting measure.  Then
\begin{align*}
 \norm{S_E}_{L^1(X_G)\to L^{1,\infty}(X_G)}
 &\le C(A+M_2).
\end{align*}
Moreover,
\begin{align*}
 \norm{S_E}_{H^1_{\mathrm{at}}(\mathcal C;\ell^1(G))\to L^1(\mathcal C;\ell^1(G))}
 &\le C(A+M_2).
\end{align*}
Finally,
\begin{align*}
 \norm{S_E}_{L^\infty(\mathcal C;\ell^\infty(G))\to\BMO(\mathcal C;\ell^\infty(G))}
 &\le C(A+M_2).
\end{align*}
Here $C$ depends only on $s$, $\mathbf N$, $G$, and the chamber geometric- and measure-doubling
constants, and is independent of $E$.
\end{proposition}

\begin{proof}
Set
$
 \beta=s-\frac{\mathbf N}{2}>0.
$
By \eqref{eq:chamber-relative-growth}, for $a\ge1$,
\begin{align*}
 \int_{\{x:\,\norm{x-y}>a r_j\}}
 \left(1+\frac{\norm{x-y}}{r_j}\right)^{-2s}
 \,d\omega(x)
 &\le
 \sum_{k=0}^\infty
 (2^ka)^{-2s}V_\mathcal C(y,2^{k+1}a r_j)\\
 &\le
 C V_\mathcal C(y,r_j)a^{\mathbf N-2s}
 \sum_{k=0}^\infty2^{k(\mathbf N-2s)}
 \le
 C V_\mathcal C(y,r_j)a^{-2\beta}.
\end{align*}
Also,
\begin{align*}
 \int_\mathcal C
 \left(1+\frac{\norm{x-y}}{r_j}\right)^{-2s}
 \,d\omega(x)
 &\le
 V_\mathcal C(y,r_j)
 +
 \sum_{k=0}^\infty2^{-2ks}V_\mathcal C(y,2^{k+1}r_j)
 \le
 C V_\mathcal C(y,r_j).
\end{align*}
Hence Cauchy--Schwarz and \eqref{eq:closure-kernel-size} give
\begin{align}\label{eq:closure-L1-tail}
& \int_{\{\norm{x-y}>a r_j\}}
 \abs{K_{j,\varepsilon}^{\rho\tau}(x,y)}\,d\omega(x)\notag\\
 \le\,&
 \left(
 \int_{\{\norm{x-y}>a r_j\}}
 \left(1+\frac{\norm{x-y}}{r_j}\right)^{-2s}
 \,d\omega(x)
 \right)^{1/2}
 \norm{
 \left(1+\frac{\norm{\cdot-y}}{r_j}\right)^s
 K_{j,\varepsilon}^{\rho\tau}(\cdot,y)
 }_2\notag\\
 \le\,&
 C V_\mathcal C(y,r_j)^{1/2}a^{-\beta}
 \cdot A V_\mathcal C(y,r_j)^{-1/2} 
 =  CAa^{-\beta},
 \qquad a\ge1.
\end{align}
Similarly, \eqref{eq:closure-kernel-diff} and the full-weight integral give
\begin{align}\label{eq:closure-L1-diff}
 &\int_\mathcal C
 \abs{K_{j,\varepsilon}^{\rho\tau}(x,y)
       -K_{j,\varepsilon}^{\rho\tau}(x,y')}
 \,d\omega(x) \notag\\
 \le\,&
 \left(
 \int_\mathcal C
 \left(1+\frac{\norm{x-y}}{r_j}\right)^{-2s}
 \,d\omega(x)
 \right)^{1/2}
 \norm{
 \left(1+\frac{\norm{\cdot-y}}{r_j}\right)^s
 \bigl(K_{j,\varepsilon}^{\rho\tau}(\cdot,y)
       -K_{j,\varepsilon}^{\rho\tau}(\cdot,y')\bigr)
 }_2\notag\\
 \le\,&
 C V_\mathcal C(y,r_j)^{1/2}
 \cdot A\frac{\norm{y-y'}}{r_j}V_\mathcal C(y,r_j)^{-1/2}\notag\\
 =\,&  CA\frac{\norm{y-y'}}{r_j},
 \qquad
 \norm{y-y'}\le r_j.
\end{align}

Fix $y\ne y'$ and set $\delta=\norm{y-y'}$.  Since
$$
 \norm{x-y}>2\delta
 \quad\Longrightarrow\quad
 \norm{x-y'}>\delta,
$$
choose $j_0\in\bbZ$ so that $r_{j_0+1}<\delta\le r_{j_0}$.  
Then for every $\varepsilon\in\{0,1\}$ and $1\le\rho,\tau\le\abs{G}$, \eqref{eq:closure-L1-tail} gives
\begin{align*}
\sum_{j\in E, r_j\le\delta}\int_{\{\norm{x-y}>2\delta\}}
 \abs{K_{j,\varepsilon}^{\rho\tau}(x,y)}
 \,d\omega(x)
&\le
\sum_{j\in E, j\ge j_0+1}\int_{\{\norm{x-y}>2^{j-j_0}r_j\}}
 \abs{K_{j,\varepsilon}^{\rho\tau}(x,y)}
 \,d\omega(x)\\
&\lesssim \sum_{j:j-j_0\ge1}2^{-\beta(j-j_0)}\,A
\lesssim A
\end{align*}
and
\begin{align*}
\sum_{j\in E, r_j\le\delta}\int_{\{\norm{x-y}>2\delta\}}
 \abs{K_{j,\varepsilon}^{\rho\tau}(x,y')}
 \,d\omega(x)
&\le
\sum_{j\in E, j\ge j_0+1}\int_{\{\norm{x-y'}>2^{j-j_0-1}r_j\}}
 \abs{K_{j,\varepsilon}^{\rho\tau}(x,y')}
 \,d\omega(x)\\
&\lesssim \sum_{j:j-j_0-1\ge0}2^{-\beta(j-j_0-1)}\,A
\lesssim A.
\end{align*}
\eqref{eq:closure-L1-diff} gives
\begin{align*}
\sum_{j\in E, r_j>\delta}\int_{\{\norm{x-y}>2\delta\}}
 \abs{K_{j,\varepsilon}^{\rho\tau}(x,y)
       -K_{j,\varepsilon}^{\rho\tau}(x,y')}
 \,d\omega(x)
 &\lesssim
 \sum_{j\in E, j\le j_0} A\frac{\norm{y-y'}}{r_j}
 \le \sum_{j: j-j_0\le 0} 2^{j-j_0}A
 \lesssim A.
\end{align*}
Summing the three terms above yields
\begin{align*}
 \int_{\{\norm{x-y}>2\delta\}}
 \abs{K_{E,\varepsilon}^{\rho\tau}(x,y)
       -K_{E,\varepsilon}^{\rho\tau}(x,y')}
 \,d\omega(x)
 \le CA,
\end{align*}
which proves \eqref{eq:closure-Hormander}.

For finite $E$, summation gives the off-diagonal association of $S_E$ with $K_{E,0}$, while coordinate
injection and projection give
$$
 \norm{S_E^{\rho\tau}}_{2\to2}
 \le
 \norm{S_E}_{L^2(\mathcal C;\ell^2(G))\to L^2(\mathcal C;\ell^2(G))}
 \le M_2.
$$
For a chamber ball $B=B_\mathcal C(z,r)$, write $2B=B_\mathcal C(z,2r)$.  We first prove
\begin{equation}\label{eq:closure-cancellation}
 \int_{\mathcal C\setminus2B}\abs{S_E^{\rho\tau}b}\,d\omega
 \le
 CA\norm{b}_1
\end{equation}
whenever
$$
 b\in L^1(\mathcal C)\cap L^2(\mathcal C),
 \qquad
 \supp b\subset B,
 \qquad
 \int_\mathcal Cb\,d\omega=0.
$$
Choose $0\le\chi_n\le1$ in $L_c^\infty(B)$ with $\chi_n\to1$ almost everywhere, fix
$\zeta\in L_c^\infty(B)$ with $\int_\mathcal C\zeta\,d\omega=1$, and set
$$
 \widetilde b_n
 =
 \chi_nb\mathbf1_{\{\abs{b}\le n\}},\qquad 
 b_n
 =
 \widetilde b_n-\zeta\int_\mathcal C\widetilde b_n\,d\omega.
$$
Then
$$
 b_n\in L_c^\infty(B),
 \qquad
 \int_\mathcal Cb_n\,d\omega=0,
 \qquad
 b_n\longrightarrow b
 \quad\hbox{in }L^1(\mathcal C)\cap L^2(\mathcal C).
$$
For $D\subset C\setminus2B$ and $h_0\in L_c^\infty(D)$,
\begin{align*}
 \ip{S_E^{\rho\tau}b_n}{h_0}
 &=
 \int_D\overline{h_0(x)}
 \int_B
 \left (K_{E,0}^{\rho\tau}(x,y)-K_{E,0}^{\rho\tau}(x,z)\right)b_n(y)
 \,d\omega(y)\,d\omega(x).
\end{align*}
Since
$$
 x\notin2B,\ y\in B
 \quad\Longrightarrow\quad
 \norm{x-z}\ge2r>2\norm{z-y},
$$
Applying \eqref{eq:closure-Hormander} to the ordered pair $(z,y)$ and then using Fubini gives
\begin{align*}
 \int_{\mathcal C\setminus2B}\abs{S_E^{\rho\tau}b_n(x)}\,d\omega(x)
 &\le
 \int_B\abs{b_n(y)}
 \int_{\{\norm{x-z}>2\norm{z-y}\}}
 \abs{K_{E,0}^{\rho\tau}(x,y)-K_{E,0}^{\rho\tau}(x,z)}
 \,d\omega(x)\,d\omega(y)
 \le
 CA\norm{b_n}_1.
\end{align*}
Moreover,
$$
 \norm{S_E^{\rho\tau}(b_n-b)}_2
 \le
 M_2\norm{b_n-b}_2
 \longrightarrow0.
$$
Thus a subsequence, Fatou's lemma, and $\norm{b_n}_1\to\norm{b}_1$ prove
\eqref{eq:closure-cancellation}.

Let $T=S_E^{\rho\tau}$ and $f\in L^1(\mathcal C)\cap L^2(\mathcal C)$.  If $A+M_2=0$, then $S_E=0$ on
$L^2(\mathcal C;\ell^2(G))$, and all conclusions are immediate.  Hence assume
$$
 L=A+M_2>0.
$$
Fix a dyadic system $\mathscr D$ from \cite[Theorem~2.2]{HK2012}, with parameter
$0<\delta_D<1$.  
For $Q\in\mathscr D_k$,
$$
 B_\mathcal C(z_Q,c_D\delta_D^k)
 \subset Q
 \subset
 B_Q:=B_\mathcal C(z_Q,C_D\delta_D^k),
 \qquad
 \omega(B_Q)\le C\omega(Q).
$$
If $Q^k(x)$ is the generation-$k$ cube containing $x$, set $R_k=2C_D\delta_D^k$.  Then
\begin{align*}
 \omega(Q^k(x))
 &\ge
 V_\mathcal C(z_{Q^k(x)},c_D\delta_D^k)
 \gtrsim
 \delta_D^{k\mathbf N}
 \longrightarrow\infty,
 \qquad k\longrightarrow-\infty.
\end{align*}
Moreover,
$$
 Q^k(x)
 \subset B_\mathcal C(x,R_k),
\qquad \text{and}\qquad
 V_\mathcal C(x,R_k)
 \le C\omega(Q^k(x)).
$$
Consequently, for almost every $x$,
\begin{align*}
 \frac1{\omega(Q^k(x))}
 \int_{Q^k(x)}\abs{f(y)}\,d\omega(y)
 \le
 \frac{\norm{f}_1}{\omega(Q^k(x))}
 \longrightarrow0,
 \qquad k\longrightarrow-\infty,
 \end{align*}
 and
 \begin{align*}
 &\frac1{\omega(Q^k(x))}
 \int_{Q^k(x)}\abs{f(y)-f(x)}\,d\omega(y)
 \le\,
 \frac{C}{V_\mathcal C(x,R_k)}
 \int_{B_\mathcal C(x,R_k)}
 \abs{f(y)-f(x)}\,d\omega(y)
 \longrightarrow0,
 \qquad k\longrightarrow\infty.
\end{align*}

At height
$
 \alpha=\frac{\lambda}{L},
$
let $\{Q_i\}$ be the maximal dyadic cubes satisfying
$$
 \frac1{\omega(Q_i)}\int_{Q_i}\abs{f}\,d\omega>\alpha.
$$
They are pairwise disjoint.  If $\widehat Q_i$ is the parent of $Q_i$, then
$$
 \abs{f_{Q_i}}
 \le
 \frac{\omega(\widehat Q_i)}{\omega(Q_i)}
 \frac1{\omega(\widehat Q_i)}
 \int_{\widehat Q_i}\abs{f}\,d\omega
 \le C\alpha.
$$
Set
$$
 b_i=(f-f_{Q_i})\mathbf1_{Q_i},
 \qquad
 g=f-\sum_i b_i.
$$
Dyadic differentiation and maximality give $\abs{f}\le\alpha$ almost everywhere outside
$\bigcup_iQ_i$.  Since $g=f$ there and $g=f_{Q_i}$ on $Q_i$,
\begin{align*}
 \abs{g}
 &=
 \abs{f}\mathbf1_{\mathcal C\setminus\bigcup_iQ_i}
 +
 \sum_i\abs{f_{Q_i}}\mathbf1_{Q_i}
 \le
 C\alpha.
\end{align*}
Moreover,
\begin{align*}
 \norm{g}_1
 &\le
 \int_{\mathcal C\setminus\bigcup_iQ_i}\abs{f}\,d\omega
 +
 \sum_i\abs{f_{Q_i}}\omega(Q_i)
 \le
 \norm{f}_1.
\end{align*}
Consequently,
\begin{align*}
 \norm{g}_2^2
 &\le
 \norm{g}_\infty\norm{g}_1
 \le
 C\alpha\norm{f}_1.
\end{align*}

The bad functions satisfy
$$
 \supp b_i\subset Q_i,
 \qquad
 \int_\mathcal Cb_i\,d\omega
 =
 \int_{Q_i}f\,d\omega-f_{Q_i}\omega(Q_i)
 =0.
$$
Their $L^1$ norms are summable, since
\begin{align*}
 \sum_i\norm{b_i}_1
 &\le
 \sum_i
 \left(
 \int_{Q_i}\abs{f}\,d\omega
 +
 \abs{f_{Q_i}}\omega(Q_i)
 \right)
 \le
 2\norm{f}_1.
\end{align*}
By the choice of $Q_i$ and their pairwise disjointness,
\begin{align*}
 \alpha\sum_i\omega(Q_i)
 &<
 \sum_i\int_{Q_i}\abs{f}\,d\omega
 \le
 \norm{f}_1.
\end{align*}
Finally, disjointness of the $Q_i$ and $\abs{f_{Q_i}}\le C\alpha$ give
\begin{align*}
 \sum_i\norm{b_i}_2^2
 &\le
 2\sum_i\int_{Q_i}\abs{f}^2\,d\omega
 +
 2\sum_i\abs{f_{Q_i}}^2\omega(Q_i)
 \\
 &\le
 2\norm{f}_2^2
 +
 C\alpha\sum_i\abs{f_{Q_i}}\omega(Q_i)
 \le
 2\norm{f}_2^2+C\alpha\norm{f}_1
 <
 \infty.
\end{align*}
Therefore, the series $\sum_i b_i$ converges to $f-g$ in $L^1$, regardless of the order of summation, since,
\begin{align*}
 \norm{\sum_{i=1}^M b_i-(f-g)}_1
 &\le
 \sum_{i>M}\norm{b_i}_1
 \longrightarrow0.
\end{align*}
Moreover,
\begin{align*}
 \norm{\sum_{i=1}^M b_i-(f-g)}_2^2
 &=
 \sum_{i>M}\norm{b_i}_2^2
 \longrightarrow0.
\end{align*}

Let
$$
 B_i=B_{Q_i},
 \qquad
 \mathcal O^*=\bigcup_i2B_i.
$$
The dyadic ball inclusions and doubling give
\begin{align*}
 \omega(\mathcal O^*)
 &\le
 C\sum_i\omega(Q_i)
 \le
 C\frac{L}{\lambda}\norm{f}_1.
\end{align*}
Moreover,
\begin{align*}
 \omega\{\abs{Tg}>\lambda/3\}
 &\le
 \frac9{\lambda^2}\norm{Tg}_2^2
 \le
 \frac{9M_2^2}{\lambda^2}\norm{g}_2^2
 \le
 C\frac{M_2^2}{\lambda L}\norm{f}_1
 \le
 C\frac{L}{\lambda}\norm{f}_1.
\end{align*}
For $b^{(M)}=\sum_{i=1}^Mb_i$, \eqref{eq:closure-cancellation} gives
\begin{align*}
 \int_{\mathcal C\setminus\mathcal O^*}\abs{Tb^{(M)}}\,d\omega
 &\le
 \sum_{i=1}^M
 \int_{\mathcal C\setminus2B_i}\abs{Tb_i}\,d\omega
 \le
 CA\sum_{i=1}^M\norm{b_i}_1.
\end{align*}
Since
$$
 \norm{Tb^{(M)}-T(f-g)}_2
 \le
 M_2\norm{b^{(M)}-(f-g)}_2
 \longrightarrow0,
$$
a subsequence and Fatou's lemma yield
$$
 \int_{\mathcal C\setminus\mathcal O^*}\abs{T(f-g)}\,d\omega
 \le
 CA\norm{f}_1.
$$
Therefore
\begin{align*}
 \omega\{\abs{Tf}>\lambda\}
 &\le
 \omega(\mathcal O^*)
 +\omega\{\abs{Tg}>\lambda/3\}
 +\frac3\lambda
 \int_{\mathcal C\setminus\mathcal O^*}\abs{T(f-g)}\,d\omega
 \le
 C\frac{A+M_2}{\lambda}\norm{f}_1.
\end{align*}
Density and the completeness of $L^{1,\infty}$ give the scalar weak-type extension. Hence, for
vector-valued $F$,
\begin{align*}
 &\sum_{1\le\rho\le\abs{G}}\omega
 \bigl\{x:\abs{(S_EF)_\rho(x)}>\lambda\bigr\}
 \le\,
 \sum_{1\le\rho,\tau\le\abs{G}}
 \omega\left\{x:
 \abs{S_E^{\rho\tau}F_\tau(x)}>\frac{\lambda}{\abs{G}}
 \right\}  \le
 \frac{C_G(A+M_2)}{\lambda}
 \sum_{\tau=1}^{\abs{G}}\norm{F_\tau}_1.
\end{align*}

Let $a$ be a scalar $(1,2)$-atom supported in $B=B_\mathcal C(z,r)$.  By doubling,
\eqref{eq:closure-cancellation}, and the scalar $L^2$ bound,
\begin{align*}
 \norm{S_E^{\rho\tau}a}_1
 &\le
 \omega(2B)^{1/2}\norm{S_E^{\rho\tau}a}_2
 +
 \int_{\mathcal C\setminus2B}\abs{S_E^{\rho\tau}a}\,d\omega
 \le
 CM_2\omega(B)^{1/2}\norm{a}_2
 +CA\norm{a}_1
 \le
 C(A+M_2).
\end{align*}
Thus every finite componentwise atomic sum $H$ satisfies
\begin{align*}
 \norm{S_EH}_{L^1(\mathcal C;\ell^1(G))}
 &\le
 \sum_{1\le\rho,\tau\le\abs{G}}
 \norm{S_E^{\rho\tau}H_\tau}_1
 \le
 C_G(A+M_2)
 \norm{H}_{H^1_{\mathrm{at}}(\mathcal C;\ell^1(G))}.
\end{align*}
Density gives the asserted $H^1\to L^1$ extension, which agrees with the weak-type extension on their
common domain.

For separated supports,
$$
 \ip{S_E^*F}{H}_{L^2(\mathcal C;\ell^2(G))}
 =
 \int_\mathcal C\int_\mathcal C
 \ip{K_{E,1}(x,y)F(y)}{H(x)}_{\ell^2(G)}
 \,d\omega(y)\,d\omega(x).
$$
Thus $K_{E,1}(x,y)=K_{E,0}(y,x)^*$ is the adjoint kernel.  Since
$\norm{S_E^*}_{2\to2}\le M_2$, the preceding atomic argument gives
$$
 \norm{S_E^*H}_{L^1(\mathcal C;\ell^1(G))}
 \le
 C_G(A+M_2)
 \norm{H}_{H^1_{\mathrm{at}}(\mathcal C;\ell^1(G))}.
$$
For $F\in L^\infty(\mathcal C;\ell^\infty(G))$ and a finite componentwise atomic sum $H$, define
$$
 \Lambda_F(H)=\ip{S_E^*H}{F}_{X_G}.
$$
Then
\begin{align*}
 \abs{\Lambda_F(H)}
 &\le
 \norm{F}_{L^\infty(\mathcal C;\ell^\infty(G))}
 \norm{S_E^*H}_{L^1(\mathcal C;\ell^1(G))}
 \le
 C_G(A+M_2)
 \norm{F}_{L^\infty(\mathcal C;\ell^\infty(G))}
 \norm{H}_{H^1_{\mathrm{at}}(\mathcal C;\ell^1(G))}.
\end{align*}
Componentwise Coifman--Weiss duality gives
$B_F\in\BMO(\mathcal C;\ell^\infty(G))$ such that
$$
 \ip{H}{B_F}_{X_G}
 =
 \Lambda_F(H).
$$
Moreover,
\begin{align*}
 \norm{B_F}_{\BMO(\mathcal C;\ell^\infty(G))}
 \le
 C_G(A+M_2)
 \norm{F}_{L^\infty(\mathcal C;\ell^\infty(G))}.
\end{align*}
Define $S_EF=B_F$.  If
$F\in L^\infty(\mathcal C;\ell^\infty(G))\cap L^2(\mathcal C;\ell^2(G))$, then
$$
 \ip{H}{B_F}_{X_G}
 =
 \ip{S_E^*H}{F}_{X_G}
 =
 \ip{H}{S_E^{(2)}F}_{X_G},
$$
where $S_E^{(2)}$ denotes the original $L^2$ operator.  Hence $B_F$ is precisely the BMO class of
$S_E^{(2)}F$.
\end{proof}

\subsection{Proof of the multiplier theorem and applications}

\begin{proof}[Proof of Theorem~\ref{thm:main}]
By Lemma~\ref{lem:dyadic},
$
 \norm{m}_{L^\infty(\bbR^N)}\le C\calS_\sigma(m).
$
Hence Plancherel gives
$
 \norm{T_m}_{L^2(d\omega)\to L^2(d\omega)}
 \le C\calS_\sigma(m).
$

Define
$
 m_j(\xi)=\psi(2^{-j}\norm{\xi})m(\xi)
 $ and $
 b_j(\eta)=m_j(2^j\eta)=\psi(\norm{\eta})m(2^j\eta).
$
By definition,
\begin{equation}\label{eq:bj-bound}
 \sup_{j\in\bbZ}\norm{b_j}_{H^\sigma}
 =
 \calS_\sigma(m).
\end{equation}
Moreover, $\sigma>\mathbf N/2\ge N/2$, so each $b_j$ has a continuous representative supported in
$
 \supp\bigl(\psi(\norm{\cdot})\bigr)\subset \mathcal A_0,
$
and so is every $b_j\circ g$.

For a Sobolev amplitude $b$ supported in $\mathcal A_0$, write
$$
 K_b=\calK_e[b].
$$
By Lemma~\ref{lem:sobolev-kernel-closure}, $K_b$ is an associated kernel of $T_b$. For $G=\{g_\rho:1\le\rho\le\abs{G}\}$,
$$
 K_b^{\rho\tau}(x,y)
 =
 K_b(g_\rho x,g_\tau y)
 =
 \calK_{g_\rho^{-1}g_\tau}[b\circ g_\rho](x,y),
 \qquad x,y\in \mathcal C.
$$

Set $r_j=2^{-j}$, $K_j=(K_{m_j}^{\rho\tau})_{\rho,\tau}$, and
\begin{align*}
 K_{j,0}^{\rho\tau}(x,y)
 &=K_{m_j}^{\rho\tau}(x,y).
\end{align*}
Also set
$$
 b_{j,0}=b_j.
$$
For the adjoint kernels, set
\begin{align*}
 K_{j,1}^{\rho\tau}(x,y)
 &=\overline{K_{m_j}^{\tau\rho}(y,x)}
 =K_{\overline{m_j}}^{\rho\tau}(x,y).
\end{align*}
Finally, set
$$
 b_{j,1}=\overline{b_j}.
$$
The change of variables $\xi=2^jg_\rho\eta$ gives, for $\varepsilon\in\{0,1\}$,
\begin{equation}\label{eq:lifted-scaling}
 K_{j,\varepsilon}^{\rho\tau}(x,y)
 =
 2^{j\mathbf N}
 \calK_{g_\rho^{-1}g_\tau}[b_{j,\varepsilon}\circ g_\rho](2^jx,2^jy).
\end{equation}
By \eqref{eq:chamber-distance}, $d(x,y)=\norm{x-y}$ on $C$.  Since $C$ is a cone, homogeneity and
\eqref{eq:chamber-volume-scaling} give
$$
 d\omega(2^{-j}X)=2^{-j\mathbf N}d\omega(X),
 \qquad
 V_\mathcal C(y,r_j)=2^{-j\mathbf N}V_\mathcal C(2^jy,1)
 \approx
 2^{-j\mathbf N}\Omega_\kappa(2^jy).
$$
Writing $Y=2^jy$, Lemma~\ref{lem:sobolev-kernel-closure} and \eqref{eq:lifted-scaling} give
\begin{align*}
 &V_\mathcal C(y,r_j)^{1/2}
 \norm{
 \left(1+\frac{\norm{\cdot-y}}{r_j}\right)^s
 K_{j,\varepsilon}^{\rho\tau}(\cdot,y)
 }_2\\
 =\,&
 V_\mathcal C(y,r_j)^{1/2}2^{j\mathbf N/2}
 \norm{(1+\norm{\cdot-Y})^s
 \calK_{g_\rho^{-1}g_\tau}[b_{j,\varepsilon}\circ g_\rho](\cdot,Y)}_2\\
 \le\,&
 C\Omega_\kappa(Y)^{1/2}
 \norm{(1+\norm{\cdot-Y})^s
 \calK_{g_\rho^{-1}g_\tau}[b_{j,\varepsilon}\circ g_\rho](\cdot,Y)}_2\\
 \le\,&
 C_s\norm{b_j}_{H^\sigma}.
\end{align*}
If $\norm{y-y'}\le r_j$, set $Y'=2^jy'$.  Since
$$
 \norm{Y-Y'}=\frac{\norm{y-y'}}{r_j}\le1,
$$
Lemma~\ref{lem:sobolev-kernel-closure} and the same change of variables give
\begin{align*}
 &V_\mathcal C(y,r_j)^{1/2}
 \norm{
 \left(1+\frac{\norm{\cdot-y}}{r_j}\right)^s
 \bigl(K_{j,\varepsilon}^{\rho\tau}(\cdot,y)
       -K_{j,\varepsilon}^{\rho\tau}(\cdot,y')\bigr)
 }_2
\le
 C_s\norm{Y-Y'}\norm{b_j}_{H^\sigma}
 =
 C_s\frac{\norm{y-y'}}{r_j}\norm{b_j}_{H^\sigma}.
\end{align*}
These bounds are uniform in $j$, $\varepsilon$, $\rho$, and $\tau$.

Each $m_j$ is bounded and compactly supported, hence belongs to $L^1(d\omega)$.  Since
$\abs{E_\kappa(iu,v)}\le1$ for real $u,v$, Fubini's theorem gives
$$
 \ip{T_{m_j}^GF}{H}_{L^2(\mathcal C;\ell^2(G))}
 =
 \int_\mathcal C\int_\mathcal C
 \ip{K_j(x,y)F(y)}{H(x)}_{\ell^2(G)}
 \,d\omega(y)\,d\omega(x)
$$
for $F,H\in L_c^\infty(\mathcal C;\ell^2(G))$.  For finite $E\subset\bbZ$, set
$$
 \Psi_E(\xi)=\sum_{j\in E}\psi(2^{-j}\norm{\xi}),
 \qquad
 m_E=\Psi_Em=\sum_{j\in E}m_j.
$$
For $\xi\ne0$,
\begin{align*}
 0\le \Psi_E(\xi)\le\sum_{j\in\bbZ}\psi(2^{-j}\norm{\xi})=1.
\end{align*}
Consequently,
$
 \abs{m_E(\xi)}\le\abs{m(\xi)},
$
and hence
$
 \norm{m_E}_\infty\le\norm{m}_\infty.
$
Finite summation and Fubini give
\begin{align*}
 K_{m_E}^{\rho\tau}(x,y)
 &=
 \sum_{j\in E}K_{m_j}^{\rho\tau}(x,y).
\end{align*}
Plancherel gives
\begin{align*}
 \norm{T_{m_E}^G}_{2\to2}
 =
 \norm{T_{m_E}}_{2\to2}
 \le
 \norm{m_E}_\infty
 \le
 \norm{m}_\infty.
\end{align*}
Here $m_E\in L^1(d\omega)$ because it is bounded and compactly supported.
Thus \eqref{eq:bj-bound} and Proposition~\ref{prop:finite-frequency-closure}, with
$A=C\calS_\sigma(m)$
and
$M_2=\norm{m}_\infty$
give, uniformly in finite $E$,
\begin{align}\label{eq:finite-endpoints}
 &\norm{T_{m_E}^G}_{L^1(X_G)\to L^{1,\infty}(X_G)}
 +\norm{T_{m_E}^G}_{H^1_{\mathrm{at}}(\mathcal C;\ell^1(G))\to L^1(\mathcal C;\ell^1(G))}
 +\norm{T_{m_E}^G}_{L^\infty(\mathcal C;\ell^\infty(G))\to\BMO(\mathcal C;\ell^\infty(G))}
 \le
 C\calS_\sigma(m).
\end{align}
Here we also used \eqref{eq:automatic-Linfty}.  If $F=Uf$, then
$$
 \omega\{x\in\bbR^N:\abs{T_{m_E}f(x)}>\lambda\}
 =
 \sum_{\rho=1}^{\abs{G}}
 \omega\{x\in \mathcal C:\abs{(T_{m_E}^GF)_\rho(x)}>\lambda\},
$$
and the finite-frequency conclusions transfer through $U$.

For $E_J=\{-J,\ldots,J\}$, set
$$
 \Psi_J(\xi)=\sum_{\abs{j}\le J}\psi(2^{-j}\norm{\xi}),
 \qquad
 m^{(J)}=\Psi_Jm.
$$
For $\xi\ne0$,
$$
 0\le \Psi_J(\xi)\le1,
 \qquad
 \Psi_J(\xi)\longrightarrow1,
 \qquad
 \abs{m^{(J)}(\xi)-m(\xi)}\le2\norm{m}_\infty.
$$
Therefore, for $f\in L^2(d\omega)$,
\begin{align*}
 &\abs{m^{(J)}(\xi)-m(\xi)}^2\abs{\calF f(\xi)}^2
 \le
 4\norm{m}_\infty^2\abs{\calF f(\xi)}^2
 \in L^1(d\omega).
\end{align*}
Hence dominated convergence gives
\begin{align}\label{eq:truncation-L2}
 \norm{T_{m^{(J)}}f-T_mf}_2^2
 &=
 \int_{\bbR^N}
 \abs{m^{(J)}(\xi)-m(\xi)}^2\abs{\calF f(\xi)}^2\,d\omega(\xi)
 \longrightarrow0.
\end{align}
The same convergence holds with $m$ replaced by $\overline m$.

Let $f\in L^1\cap L^2$.  Along a subsequence $J_k$, the preceding convergence is almost everywhere.
Hence, for $0<\theta<1$,
$$
 \{\abs{T_mf}>\lambda\}
 \subset
 \liminf_{k\to\infty}
 \{\abs{T_{m^{(J_k)}}f}>\theta\lambda\}.
$$
Consequently,
\begin{align*}
 \omega\{\abs{T_mf}>\lambda\}
 &\le
 \liminf_{k\to\infty}
 \omega\{\abs{T_{m^{(J_k)}}f}>\theta\lambda\} \le
 \frac{C\calS_\sigma(m)}{\theta\lambda}\norm{f}_1.
\end{align*}
Letting $\theta\uparrow1$ gives the weak bound on $L^1\cap L^2$.  Since this space is dense in
$L^1$ and $L^{1,\infty}$ is complete under an equivalent quasi-norm, $T_m$ extends uniquely to $L^1$,
with
$$
 \norm{T_mf}_{L^{1,\infty}}
 \le
 C\calS_\sigma(m)\norm{f}_1,
 \qquad f\in L^1(d\omega).
$$
Interpolation with the $L^2$ multiplier bound gives the strong $L^p$ estimate for $1<p\le2$.
Since
$$
 T_m^*=T_{\overline m},
 \qquad
 \calS_\sigma(\overline m)=\calS_\sigma(m),
$$
duality gives the estimate for $2\le p<\infty$.

The weak-$L^1$ and $L^p$ realizations agree on $L^1\cap L^p$.  If $f\in L^1\cap L^p$, choose $f_n\in L_c^\infty$ with
$f_n\to f$ in both spaces.  The weak-$L^1$ and $L^p$ realizations agree with the $L^2$ multiplier on
$f_n$ and their common images converge in measure to both extensions. Hence the two limits agree.

Write
$$
 T_m^{G,(2)}=UT_m U^{-1},
 \qquad
  T_{\overline m}^{G,(2)}=UT_{\overline m} U^{-1}
$$
for the $L^2$ realization.  Let $F$ be a finite componentwise atomic sum.  Then
$F\in L^2(\mathcal C;\ell^2(G))$ and
$$
 T_{m^{(J)}}^GF
 \longrightarrow
 T_m^{G,(2)}F
 \quad\hbox{in }L^2(\mathcal C;\ell^2(G)).
$$
After passing to a subsequence, Fatou's lemma gives
\begin{align*}
 \norm{T_m^{G,(2)}F}_{L^1(\mathcal C;\ell^1(G))}
 &\le
 \liminf_{k\to\infty}
 \norm{T_{m^{(J_k)}}^GF}_{L^1(\mathcal C;\ell^1(G))}
 \le
 C\calS_\sigma(m)
 \norm{F}_{H^1_{\mathrm{at}}(\mathcal C;\ell^1(G))}.
\end{align*}
Density gives
$$
 T_m^G:H^1_{\mathrm{at}}(\mathcal C;\ell^1(G))\longrightarrow L^1(\mathcal C;\ell^1(G)),
$$
and hence $T_m:H^1_{\ch}\to L^1(d\omega)$.

Denote the atomic and weak-$L^1$ extensions by $T_m^{G,H^1}$ and $T_m^{G,w}$, respectively.  For
finite componentwise atomic sums $F_n\to F$ in $H^1_{\mathrm{at}}(\mathcal C;\ell^1(G))$,
\begin{align*}
 T_m^{G,H^1}F_n
 \longrightarrow T_m^{G,H^1}F
 \qquad \text{in }L^1(\mathcal C;\ell^1(G))\hookrightarrow L^{1,\infty}(X_G).
\end{align*}
Moreover,
\begin{align*}
 T_m^{G,w}F_n
 \longrightarrow T_m^{G,w}F
 \qquad \text{in }L^{1,\infty}(X_G).
\end{align*}
Both realizations agree with $T_m^{G,(2)}$ on every $F_n$, so their limits coincide.

Since $\calS_\sigma(\overline m)=\calS_\sigma(m)$, the same argument gives
$$
 \norm{T_{\overline m}^GH}_{L^1(\mathcal C;\ell^1(G))}
 \le
 C\calS_\sigma(m)
 \norm{H}_{H^1_{\mathrm{at}}(\mathcal C;\ell^1(G))}.
$$
This realization agrees with $T_{\overline m}^{G,(2)}$ on finite componentwise atomic sums.
For $F\in L^\infty(\mathcal C;\ell^\infty(G))$ and a finite componentwise atomic sum $H$, define
$$
 \Lambda_F(H)=\ip{T_{\overline m}^GH}{F}_{X_G}.
$$
Then
\begin{align*}
 \abs{\Lambda_F(H)}
 &\le
 \norm{T_{\overline m}^GH}_{L^1(\mathcal C;\ell^1(G))}
 \norm{F}_{L^\infty(\mathcal C;\ell^\infty(G))} \le
 C\calS_\sigma(m)
 \norm{H}_{H^1_{\mathrm{at}}(\mathcal C;\ell^1(G))}
 \norm{F}_{L^\infty(\mathcal C;\ell^\infty(G))}.
\end{align*}
Componentwise Coifman--Weiss duality yields $B_F\in\BMO(\mathcal C;\ell^\infty(G))$ such that
$$
 \ip{H}{B_F}_{X_G}
 =\Lambda_F(H).
$$
Moreover,
\begin{align*}
 \norm{B_F}_{\BMO(\mathcal C;\ell^\infty(G))}
 \le
 C\calS_\sigma(m)\norm{F}_{L^\infty(\mathcal C;\ell^\infty(G))}.
\end{align*}
Define $T_m^GF=B_F$.  Applying $U^{-1}$ gives
$$
 T_m:L^\infty(d\omega)\longrightarrow\BMO_{\ch}.
$$
If $F\in L^\infty(\mathcal C;\ell^\infty(G))\cap L^2(\mathcal C;\ell^2(G))$, then, for every finite
componentwise atomic sum $H$,
\begin{align*}
 \ip{H}{B_F}_{X_G}
 &=
 \ip{T_{\overline m}^GH}{F}_{X_G}
 =
 \ip{T_{\overline m}^{G,(2)}H}{F}_{X_G}
 =
 \ip{H}{T_m^{G,(2)}F}_{X_G}.
\end{align*}
Thus $B_F$ is the BMO class of $T_m^{G,(2)}F$.
\end{proof}

\subsection{Examples of multipliers}

We give two applications of Theorem~\ref{thm:main}.
The first concerns symbols that depend only on the direction
of the frequency:
$$
 m_a(\xi)=a\left(\frac{\xi}{\norm{\xi}}\right),
 \qquad \xi\ne0.
$$
For example, $a(\theta)=\theta_k$ gives
$m_a(\xi)=\xi_k/\norm{\xi}$.
Since $m_a(t\xi)=m_a(\xi)$ for $t>0$, all its normalised
dyadic pieces coincide. The required dyadic Sobolev bound
then follows from the Sobolev regularity of $a$ on
$\mathbb S^{N-1}$. The function $a$ need not be $G$-invariant
or vanish near the sets
$\mathbb S^{N-1}\cap\alpha^\perp$, $\alpha\in R$.

The second application gives sufficient conditions in terms
of ordinary partial derivatives of $m$. We also consider
symbols of the form
$m(\xi)=a(\xi/\norm{\xi})F(\norm{\xi})$.
Taking $a\equiv1$ and $F(r)=r^{iu}$ gives the imaginary powers
$(-\Delta_\kappa)^{iu/2}$, $u\in\bbR$.

\begin{corollary}\label{cor:angular}
Let $\sigma>\mathbf N/2$ and $a\in H^\sigma(\mathbb S^{N-1})$.  Set
$$
 m_a(\xi)=a\left(\frac{\xi}{\norm{\xi}}\right),
 \quad \xi\ne0,
 \qquad
 m_a(0)=0.
$$
Then all conclusions of Theorem~\ref{thm:main} hold for $T_{m_a}$, and
\begin{align*}
 \norm{T_{m_a}}_{L^p\to L^p}
 &\le
 C_p\norm{a}_{H^\sigma(\mathbb S^{N-1})},
 \qquad 1<p<\infty,
 \end{align*}
 and
 \begin{align*}
 \norm{T_{m_a}}_{L^1\to L^{1,\infty}}
 +\norm{T_{m_a}}_{H^1_{\ch}\to L^1}
 +\norm{T_{m_a}}_{L^\infty\to\BMO_{\ch}}
 &\le
 C\norm{a}_{H^\sigma(\mathbb S^{N-1})}.
\end{align*}
No $G$-invariance or wall-separation assumption is imposed on $a$.
\end{corollary}

\begin{proof}
The case $N=1$ is immediate.  Suppose $N\ge2$ and define
$$
 \mathcal E a(\eta)
 =
 \psi(\norm{\eta})a\left(\frac{\eta}{\norm{\eta}}\right),
 \qquad \eta\ne0.
$$
For $a\in C^\infty(\mathbb S^{N-1})$, $q\in\mathbb N_0$, $\abs{\beta}\le q$, and
$r\in\supp\psi$,
$$
 \abs{\partial^\beta(\mathcal Ea)(r\theta)}
 \le
 C_{q,\psi}
 \sum_{k=0}^{\abs{\beta}}
 \abs{\nabla_{\mathbb S}^{k}a(\theta)}.
$$
Hence, by polar coordinates,
\begin{align*}
 \norm{\mathcal Ea}_{H^q(\bbR^N)}^2
 &\le
 C_{q,\psi}
 \sum_{\abs{\beta}\le q}
 \int_{\supp\psi}\int_{\mathbb S^{N-1}}
 \bigg(
 \sum_{k=0}^{\abs{\beta}}
 \abs{\nabla_{\mathbb S}^{k}a(\theta)}
 \bigg)^2
 r^{N-1}\,d\theta\,dr
 \le
 C_{q,\psi}
 \sum_{k=0}^q
 \norm{\nabla_{\mathbb S}^{k}a}_{L^2(\mathbb S^{N-1})}^2
 \le
 C_{q,\psi}\norm{a}_{H^q(\mathbb S^{N-1})}^2.
\end{align*}
By density and interpolation,
$$
 \norm{\mathcal Ea}_{H^\sigma(\bbR^N)}
 \le
 C_{\sigma,\psi}\norm{a}_{H^\sigma(\mathbb S^{N-1})}.
$$
Moreover, for every $j\in\bbZ$ and $\eta\ne0$,
$$
 \psi(\norm{\eta})m_a(2^j\eta)
 =
 \psi(\norm{\eta})a\left(\frac{\eta}{\norm{\eta}}\right)
 =
 \mathcal Ea(\eta),
$$
and therefore
$$
 \calS_\sigma(m_a)
 =
 \norm{\mathcal Ea}_{H^\sigma(\bbR^N)}
 \le
 C_{\sigma,\psi}\norm{a}_{H^\sigma(\mathbb S^{N-1})}.
$$
Theorem~\ref{thm:main} completes the proof.
\end{proof}

For $L\in\mathbb N_0$, define
$$
 M_L(m)
 =
 \max_{\abs{\beta}\le L}
 \sup_{\xi\ne0}
 \norm{\xi}^{\abs{\beta}}\abs{\partial^\beta m(\xi)}.
$$

\begin{corollary}\label{cor:mikhlin}
Let $L\in\mathbb N_0$ satisfy $L>\mathbf N/2$.  Suppose that
$m\in C^L(\bbR^N\setminus\{0\})$ and $M_L(m)<\infty$, and extend $m$ arbitrarily to the origin.
Then all conclusions of Theorem~\ref{thm:main} hold, with
\begin{align*}
 \norm{T_m}_{L^p\to L^p}
 &\le
 C_pM_L(m),
 \qquad 1<p<\infty,
 \end{align*}
 and
$ \norm{T_m}_{L^1\to L^{1,\infty}}
 +\norm{T_m}_{H^1_{\ch}\to L^1}
 +\norm{T_m}_{L^\infty\to\BMO_{\ch}}
 \le
 CM_L(m).
$
Here $C$ and $C_p$ depend only on $R$, $\kappa$, $L$, $\psi$, and, for $C_p$, $p$.

For angular--radial symbols, let $a\in C^L(\mathbb S^{N-1})$ and suppose that
$F\in C^L(0,\infty)$ satisfies
$$
 R_L(F)
 =
 \max_{0\le q\le L}\sup_{r>0}r^q\abs{F^{(q)}(r)}
 <\infty.
$$
For the symbol
$$
 m(\xi)=a\left(\frac{\xi}{\norm{\xi}}\right)F(\norm{\xi}),
 \qquad
 \xi\ne0,
$$
one has
$$
 M_L(m)
 \le
 C_{N,L}\norm{a}_{C^L(\mathbb S^{N-1})}R_L(F),
$$
and hence the same operator bounds hold with $M_L(m)$ replaced by
$C_{N,L}\norm{a}_{C^L(\mathbb S^{N-1})}R_L(F)$.  In particular, if $u\in\bbR$ and $F(r)=r^{iu}$,
they hold with $M_L(m)$ replaced by
$$
 C_L(1+\abs{u})^L\norm{a}_{C^L(\mathbb S^{N-1})}.
$$
When $a\equiv1$, $T_m=(-\Delta_\kappa)^{iu/2}$.
\end{corollary}

\begin{proof}
Let
$
 \mathcal A_\psi=\supp\bigl(\psi(\norm{\cdot})\bigr)
 \subset
 \bbR^N\setminus\{0\}.
$
For $\abs{\gamma}\le L$,
\begin{align*}
   \partial^\gamma\bigl[\psi(\norm{\eta})m(2^j\eta)\bigr]
 &=
 \sum_{\beta\le\gamma}
 \binom{\gamma}{\beta}
 \partial^{\gamma-\beta}\bigl[\psi(\norm{\cdot})\bigr](\eta)
 2^{j\abs{\beta}}(\partial^\beta m)(2^j\eta),
 \end{align*}
 and
 \begin{align*}
 \sup_{j\in\bbZ}
 \norm{\psi(\norm{\cdot})m(2^j\cdot)}_{H^L(\bbR^N)}
 &\le
 C_{L,\psi}
 \max_{\abs{\gamma}\le L}
 \sum_{\beta\le\gamma}
 \sup_{\substack{j\in\bbZ\\\eta\in \mathcal A_\psi}}
 2^{j\abs{\beta}}\abs{(\partial^\beta m)(2^j\eta)}\\
 &\le
 C_{L,\psi}M_L(m)
 \max_{\abs{\gamma}\le L}
 \sum_{\beta\le\gamma}
 \sup_{\eta\in \mathcal A_\psi}\norm{\eta}^{-\abs{\beta}}
 \le
 C_{L,\psi}M_L(m).
\end{align*}
Thus $\calS_L(m)\le C_{L,\psi}M_L(m)$, and Theorem~\ref{thm:main} applies with $\sigma=L$.

For $\xi=r\theta$, $r>0$, and $\abs{\beta}\le L$, the polar-coordinate chain rule gives
\begin{align*}
 r^{\abs{\beta}}\abs{(\partial^\beta m)(r\theta)}
 &\le
 C_{N,L}
 \sum_{q=0}^{\abs{\beta}}
 \sum_{k=0}^{\abs{\beta}-q}
 \abs{\nabla_{\mathbb S}^{k}a(\theta)}
 r^q\abs{F^{(q)}(r)}
 \le
 C_{N,L}\norm{a}_{C^L(\mathbb S^{N-1})}R_L(F).
\end{align*}
Taking the supremum over $r$, $\theta$, and $\beta$ proves
$$
 M_L(m)
 \le
 C_{N,L}\norm{a}_{C^L(\mathbb S^{N-1})}R_L(F).
$$
Finally, for $0\le q\le L$,
$$
 r^q\abs{\frac{d^q}{dr^q}r^{iu}}
 =
 \prod_{\ell=0}^{q-1}\abs{iu-\ell}
 \le
 C_L(1+\abs{u})^L,
$$
which gives the last assertion.
\end{proof}

\subsection{A comparison based only on spectral averages}

The estimates above concern derivatives in the frequency variable.
They do not show how the kernel decays in the spatial variable.  The heat
kernel yields a spectral $L^2$ average which, when used without Proposition~\ref{thm:joint}, costs an
additional $N/2$ Euclidean derivatives. 

\begin{proposition}\label{prop:raw-average}
Let $\mathcal A\subset\bbR^N\setminus\{0\}$ be compact.  For every multi-index $\nu$, there is a constant
$C_{\mathcal A,\nu}$ such that
\begin{equation}\label{eq:raw-average}
 \sup_{y\in\bbR^N}
 \Omega_\kappa(y)
 \int_\mathcal A
 \abs{\partial_y^\nu E_\kappa(i\xi,y)}^2\,d\omega(\xi)
 \le
 C_{\mathcal A,\nu}.
\end{equation}
\end{proposition}

\begin{proof}
The heat-kernel representation is
\begin{align*}
 h_t(x,z)
 &=
 (c_\kappa^{\mathrm M})^{-1}
 \int_{\bbR^N}
 e^{-t\norm{\xi}^2}
 E_\kappa(i\xi,x)E_\kappa(-i\xi,z)\,d\omega(\xi),
\end{align*}
By \cite[Proposition~2.1(4)]{Langen2026}, the differentiated integrand satisfies
\begin{align*}
 e^{-\norm{\xi}^2}
 \abs{\partial_x^\nu E_\kappa(i\xi,x)
       \partial_z^\nu E_\kappa(-i\xi,z)}
 &\le
 C_\nu e^{-\norm{\xi}^2}(1+\norm{\xi})^{2\abs{\nu}}
 \in L^1(\bbR^N,d\omega),
\end{align*}
uniformly in $x,z\in\bbR^N$.  Hence differentiation under the integral is valid, and using
$\overline{\partial_y^\nu E_\kappa(i\xi,y)}=\partial_y^\nu E_\kappa(-i\xi,y)$ yields
\begin{align*}
 \int_{\bbR^N}
 e^{-\norm{\xi}^2}
 \abs{\partial_y^\nu E_\kappa(i\xi,y)}^2\,d\omega(\xi)
 =
 c_\kappa^{\mathrm M}
 \partial_x^\nu\partial_z^\nu h_1(x,z)\big|_{x=z=y}
 \le
 c_\kappa^{\mathrm M}
 \abs{\partial_x^\nu\partial_z^\nu h_1(x,z)\big|_{x=z=y}}
 \le
 \frac{C_\nu}{\omega(B(y,1))}.
\end{align*}
Here the last inequality is
\cite[Theorem~4.1(d)]{ADH2019} with $t=1$, $x=z=y$, equal spatial multi-indices, and no time
derivatives.  If $M_\mathcal A=\max_{\xi\in \mathcal A}\norm{\xi}$, then 
\begin{align*}
 \Omega_\kappa(y)
 \int_\mathcal A\abs{\partial_y^\nu E_\kappa(i\xi,y)}^2\,d\omega(\xi)
 &\le
 e^{M_\mathcal A^2}\Omega_\kappa(y)
 \int_{\bbR^N}
 e^{-\norm{\xi}^2}
 \abs{\partial_y^\nu E_\kappa(i\xi,y)}^2\,d\omega(\xi)
  \le
 C_{\mathcal A,\nu}\frac{\Omega_\kappa(y)}{\omega(B(y,1))}
 \le
 C_{\mathcal A,\nu},
\end{align*}
where the last inequality follows from the fact that
$ \omega(B(x,1))\approx\Omega_\kappa(x).
$ This is the standard volume estimate with radius $1$, see, for example,
\cite[Section~2]{DH2019}.  Unless otherwise stated, constants may depend on $R$, $\kappa$, and fixed
compact annuli.

Taking the supremum over $y$ in the above inequality proves \eqref{eq:raw-average}.
\end{proof}

For $u\in L^\infty(\mathcal A)$, Proposition~\ref{prop:raw-average} gives
\begin{align}\label{eq:raw-u}
 \Omega_\kappa(y)^{1/2}
 \norm{u(\xi)\partial_y^\nu E_\kappa(i\xi,y)}_{L^2(A,d\omega)}
 \le\,
 \norm{u}_{L^\infty(\mathcal A)}
 \left(
 \Omega_\kappa(y)
 \int_\mathcal A
 \abs{\partial_y^\nu E_\kappa(i\xi,y)}^2\,d\omega(\xi)
 \right)^{1/2}
 \le
 C_{\mathcal A,\nu}\norm{u}_{L^\infty(\mathcal A)}.
\end{align}
If $u=\mathcal D a$, where $\mathcal D$ is a operator composition of total order
$L$, then Sobolev embedding and the preceding composition estimate give
\begin{align*}
 \norm{u}_{L^\infty(\mathcal A)}
 &\le
 \norm{\mathcal D a}_\infty
 \le
 C_\varepsilon\norm{\mathcal D a}_{H^{N/2+\varepsilon}}
 \le
 C_{\mathcal D,\varepsilon}
 \norm{a}_{H^{L+N/2+\varepsilon}}.
\end{align*}
Thus this argument costs an additional $N/2+\varepsilon$ derivatives.

Already for $\nu=0$, the right side of \eqref{eq:packet-joint} cannot be replaced by
$\norm{u}_{L^2(d\omega)}$. Consider the reflection group on $\bbR^2$ generated by
the reflections across $x_1=0$ and $x_2=0$, with respective
multiplicities $\kappa_1,\kappa_2>0$. Fix $0<c<1/4$ sufficiently
small and, for $\ell\ge1$, set
$$
 y_\ell=(0,\ell),
 \qquad
 Q_\ell=\left\{\xi\in\bbR^2:
 \abs{\xi_1-1}<c,\quad \frac{c}{2\ell}<\xi_2<\frac{c}{\ell}\right\}.
$$
The sets $Q_\ell$ consist of regular points and lie in one fixed compact annulus.  Writing
$e_{\kappa_j}$ for the rank-one Dunkl kernels and using $e_{\kappa_j}(0)=1$, the product formula gives,
for $\xi\in Q_\ell$,
\begin{align*}
 E_\kappa(i\xi,y_\ell)
 &=
 e_{\kappa_1}(0)e_{\kappa_2}(i\ell\xi_2)
 =
 e_{\kappa_2}(i\ell\xi_2).
\end{align*}
By continuity at the origin,
$
 \abs{E_\kappa(i\xi,y_\ell)}
 \ge c_0,
$
whereas
$
 \Omega_\kappa(y_\ell)^{1/2}
 \approx \ell^{\kappa_2}
$
and, on $Q_\ell$,
$
 w_\kappa(\xi)^{1/2}
 \approx \ell^{-\kappa_2}.
$

Hence, for any $u_L\in C_c^\infty(Q_\ell)\setminus\{0\}$,
\begin{align*}
 \frac{
 \Omega_\kappa(y_\ell)^{1/2}
 \norm{u_\ell(\xi)E_\kappa(i\xi,y_\ell)}_{L^2(d\omega)}
 }{\norm{u_\ell}_{L^2(d\omega)}}
 &\ge
 c_0\Omega_\kappa(y_\ell)^{1/2}
 \gtrsim
 \ell^{\kappa_2}
 \longrightarrow\infty.
\end{align*}
Moreover, for $0\le\theta<1/2$,
$$
 \inf_{\xi\in Q_\ell}
 w_\kappa(\xi)^\theta
 \Omega_\kappa(y_\ell)^{1/2}
 \abs{E_\kappa(i\xi,y_\ell)}
 \gtrsim
 \ell^{(1-2\theta)\kappa_2}
 \longrightarrow\infty.
$$
Thus the power $1/2$ of $w_\kappa$ in Proposition~\ref{thm:joint} is sharp in this model.

\begin{corollary}
Let $m$ be measurable and suppose that $\calS_\sigma(m)<\infty$ for some
$
 \sigma>\frac{\mathbf N+N}{2}.
$
Then all conclusions of Theorem~\ref{thm:main} hold.
\end{corollary}

\begin{proof}
Set
$
 \mathscr D_0=\{I\}
 $
 and 
 $
 \mathscr D_{m+1}
 =
 \mathscr D_m\cup
 \bigl\{\partial_j\mathcal D,\delta_\alpha\mathcal D:
 \mathcal D\in\mathscr D_m,\ 1\le j\le N,\ \alpha\in R_+\bigr\}.
$

Choose $s_0$ and $\eta$ so that
$$
 \frac{\mathbf N}{2}<s_0<\sigma-\frac N2,
 \qquad
 0<\eta<\sigma-s_0-\frac N2,
 \qquad
 t_0=s_0+\frac N2+\eta<\sigma,
 \qquad
 M=\lfloor s_0\rfloor,
 \qquad
 \theta=s_0-M.
$$
For $q\in\{M,M+1\}$, if $\mathcal D\in\mathscr D_q$ has order $L(\mathcal D)\le q$, then
Sobolev embedding and Proposition~\ref{prop:compositions} give
\begin{align*}
 \norm{\mathcal D a}_\infty
 &\le
 C_\eta\norm{\mathcal D a}_{H^{N/2+\eta}}
 \le
 C_{q,\eta}\norm{a}_{H^{L(\mathcal D)+N/2+\eta}}
 \le
 C_{q,\eta}\norm{a}_{H^{q+N/2+\eta}}.
\end{align*}
For $u$ supported in $\mathcal A_2$, Plancherel and \eqref{eq:raw-u} give
\begin{align*}
 \Omega_\kappa(y)^{1/2}
 \norm{\calK_h[u](\cdot,y)}_2
 &=
 (c_\kappa^{\mathrm M})^{-1}
 \Omega_\kappa(y)^{1/2}
 \norm{u(\xi)E_\kappa(-ihy,\xi)}_{L^2(d\omega)}
 \le
 C\norm{u}_\infty,
 \end{align*}
 and
 \begin{align*}
 \Omega_\kappa(y)^{1/2}
 \norm{\calK_{h,v}^{(1)}[u](\cdot,y)}_2
 &\le
 C\norm{v}\norm{u}_\infty.
\end{align*}
For smooth $a$ supported in $\mathcal A_1$, \eqref{eq:packet-composition-expansion} and
\eqref{eq:jet-composition-expansion} therefore give, with the direct estimate used when $q=0$,
\begin{align*}
 \Omega_\kappa(y)^{1/2}
 \norm{(1+d(\cdot,y))^q\calK_g[a](\cdot,y)}_{L^2(d\omega)}
 &\le
 C_q
 \sum_{\substack{h\in G\\\mathcal D\in\mathscr D_q}}
 \Omega_\kappa(y)^{1/2}
 \norm{\calK_h[\mathcal D a](\cdot,y)}_2\\
 &\le
 C_q\sum_{\mathcal D\in\mathscr D_q}\norm{\mathcal D a}_\infty
 \le
 C_{q,\eta}\norm{a}_{H^{q+N/2+\eta}}.
\end{align*}
Similarly,
\begin{align*}
 &\Omega_\kappa(y)^{1/2}
 \norm{(1+d(\cdot,y))^q\calK_{g,v}^{(1)}[a](\cdot,y)}_{L^2(d\omega)}\\
 \le\,&
 C_q
 \sum_{\substack{h\in G\\\mathcal D\in\mathscr D_q}}
 \Omega_\kappa(y)^{1/2}
 \norm{\calK_{h,v}^{(1)}[\mathcal D a](\cdot,y)}_2
 +
 C_q\norm{v}
 \sum_{\substack{h\in G\\\mathcal D\in\mathscr D_q}}
 \Omega_\kappa(y)^{1/2}
 \norm{\calK_h[\mathcal D a](\cdot,y)}_2\\
 \le\,&
 C_q\norm{v}
 \sum_{\mathcal D\in\mathscr D_q}\norm{\mathcal D a}_\infty
 \le
 C_{q,\eta}\norm{v}\norm{a}_{H^{q+N/2+\eta}}.
\end{align*}
Here we apply \eqref{eq:raw-u} to each kernel $\calK_h[\mathcal D a]$ with $\abs{\nu}=0$ and each kernel $\calK_{h, v}^{(1)}[\mathcal D a]$ with $\abs{\nu}=1$.

The same estimates hold for amplitudes supported in $\mathcal A_2$, with a different fixed constant.
Choose $\chi\in C_c^\infty(\operatorname{int}\mathcal A_2)$ equal to one near $\mathcal A_1$.  For $\theta\in(0,1)$,
interpolate the operators
$$
 a\longmapsto
 \Omega_\kappa(y)^{1/2}\calK_g[\chi a](\cdot,y),
 \qquad
 a\longmapsto
 \frac{\Omega_\kappa(y)^{1/2}}{\norm{v}}
 \calK_{g,v}^{(1)}[\chi a](\cdot,y)
$$
between the two integer bounds.  The interpolation identities are
\begin{align*}
 [H^{M+N/2+\eta},H^{M+1+N/2+\eta}]_\theta
 &=H^{t_0}.
\end{align*}
and
\begin{align*}
 [L^2((1+d(\cdot,y))^{2M}d\omega),
   L^2((1+d(\cdot,y))^{2M+2}d\omega)]_\theta
 &=L^2((1+d(\cdot,y))^{2s_0}d\omega),
\end{align*}
respectively.  Hence
\begin{align*}
 \Omega_\kappa(y)^{1/2}
 \norm{(1+d(\cdot,y))^{s_0}\calK_g[a](\cdot,y)}_2
 &\le
 C\norm{a}_{H^{t_0}}.
\end{align*}
Similarly,
\begin{align*}
 \Omega_\kappa(y)^{1/2}
 \norm{(1+d(\cdot,y))^{s_0}\calK_{g,v}^{(1)}[a](\cdot,y)}_2
 &\le
 C\norm{v}\norm{a}_{H^{t_0}}.
\end{align*}
The case $\theta=0$ follows from the $q=M$ estimate, and the case $v=0$ is immediate.  We follow the proof of Lemma~\ref{lem:sobolev-kernel-closure},
using the two estimates above and \eqref{eq:packet-transpose}. Replacing $(s,H^s)$ there by $(s_0,H^{t_0})$ gives the corresponding size and
first-difference bounds in both kernel variables for $H^{t_0}$ functinos supported in $\mathcal A_0$.

Retain the notation $r_j$, $b_j$, $b_{j,\varepsilon}$, $K_{j,\varepsilon}$, $\Psi_J$, and $m^{(J)}$ from
the proof of Theorem~\ref{thm:main}.  On $\mathcal C$, \eqref{eq:chamber-distance} and the dilation identity
\eqref{eq:lifted-scaling} give, uniformly
in $j$, $\varepsilon\in\{0,1\}$, $\rho$, and $\tau$,
\begin{align*}
 V_\mathcal C(y,r_j)^{1/2}
 \norm{
 \left(1+\frac{\norm{\cdot-y}}{r_j}\right)^{s_0}
 K_{j,\varepsilon}^{\rho\tau}(\cdot,y)
 }_2
 &\le
 C\norm{b_j}_{H^{t_0}}.
\end{align*}
If $\norm{y-y'}\le r_j$, then
\begin{align*}
 V_\mathcal C(y,r_j)^{1/2}
 \norm{
 \left(1+\frac{\norm{\cdot-y}}{r_j}\right)^{s_0}
 \bigl(K_{j,\varepsilon}^{\rho\tau}(\cdot,y)
       -K_{j,\varepsilon}^{\rho\tau}(\cdot,y')\bigr)
 }_2
 &\le
 C\frac{\norm{y-y'}}{r_j}\norm{b_j}_{H^{t_0}}.
\end{align*}
Since $s_0>\mathbf N/2$ and $t_0<\sigma$,
\begin{align*}
 A
 \le
 C\sup_{j\in\bbZ}\norm{b_j}_{H^{t_0}}
 \le
 C\calS_\sigma(m).
\end{align*}
Moreover,
\begin{align*}
 M_2
 \le
 \norm{m}_{L^\infty}
 \le
 C\calS_\sigma(m).
\end{align*}
Proposition~\ref{prop:finite-frequency-closure} therefore gives \eqref{eq:finite-endpoints} uniformly
for the finite dyadic sums.  Moreover,
$$
 0\le \Psi_J\le1,
 \qquad
 \Psi_J\longrightarrow1,
 \qquad
 \abs{\Psi_Jm-m}\le2\norm{m}_\infty,
$$
so \eqref{eq:truncation-L2} remains valid.  The rest of the proof of Theorem~\ref{thm:main} gives the
weak, strong, Hardy, and BMO bounds.
\end{proof}

\begin{remark}
Assign $m(0)=0$ and let $\widetilde\psi$ be the cutoff in \cite[Theorem~1.3]{Bui2025}.  Finite dyadic
overlap, bounded dilations by factors comparable to $1$, and
$H^\sigma\hookrightarrow W_\infty^{s_0}$ give
$$
 \sup_{t>0}
 \norm{\widetilde\psi(\cdot)m(t\cdot)}_{W_\infty^{s_0}}
 \le
 C\sup_{t>0}
 \norm{\widetilde\psi(\cdot)m(t\cdot)}_{H^\sigma}
 \le
 C\calS_\sigma(m).
$$
Thus Bui's theorem also gives the strong $L^p$ bounds.  The proof above additionally gives weak type
$(1,1)$ and the chamber-pullback endpoints.  Its extra $N/2$ loss is precisely the Euclidean embedding
$$
 \norm{\mathcal D a}_\infty
 \le
 C_\eta\norm{\mathcal D a}_{H^{N/2+\eta}}.
$$
No wall separation is used.
\end{remark}

\medskip
\medskip
\medskip
\noindent\textbf{Acknowledgements:}\par
Li is supported by the Australian Research Council through grant DP260100485.  Wu is supported by the
National Natural Science Foundation of China
(grant no.~12201002).

\end{document}